\documentclass[a4paper,10pt]{article}
\usepackage{a4wide}
\usepackage{amsmath,amssymb,amsfonts,amsthm}
\usepackage[all,cmtip]{xy}
\usepackage{fouridx}
\usepackage[usenames,dvipsnames]{color}
\usepackage[colorlinks=true,linktocpage=true,linkcolor=blue,citecolor=blue]{hyperref}

\theoremstyle{definition}
\newtheorem{definition}{Definition}[section]
\newtheorem{assumption}[definition]{Assumption}
\theoremstyle{plain}
\newtheorem{theorem}[definition]{Theorem}
\newtheorem{proposition}[definition]{Proposition}
\newtheorem{lemma}[definition]{Lemma}
\newtheorem{corollary}[definition]{Corollary}
\newtheorem{example}[definition]{Example}
\theoremstyle{remark}
\newtheorem{remark}[definition]{Remark}

\numberwithin{equation}{section}

\newcommand{\tildeHD}[2]{\langle #1,#2\rangle_{\tilde{\mathsf{d}}}^{-}}

\allowdisplaybreaks[4]

\begin{document}
\title{Central extensions of generalized orthosymplectic Lie superalgebras}
\author{Zhihua Chang\footnote{Zhihua Chang is supported by the National Natural Science Foundation of China (no. 11501213), the China Postdoctoral Science Foundation (no. 2015M570705) and the Fundamental Research Funds for the Central Universities (no. 2015ZM085).} ${}^1$ and Yongjie Wang\footnote{Yongjie Wang is supported by the China Postdoctoral Science Foundation (no. 2015M571928) and the Fundamental Research Funds for the Central Universities.} ${}^2$}
\maketitle

\begin{center}
\footnotesize
\begin{itemize}
\item[1] School of Mathematics, South China University of Technology, Guangzhou, Guangdong, 510640, China.
\item[2] School of Mathematical Sciences, University of Science and Technology of China, Hefei, Anhui, 230026, China.
\end{itemize}
\end{center}

\begin{abstract}
The key ingredient of this paper is the universal central extension of the generalized orthosymplectic Lie superalgebra $\mathfrak{osp}_{m|2n}(R,{}^-)$ coordinatized by a unital associative superalgebra $(R,{}^-)$ with superinvolution. Such a universal central extension will be constructed via a Steinberg orthosymplectic Lie superalgebra coordinated by $(R,{}^-)$. The research on the universal central extension of $\mathfrak{osp}_{m|2n}(R,{}^-)$ will yield an identification between the second homology group of the generalized orthosymplectic Lie superalgebra $\mathfrak{osp}_{m|2n}(R,{}^-)$ and the first  $\mathbb{Z}/2\mathbb{Z}$-graded skew-dihedral homology group of $(R,{}^-)$ for $(m,n)\neq(2,1),(1,1)$. The universal central extensions of $\mathfrak{osp}_{2|2}(R,{}^-)$ and $\mathfrak{osp}_{1|2}(R,{}^-)$ will also be treated separately.
\medskip

\noindent\textit{MSC(2010):} 17B05, 19D55.
\medskip

\noindent\textit{Key words:} Orthosymplectic Lie superalgebra; Universal central extension; $\mathbb{Z}/2\mathbb{Z}$-graded skew-dihedral homology; Associative superalgebra with superinvolution.
\end{abstract}

\section{Introduction}
\label{sec:intr}

Central extension theory is of vital importance in the study of Lie algebras and Lie superalgebras, which has attracted the attentions from mathematicians and physicist in recent decades (see \cite{Neher2003} and \cite{ScheunertZhang1998} for a survey). In recent paper \cite{ChangWang2015}, the authors determined the universal central extensions of generalized periplectic Lie superalgebras and established a connection between the second homology group of a generalized periplectic Lie superalgebra and the dihedral group of its coordinate algebra with certain superinvolution.  The current paper is devoted to study the universal central extensions of the generalized orthosymplectic Lie superalgebras coordinatized by unital associative superalgebras with superinvolution, which form another family of Lie superalgebras determined by superinvolutions.

Let $\Bbbk$ be a unital commutative ring with $2$ invertible and $R$ an associative $\Bbbk$-superalgebra. We consider the associative superalgebra $\mathrm{M}_{m|2n}(R)$ of $(m+2n)\times(m+2n)$-matrices, in which the degree of the matrix unit $e_{ij}(a)$ is
$$|e_{ij}(a)|:=|i|+|j|+|a|$$
for a homogeneous element $a\in R$ of degree $|a|$ and $1\leqslant i,j\leqslant m+2n$ with the parity given by
\begin{equation}
|i|:=\begin{cases}0,&\text{if }i\leqslant m,\\1,&\text{if }i>m.\end{cases}\label{intro:eq:deg_i}
\end{equation}
Analogous to the generalized periplectic Lie superalgebra discussed in \cite{ChangWang2015}, we further assume that the associative superalgebra $R$ is equipped with a superinvolution, which is a $\Bbbk$-linear map ${}^-:R\rightarrow R$ preserving the $\mathbb{Z}/2\mathbb{Z}$-gradings and satisfying
\begin{equation}
\overline{ab}=(-1)^{|a||b|}\bar{b}\bar{a},\text{ and }\bar{\bar{a}}=a,\label{intro:eq:supinv}
\end{equation}
for homogeneous $a,b\in R$. In this situation, the associative superalgebra $\mathrm{M}_{m|2n}(R)$ possesses \textit{the generalized orthosymplectic superinvolution} given by
\begin{equation}
\begin{pmatrix}A&B_1&C_2\\ C_1&D_{11}&D_{12}\\ B_2&D_{21}&D_{22}\end{pmatrix}^{\mathrm{osp}}
:=\begin{pmatrix}\overline{A}^t&-\overline{\rho(B_2)}^t&\overline{\rho(C_1)}^t\\
\overline{\rho(C_2)}^t&\overline{D}_{22}^t&-\overline{D}_{12}^t\\
-\overline{\rho(B_1)}^t&-\overline{D}_{21}^t&\overline{D}_{11}^t\end{pmatrix},
\label{intro:eq:inv_ospR}
\end{equation}
where $A$ is an $m\times m$-matrix, $B_1, B_2^t, C_1^t, C_2$ are $m\times n$-matrices, $D_{ij},i,j=1,2$ are $n\times n$-matrices and $\rho:R\rightarrow R$ is a $\Bbbk$-linear map such that
\begin{equation}
\rho(a)=(-1)^{|a|}a,\label{intro:eq:rho}
\end{equation}
for homogeneous $a\in R$, $\rho(A)$ denotes the matrix $(\rho(a_{ij}))$ and $\overline{A}=(\overline{a_{ij}})$ for $A=(a_{ij})$.  The generalized orthosymplectic superinvolution determines a Lie superalgebra over $\Bbbk$:
\begin{equation}
\widetilde{\mathfrak{osp}}_{m|2n}(R,{}^-):=\{X\in\mathrm{M}_{m|2n}(R)|X^{\mathrm{osp}}=-X\},\label{intro:eq:osptilde}
\end{equation}
on which the Lie superbracket is given by the standard super-commutator of matrices. Its derived Lie sub-superalgebra
\begin{equation}
\mathfrak{osp}_{m|2n}(R,{}^-):=[\widetilde{\mathfrak{osp}}_{m|2n}(R,{}^-),\widetilde{\mathfrak{osp}}_{m|2n}(R,{}^-)],
\label{intro:eq:osp}
\end{equation}
is called \textit{a generalized orthosymplectic Lie superalgebra coordinatized by $(R,{}^-)$}.
\bigskip

Generalized orthosymplectic Lie superalgebras realize many important finite-dimensional or infinite-dimensional Lie superalgebras when $(R,{}^-)$ varies. Their universal central extensions have been studied separately by many mathematicians. In the especial case where the base ring $\Bbbk$ is the field $\mathbb{C}$ of complex numbers, we view $\mathbb{C}$ as an associative superalgebra with zero odd part and identity map is a superinvolution on $\mathbb{C}$. Then the above realization of a generalized orthosymplectic Lie superalgebra exactly gives the ordinary orthosymplectic Lie superalgebra $\mathfrak{osp}_{m|2n}(\mathbb{C})$ as decribed in \cite{Kac1977}, which provides us with the finite-dimensional simple complex Lie superalgebras of type $B(m,n)$, $C(n)$ and $D(m,n)$. The universal central extensions of these simple Lie superalgebras have been given in \cite{FuchLeites1984}.

If $R$ is a unital super-commutative associative algebra over an arbitrary base ring $\Bbbk$ with $2$ invertible, we will see in Section~\ref{sec:osp} that identity map is a superinvolution on $R$ and $\mathfrak{osp}_{m|2n}(R,\mathrm{id})$ is isomorphic to the Lie superalgebra $\mathfrak{osp}_{m|2n}(\Bbbk)\otimes_{\Bbbk}R$. These Lie superalgebras have been demonstrated to be crucial in root graded Lie superalgebras of type $C(n)$ and $D(m,n)$ (c.f. \cite{BenkartElduque2002, BenkartElduque2003}). Under the assumption that $R$ is super-commutative, the Lie superalgebra $\mathfrak{osp}_{m|2n}(\Bbbk)\otimes_{\Bbbk}R$ has also been interpreted using the language of bilinear forms in \cite{Duff2002}. The universal central extensions and second homology groups of these Lie superalgebras have been completely determined in \cite{IoharaKoga2001} and \cite{IoharaKoga2005}.

For a unital associative superalgebra $R$ that is not necessarily super-commutative, the generalized orthosymplectic Lie superalgebras $\mathfrak{osp}_{m|2n}(R,{}^-)$ also realize a series of important objects in Lie theory. For instance, an arbitrary unital associative superalgebra $S$ and its opposite superalgebra $S^{\mathrm{op}}$ give rise to a new associative superalgebra $S\oplus S^{\mathrm{op}}$, on which the $\Bbbk$-linear map $\mathrm{ex}:S\oplus S^{\mathrm{op}}\rightarrow S\oplus S^{\mathrm{op}}$ that exchanges the two direct summands is a superinvolution. We will see in Section~\ref{sec:osp} that the generalized orthosymplectic Lie superalgebra $\mathfrak{osp}_{m|2n}(S\oplus S^{\mathrm{op}},\mathrm{ex})$ is isomorphic to the special linear Lie superalgebra $\mathfrak{sl}_{m|2n}(S)$, that is the derived sub-superalgebra of the Lie superalgebra $\mathfrak{gl}_{m|2n}(S)$ of all $(m+2n)\times(m+2n)$-matrices with entries in $S$. Recent research \cite{ChenSun2015} on the universal central extension of the special linear Lie superalgebra $\mathfrak{sl}_{m|n}(S)$\footnote{The universal central extension of $\mathfrak{sl}_{m|n}(S)$ under different assumptions has also been studied in \cite{MikhalevPinchuk2002} and \cite{ChenGuay2013}} yielded the notion of Steinberg Lie superalgebras coordinatized by $S$ and established
the connection between the second homology group of $\mathfrak{sl}_{m|n}(S)$ and the first $\mathbb{Z}/2\mathbb{Z}$-graded cyclic homology group of $S$.  This indeed a super-version generalization of the remarkable results about the universal central extension of the Lie algebra $\mathfrak{sl}_n(S)$ given in \cite{KasselLoday1982}.

A generalized orthosymplectic Lie superalgebra coordinatized by a unital associative superalgebra with superinvolution can also be regarded as a generalization of an elementary unitary Lie algebra coordinatized by a unital associative algebra with anti-involution. In the case of $n=0$, $\mathfrak{osp}_{m|0}(R,{}^-)$ is indeed an elementary unitary Lie algebra in the sense of \cite{Gao1996II} provided that $R$ is a unital associative algebra (that is viewed as an associative superalgebra with zero odd part). For $m\geqslant5$, the universal central extensions of these elementary unitary Lie algebras have been demonstrated to be the Steinberg unitary Lie algebras, which leads to the identification between the second homology group of an elementary unitary Lie algebra and the first skew-dihedral homology group of its coordinates algebras for $m\geqslant5$ \cite{Gao1996II}. The Steinberg unitary Lie algebras have also been successfully used in the study of the compact forms of certain intersection matrix algebras of Slodowy \cite{Gao1996I}.
\bigskip

Our primary aim in this paper is to explicitly characterize the universal central extensions of the generalized orthosymplectic Lie superalgebras $\mathfrak{osp}_{m|2n}(R,{}^-)$ for positive integers $m$ and $n$. We will work under a quite general setting that $\Bbbk$ is a unital commutative associative ring with $2$ invertible and $(R,{}^-)$ is an arbitrary unital associative $\Bbbk$-superalgebra with superinvolution. Our aim will be achieved via introducing the notion of Steinberg orthosymplectic Lie superalgebra $\mathfrak{sto}_{m|2n}(R,{}^-)$ coordinatized by $(R,{}^-)$ (See Section~\ref{sec:sto}). Then, we will prove in Section~\ref{sec:cext} that the canonical homomorphism $\psi:\mathfrak{sto}_{m|2n}(R,{}^-)\rightarrow\mathfrak{osp}_{m|2n}(R,{}^-)$ is a central extension for $(m,n)\neq(1,1)$, whose kernel is identified with the first $\mathbb{Z}/2\mathbb{Z}$-graded skew-dihedral homology group $\fourIdx{}{-}{}{1}{\mathrm{HD}}(R,{}^-)$ as described in \cite{ChangWang2015}. As will be stated in Section~\ref{sec:uce}, the central extension $\psi:\mathfrak{sto}_{m|2n}(R,{}^-)\rightarrow\mathfrak{osp}_{m|2n}(R,{}^-)$ turns out to be universal for $(m,n)\neq(2,1),(1,1)$. The cases of $\mathfrak{osp}_{2|2}(R,{}^-)$ and $\mathfrak{osp}_{1|2}(R,{}^-)$ are exceptional. We will provide a concrete construction for the universal central extension of $\mathfrak{osp}_{2|2}(R,{}^-)$ under certain assumption in Section~\ref{sec:osp22}. Finally, the universal central extension of $\mathfrak{osp}_{1|2}(R,{}^-)$ will be obtained in Section~\ref{sec:osp12}. It yields an interesting result that the second homology group of the Lie superalgebra $\mathfrak{osp}_{1|2}(R,{}^-)$ for an arbitrary $(R,{}^-)$ can be explicitly characterized using a modified version of the first $\mathbb{Z}/2\mathbb{Z}$-graded skew-dihedral homology group.

In summary, our results encompass explicit characterizations of the universal central extensions of $\mathfrak{osp}_{m|2n}(R,{}^-)$ for $(m,n)\neq(2,1)$ and $(R,{}^-)$ an arbitrary associative superalgebra with superinvolution (see Theorem~\ref{thm:sto_uce} and Theorem~\ref{thm:osp12_uce}), as well as the universal central extension of $\mathfrak{osp}_{2|2}(R,{}^-)$ for $(R,{}^-)$ satisfying certain assumption (see Theorem~\ref{thm:osp22_uce}). Consequently, we reveal the second homology groups of all these Lie superalgebras (see Corollaries~\ref{cor:hml},~\ref{cor:osp22_hml} and~\ref{cor:osp12_hml}). Our results recover the  consequences about the universal central extensions of the Lie superalgebras $\mathfrak{osp}_{m|2n}(\mathbb{C})$, $\mathfrak{osp}_{m|2n}(\Bbbk)\otimes_{\Bbbk}R$ for a unital supercommutative associative superalgerba $R$ and $\mathfrak{sl}_{m|2n}(S)$ for a unital associative superalgebra $S$, which have been separately obtained in \cite{FuchLeites1984}, \cite{IoharaKoga2001} and \cite{ChenSun2015}.
\bigskip

\noindent{\bf Notations and terminologies:}

Throughout this paper, $\mathbb{Z}$, $\mathbb{Z}_+$ and $\mathbb{N}$ will denote the sets of integers, non-negative integers and positive integers, respectively. $\Bbbk$ always denotes a unital commutative associative base ring with $2$ invertible. All modules, associative (super)algebras and Lie (super)algebras are assumed to be over $\Bbbk$.

We always write $\mathbb{Z}/2\mathbb{Z}=\{0,1\}$. For an element $x$ of an associative or Lie superalgebra, we use $|x|$ to denote the degree of $|x|$ with respect to the $\mathbb{Z}/2\mathbb{Z}$-grading.

Let $R$ be an associative superalgebra. Then $\mathrm{M}_m(R)$, $\mathrm{M}_{m|n}(R)$ and $\mathrm{M}_{m\times n}(R)$ will denote the associative algebra of all $m\times m$-matrices with entries in $R$, the associative superalgebra of all $(m+n)\times(m+n)$-matrices with entries in $R$ and the set of all $m\times n$-matrices with entries in $R$, respectively. The notation $e_{ij}(a)$ means the matrix with $a$ at the $(i,j)$-position and $0$ elsewhere.

Given an associative superalgebra $(R,{}^-)$ with superinvolution, we set
\begin{equation}
R_{\pm}:=\{a\in R|\bar{a}=\pm a\}.
\end{equation}
Note that $2$ is invertible in $\Bbbk$, we know that $R=R_+\oplus R_-$ as $\Bbbk$-modules. For an element $a\in R$, we denote
$a_{\pm}:=a\pm\bar{a}\in R_{\pm}$.

\section{Basics on generalized orthosymplectic Lie superalgebras}
\label{sec:osp}

In this section, we will discuss the perfectness and generators of the generalized orthosymplectic Lie superalgebra $\mathfrak{osp}_{m|2n}(R,{}^-)$ coordinatized by an associative superalgebra $(R,{}^-)$ with superinvolution. The consequences will be used in our discussion on Steinberg orthosymplectic Lie superalgebras and central extensions of $\mathfrak{osp}_{m|2n}(R,{}^-)$ later.

It is easy to observe from the definition (\ref{intro:eq:osptilde}) that every element of $\widetilde{\mathfrak{osp}}_{m|2n}(R,{}^-)$ is of the form
$$\begin{pmatrix}A&B&-\overline{\rho(C)}^t\\ C&D_{11}&D_{12}\\ \overline{\rho(B)}^t&D_{21}&-\overline{D}_{11}^t\end{pmatrix}$$
where $A\in\mathrm{M}_m(R)$, $B,C^t\in\mathrm{M}_{m\times n}(R)$ and $D_{11}, D_{12}, D_{21}\in\mathrm{M}_n(R)$ satisfy $\overline{A}^t=-A$, $\overline{D}_{12}^t=D_{12}$  and $\overline{D}_{21}^t=D_{21}$. The generalized orthosymplectic Lie superalgebra $\mathfrak{osp}_{m|2n}(R,{}^-)$ is the derived Lie sub-superalgebra of the Lie superalgebra $\widetilde{\mathfrak{osp}}_{m|2n}(R,{}^-)$,
i.e.,
$$\mathfrak{osp}_{m|2n}(R,{}^-):=[\widetilde{\mathfrak{osp}}_{m|2n}(R,{}^-),\widetilde{\mathfrak{osp}}_{m|2n}(R,{}^-)].$$
Before proving the properties of $\mathfrak{osp}_{m|2n}(R,{}^-)$, we first give a few examples:

\begin{example}
\label{eg:osp_com}
Let $R$ be a unital super-commutative associative superalgebra on which the identity map is a superinvolution. Then
$$\mathfrak{osp}_{m|2n}(R,\mathrm{id})\cong\mathfrak{osp}_{m|2n}(\Bbbk)\otimes_{\Bbbk}R,$$
for $m,n\in\mathbb{Z}_+$ with $m+n\geqslant1$.
\end{example}
\begin{proof}
Under the assumption that $R$ is super-commutative, an isomorphism from $\mathfrak{osp}_{m|2n}(\Bbbk)\otimes_{\Bbbk}R$ onto $\mathfrak{osp}_{m|2n}(R,\mathrm{id})$ can be given as follows:
\begin{align*}
\mathfrak{osp}_{m|2n}(\Bbbk)\otimes_{\Bbbk}R&\rightarrow\mathfrak{osp}_{m|2n}(R,\mathrm{id}),\\
(e_{ij}-e_{ji})\otimes a&\mapsto e_{ij}(a)-e_{ji}(a),\\
(e_{m+k,m+l}-e_{m+n+l,m+n+k})\otimes a&\mapsto e_{m+k,m+l}(\rho(a))-e_{m+n+l,m+n+k}(\rho(a)),\\
(e_{m+k,m+n+l}+e_{m+l,m+n+k})\otimes a&\mapsto e_{m+k,m+n+l}(\rho(a))+e_{m+l,m+n+k}(\rho(a)),\\
(e_{m+n+k,m+l}+e_{m+n+l,m+k})\otimes a&\mapsto e_{m+n+k,m+l}(\rho(a))+e_{m+n+l,m+k}(\rho(a)),\\
(e_{i,m+k}+e_{m+n+k,i})\otimes a&\mapsto e_{i,m+k}(\rho(a))+e_{m+n+k,i}(a),\\
(e_{m+k,i}-e_{i,m+n+k})\otimes a&\mapsto e_{m+k,i}(a)-e_{i,m+n+k}(\rho(a)),
\end{align*}
where $\rho$ is the $\Bbbk$-linear map given by (\ref{intro:eq:rho}).
\end{proof}

\begin{example}
\label{eg:osp_SS}
Let $S$ be an arbitrary unital associative superalgebra and $S^{\mathrm{op}}$ be \textit{the opposite superalgebra of $S$} with the multilication
\begin{equation}
a\overset{\mathrm{op}}{\cdot}b=(-1)^{|a||b|}b\cdot a,\label{hml:eq:opalg}
\end{equation}
for homogeneous $a,b\in S$. Then the associative superalgebra $S\oplus S^{\mathrm{op}}$ is naturally equipped with the superinvolution
\begin{equation}
\mathrm{ex}:S\oplus S^{\mathrm{op}}\rightarrow S\oplus S^{\mathrm{op}}, \quad a\oplus b\mapsto b\oplus a.
\label{hml:eq:inv_ex}
\end{equation}
In this situation, the generalized orthosymplectic Lie superalgebra $\mathfrak{osp}_{m|2n}(S\oplus S^{\mathrm{op}},\mathrm{ex})$ is isomorphic to the Lie superalgebra $\mathfrak{sl}_{m|2n}(S)$ for $m,n\in\mathbb{Z}_+$ with $m+n\geqslant1$.
\end{example}
\begin{proof}
Indeed, there is an isomorphism from the Lie superalgebra $\mathfrak{gl}_{m|2n}(S)$ to the Lie superalgebra $\widetilde{\mathfrak{osp}}_{m|2n}(S\oplus S^{\mathrm{op}},\mathrm{ex})$ given as follows:
\begin{align*}
\mathfrak{gl}_{m|2n}(S)&\rightarrow \widetilde{\mathfrak{osp}}_{m|2n}(S\oplus S^{\mathrm{op}},\mathrm{ex})\\
e_{ij}(a)&\mapsto e_{ij}(a\oplus0)-e_{ji}(0\oplus a),\\
e_{i,m+k}(a)&\mapsto e_{i,m+k}(a\oplus0)+e_{m+n+k,i}(0\oplus \rho(a)),\\
e_{i,m+n+k}(a)&\mapsto e_{i,m+n+k}(a\oplus0)-e_{m+k,i}(0\oplus \rho(a)),\\
e_{m+k,i}(a)&\mapsto-e_{i,m+n+k}(0\oplus \rho(a))+e_{m+k,i}(a\oplus 0),\\
e_{m+k,m+l}(a)&\mapsto e_{m+k,m+l}(a\oplus 0)-e_{m+n+l,m+n+k}(0\oplus a),\\
e_{m+k,m+n+l}(a)&\mapsto e_{m+k,m+n+l}(a\oplus 0)+e_{m+l,m+n+k}(0\oplus a), \\
e_{m+n+k,i}(a)&\mapsto e_{i,m+k}(0\oplus \rho(a))+e_{m+n+k,i}(a\oplus 0),\\
e_{m+n+k,m+l}(a)&\mapsto e_{m+n+k,m+l}(a\oplus0)-e_{m+n+l,m+k}(0\oplus a),\\
e_{m+n+k,m+n+l}(a)&\mapsto -e_{m+l,m+k}(0\oplus a)+e_{m+n+k,m+n+l}(a\oplus0),
\end{align*}
for $a\in S$ and $1\leqslant i,j\leqslant m$, $1\leqslant k,l\leqslant n$. It further yields that $\mathfrak{osp}_{m|2n}(S\oplus S^{\mathrm{op}},\mathrm{ex})$ is isomorphic to $\mathfrak{sl}_{m|2n}(S)$.
\end{proof}

\begin{example}
\label{eg:osp_prp}
Suppose that the base ring $\Bbbk$ is an algebraically closed field of characteristic zero. Then the associative superalgebra $\mathrm{M}_{l|l}(\Bbbk)$ for $l\in\mathbb{N}$ has the periplectic superinvolution given by
$$\begin{pmatrix}A&B\\C&D\end{pmatrix}^{\mathrm{prp}}:=\begin{pmatrix}D^t&-B^t\\C^t&A^t\end{pmatrix}.$$
The generalized orthosymplectic Lie superalgebra $\mathfrak{osp}_{m|2n}(\mathrm{M}_{l|l}(\Bbbk),\mathrm{prp})$ is isomorphic to the ordinary periplectic Lie superlagebra $\mathfrak{p}_{(m+2n)l}(\Bbbk)$ that is the derived Lie sub-superalgebra of $$\widetilde{\mathfrak{p}}_{(m+2n)l}(\Bbbk)=\{X\in\mathrm{M}_{(m+2n)l|(m+2n)l}(\Bbbk)|X^{\mathrm{prp}}=-X\}.$$
\end{example}
\begin{proof}
We first deduce from \cite{ChangWang2015} that the associative superalgebra $\mathrm{M}_{m|2n}(\mathrm{M}_{l|l}(\Bbbk))$ is isomorphic to the associative superalgebra $\mathrm{M}_{(m+2n)l|(m+2n)l}(\Bbbk)$. Hence, the orthosymplectic superinvolution $\sigma$ on $\mathrm{M}_{m|2n}(\mathrm{M}_{l|l}(\Bbbk))$ induces an superinvolution $\tilde{\sigma}$ on $\mathrm{M}_{(m+n)l|(m+n)l}(\Bbbk)$.  We first claim that $\tilde{\sigma}$ is equivalent\footnote{Two superinvolution $\sigma_1$ and $\sigma_2$ on a given associative superalgebra $R$ are said to be equivalent if there is an automorphism of superalgebras $\varphi:R\rightarrow R$ such that $\sigma_1\circ \varphi=\varphi\circ\sigma_2$.} to the periplectic superinvolution on $\mathrm{M}_{(m+2n)l|(m+2n)l}(\Bbbk)$. By \cite[Propositions 13 and 14]{Racine1998}, it suffices to show that the even part $\mathrm{M}_{(m+2n)l|(m+2n)l}(\Bbbk)_0$ has no nonzero proper two-sided ideal that is invariant under $\tilde{\sigma}$, or equivalently, the even part $\mathrm{M}_{m|2n}(\mathrm{M}_{l|l}(\Bbbk))_0$ has no nonzero proper two-sided ideal that is invariant under $\sigma$.

The $\Bbbk$-module $\mathrm{M}_{m|2n}(\mathrm{M}_{l|l}(\Bbbk))_0$ is spanned by:
\begin{align*}
&e_{i,i'}\begin{pmatrix}A&0\\0&B\end{pmatrix},\,
e_{m+j,m+j'}\begin{pmatrix}B&0\\0&A\end{pmatrix},\,
e_{m+j,m+n+j'}\begin{pmatrix}B&0\\0&A\end{pmatrix},\\
&e_{m+n+j,m+j'}\begin{pmatrix}B&0\\0&A\end{pmatrix},\,
e_{m+n+j,m+n+j'}\begin{pmatrix}B&0\\0&A\end{pmatrix},\,
e_{i,m+j}\begin{pmatrix}0&A\\B&0\end{pmatrix},\\
&e_{i,m+n+j}\begin{pmatrix}0&A\\B&0\end{pmatrix},\,
e_{m+j,i}\begin{pmatrix}0&B\\A&0\end{pmatrix},\,
e_{m+n+j,i}\begin{pmatrix}0&B\\A&0\end{pmatrix},
\end{align*}
where $A,B\in\mathrm{M}_l(\Bbbk)$, $1\leqslant i,i'\leqslant m$ and $1\leqslant j,j'\leqslant n$.

Let $\mathcal{A}$ be the $\Bbbk$-submodule of $\mathrm{M}_{m|2n}(\mathrm{M}_{l|l}(\Bbbk))_0$ spanned by
\begin{align*}
&e_{i,i'}\begin{pmatrix}A&0\\0&0\end{pmatrix},\,
e_{m+j,m+j'}\begin{pmatrix}0&0\\0&A\end{pmatrix},\,
e_{m+j,m+n+j'}\begin{pmatrix}0&0\\0&A\end{pmatrix},\\
&e_{m+n+j,m+j'}\begin{pmatrix}0&0\\0&A\end{pmatrix},\,
e_{m+n+j,m+n+j'}\begin{pmatrix}0&0\\0&A\end{pmatrix},
e_{i,m+j}\begin{pmatrix}0&A\\0&0\end{pmatrix},\\
&e_{i,m+n+j}\begin{pmatrix}0&A\\0&0\end{pmatrix},\,
e_{m+j,i}\begin{pmatrix}0&0\\A&0\end{pmatrix},\,
e_{m+n+j,i}\begin{pmatrix}0&0\\A&0\end{pmatrix},
\end{align*}
for $A\in\mathrm{M}_l(\Bbbk)$, $1\leqslant i,i'\leqslant m$ and $1\leqslant j,j'\leqslant n$. Let $\mathcal{B}$ be the $\Bbbk$-submodule of $\mathrm{M}_{m|2n}(\mathrm{M}_{l|l}(\Bbbk))_0$ spanned by
\begin{align*}
&e_{i,i'}\begin{pmatrix}0&0\\0&B\end{pmatrix},\,
e_{m+j,m+j'}\begin{pmatrix}B&0\\0&0\end{pmatrix},\,
e_{m+j,m+n+j'}\begin{pmatrix}B&0\\0&0\end{pmatrix},\\
&e_{m+n+j,m+j'}\begin{pmatrix}B&0\\0&0\end{pmatrix},\,
e_{m+n+j,m+n+j'}\begin{pmatrix}B&0\\0&0\end{pmatrix},
e_{i,m+j}\begin{pmatrix}0&0\\B&0\end{pmatrix},\\
&e_{i,m+n+j}\begin{pmatrix}0&0\\B&0\end{pmatrix},\,
e_{m+j,i}\begin{pmatrix}0&B\\0&0\end{pmatrix},\,
e_{m+n+j,i}\begin{pmatrix}0&B\\0&0\end{pmatrix},
\end{align*}
where $B\in\mathrm{M}_l(\Bbbk)$, $1\leqslant i,i'\leqslant m$ and $1\leqslant j,j'\leqslant n$.

Then both $\mathcal{A}$ and $\mathcal{B}$ are two-sided ideals of $\mathrm{M}_{m|2n}(\mathrm{M}_{l|l}(\Bbbk))_0$ and
$$\mathrm{M}_{m|2n}(\mathrm{M}_{l|l}(\Bbbk))_0=\mathcal{A}\oplus\mathcal{B}.$$
Moreover, both $\mathcal{A}$ and $\mathcal{B}$ are isomorphic to $\mathrm{M}_{(m+2n)l}(\Bbbk)$ as associative algebras and $\sigma(\mathcal{A})=\mathcal{B}$. It follows that
$$(\mathrm{M}_{m|2n}(\mathrm{M}_{l|l}(\Bbbk))_0,\sigma|)\cong(\mathrm{M}_{(m+2n)l}(\Bbbk)\oplus\mathrm{M}_{(m+2n)l}(\Bbbk)^{\mathrm{op}},\mathrm{ex})$$
as associative superalgebras with superinvolution. Hence, the superinvolution $\tilde{\sigma}$ on the associative superalgebra $\mathrm{M}_{(m+2n)l|(m+2n)l}(\Bbbk)$ is equivalent to the periplectic superinvolution, which proves the claim.

Now, we conclude from the claim that
$$\widetilde{\mathfrak{osp}}_{m|2n}(\mathrm{M}_{l|l}(\Bbbk),\mathrm{prp})\cong\widetilde{\mathfrak{p}}_{(m+2n)l}(\Bbbk),$$
which yields the desired isomorphism.
\end{proof}

\begin{example}
\label{eg:osp_osp}
Suppose that the base ring $\Bbbk$ is an algebraically closed field of characteristic zero. Then the associative superalgebra $\mathrm{M}_{k|2l}(\Bbbk)$ for $k,l\in\mathbb{Z}_+$ with $k+l\geqslant1$ possesses the orthosymplectic superinvolution. In this situation, the generalized orthosymplectic Lie superalgebra $\mathfrak{osp}_{m|2n}(\mathrm{M}_{k|2l}(\Bbbk),\mathrm{osp})$ is isomorphic to $\mathfrak{osp}_{(mk+4nl)|2(nk+ml)}(\Bbbk)$ for $m,n\in\mathbb{Z}_+$ with $m+n\geqslant1$.
\end{example}
\begin{proof}
The proof is similar to Example~\ref{eg:osp_prp}. Let $\sigma$ be the orthosymplectic superinvolution on $\mathrm{M}_{m|2n}(\mathrm{M}_{k|2l}(\Bbbk))$ defined by (\ref{intro:eq:inv_ospR}). It induces a superinvolution $\widetilde{\sigma}$ on $\mathrm{M}_{(mk+4nl)|2(ml+nk)}(\Bbbk)$. We will show that  $\widetilde{\sigma}$ is equivalent to the orthosymplectic superinvolution on $\mathrm{M}_{(mk+4nl)|2(ml+nk)}(\Bbbk)$. It suffices to show that the even part $\mathrm{M}_{m|2n}(\mathrm{M}_{k|2l}(\Bbbk))_0$ has a nonzero proper two-sided ideal that is invariant under $\sigma$  \cite[Propositions~13 and 14]{Racine1998}.

Let $\mathcal{A}$ be the $\Bbbk$-submodule of $\mathrm{M}_{m|2n}(\mathrm{M}_{k|2l}(\Bbbk))_0$ spanned by
\begin{align*}
&e_{i,i'}\begin{pmatrix}A&0\\0&0\end{pmatrix},
e_{m+j,m+j'}\begin{pmatrix}0&0\\0&D\end{pmatrix},
e_{m+j,m+n+j'}\begin{pmatrix}0&0\\0&D\end{pmatrix},\\
&e_{m+n+j,m+j'}\begin{pmatrix}0&0\\0&D\end{pmatrix},
e_{m+n+j,m+n+j'}\begin{pmatrix}0&0\\0&D\end{pmatrix},
e_{i,m+j}\begin{pmatrix}0&B\\0&0\end{pmatrix},\\
&e_{i,m+n+j}\begin{pmatrix}0&B\\0&0\end{pmatrix},
e_{m+j,i}\begin{pmatrix}0&0\\B^t&0\end{pmatrix},
e_{m+n+j,i}\begin{pmatrix}0&0\\B^t&0\end{pmatrix},
\end{align*}
for $A\in\mathrm{M}_k(\Bbbk)$, $B\in\mathrm{M}_{k\times 2l}(\Bbbk)$, $D\in\mathrm{M}_{2l}(\Bbbk)$, $1\leqslant i,i'\leqslant m, 1\leqslant j,j'\leqslant n$. Let $\mathcal{B}$ be the $\Bbbk$-submodule of $\mathrm{M}_{m|2n}(\mathrm{M}_{k|2l}(\Bbbk))_0$ spanned by
\begin{align*}
&e_{i,i'}\begin{pmatrix}0&0\\0&D\end{pmatrix},
e_{m+j,m+j'}\begin{pmatrix}A&0\\0&0\end{pmatrix},
e_{m+j,m+n+j'}\begin{pmatrix}A&0\\0&0\end{pmatrix},\\
&e_{m+n+j,m+j'}\begin{pmatrix}A&0\\0&0\end{pmatrix},
e_{m+n+j,m+n+j'}\begin{pmatrix}A&0\\0&0\end{pmatrix},
e_{i,m+j}\begin{pmatrix}0&0\\B^t&0\end{pmatrix},\\
&e_{i,m+n+j}\begin{pmatrix}0&0\\B^t&0\end{pmatrix},
e_{m+j,i}\begin{pmatrix}0&B\\0&0\end{pmatrix},
e_{m+n+j,i}\begin{pmatrix}0&B\\0&0\end{pmatrix},
\end{align*}
for $A\in\mathrm{M}_k(\Bbbk)$, $B\in\mathrm{M}_{k\times 2l}(\Bbbk)$, $D\in\mathrm{M}_{2l}(\Bbbk)$, $1\leqslant i,i'\leqslant m, 1\leqslant j,j'\leqslant n$.

Then both $\mathcal{A}$ and $\mathcal{B}$ are two-sided ideals of $\mathrm{M}_{m|2n}(\mathrm{M}_{k|2l}(\Bbbk))_0$ that are invariant under $\sigma$. We also observe that $\mathrm{M}_{m|2n}(\mathrm{M}_{k|2l}(\Bbbk))_0=\mathcal{A}\oplus\mathcal{B}$. Furthermore, under the isomorphism between the associative superalgebras $\mathrm{M}_{m|2n}(\mathrm{M}_{k|2l}(\Bbbk))$ and $\mathrm{M}_{(mk+4nl)|2(ml+nk)}(\Bbbk)$, we have $\mathcal{A}\cong\mathrm{M}_{mk+4nl}(\Bbbk)$ on which the restriction of $\sigma$ is of the orthogonal type, and $\mathcal{B}\cong\mathrm{M}_{2(ml+nk)}(\Bbbk)$ on which the restriction of $\sigma$ is of the symplectic type. Hence, $\widetilde{\sigma}$ is equivalent to the orthosymplectic superinvolution on $\mathrm{M}_{(mk+4nl)|2(nk+ml)}(\Bbbk)$, which implies that
$$\widetilde{\mathfrak{osp}}_{m|2n}(\mathrm{M}_{k|2l}(\Bbbk),\mathrm{osp})\cong\widetilde{\mathfrak{osp}}_{(mk+4nl)|2(ml+nk)}(\Bbbk)$$
and their derived Lie sub-superalgebras are also isomorphic.
\end{proof}
\bigskip

The rest of this section will be devoted to prove the perfectness of $\mathfrak{osp}_{m|2n}(R,{}^-)$ for an arbitrary unital associative superalgebra $(R,{}^-)$ with superinvolution and to calculate the generators of $\mathfrak{osp}_{m|2n}(R,{}^-)$. We introduce the following notation:
\begin{align*}
t_{ij}(a):&=e_{ij}(a)-e_{ji}(\bar{a}),\\
u_{kl}(a):&=e_{m+k,m+l}(a)-e_{m+n+l,m+n+k}(\bar{a}),\\
v_{kl}(a):&=e_{m+k,m+n+l}(a)+e_{m+l,m+n+k}(\bar{a}),\\
w_{kl}(a):&=e_{m+n+k,m+l}(a)+e_{m+n+l,m+k}(\bar{a}),\\
f_{ik}(a):&=e_{i,m+k}(a)+e_{m+n+k,i}(\rho(\bar{a})),\\
g_{ki}(a):&=e_{m+k,i}(a)-e_{i,m+n+k}(\rho(\bar{a})),
\end{align*}
where $a\in R$, $1\leqslant i,j\leqslant m$ and $1\leqslant k,l\leqslant n$. These elements will serve as generators of $\mathfrak{osp}_{m|2n}(R,{}^-)$. We first provide an explicit description for elements of $\mathfrak{osp}_{m|2n}(R,{}^-)$.

\begin{lemma}
\label{lem:osp_ele}
Let $m,n\in\mathbb{N}$ and $(R,{}^-)$ be a unital associative superalgebra with superinvolution. Then every element $x\in\widetilde{\mathfrak{osp}}_{m|2n}(R,{}^-)$ is written as
\begin{equation}
\begin{aligned}
x=&e_{11}(a)+\sum\limits_{i=2}^m(e_{ii}(a_i)-e_{11}(a_i))
+\sum\limits_{1\leqslant i< j\leqslant m}{t}_{ij}(a_{ij})\\
&+\sum\limits_{k=1}^n({u}_{kk}(d_k)+e_{11}(\rho(d_k)-\rho(\overline{d_k})))
+\sum\limits_{1\leqslant k\neq l\leqslant n}{u}_{kl}(d_{kl})\\
&+\sum\limits_{k=1}^n(e_{m+k,m+n+k}(d'_k)+e_{m+n+k,m+k}(d''_k))
+\sum\limits_{1\leqslant k< l\leqslant n}({v}_{kl}(d'_{kl})+{w}_{kl}(d''_{kl}))\\
&+\sum\limits_{\substack{1\leqslant i\leqslant m\\ 1\leqslant k\leqslant n}}({f}_{ik}(b_{ik})+{g}_{ki}(c_{ki})),
\end{aligned}
\label{eq:osp_ele}
\end{equation}
where $a,a_i\in R_-, a_{ij},b_{ik},c_{ki},d_{kl},d'_{kl},d''_{kl},d_k\in R$, and $d'_k, d''_k\in R_+$ are uniquely determined by $x$. Moreover, $x\in\mathfrak{osp}_{m|2n}(R,{}^-)$ if and only if $a\in [R,R]\cap R_-$.
\end{lemma}
\begin{proof}
The first statement is obvious. We now show that $x\in \mathfrak{osp}_{m|2n}(R,{}^-)$ if and only if $a\in [R,R]\cap R_-$.

Since $\frac{1}{2}\in\Bbbk$, we deduce that
\begin{align}
e_{ii}(a_i)-e_{11}(a_i)&=[{t}_{i1}(a_i),{t}_{1i}(1)],&&\text{for }i=2,\ldots,m,\label{eq:osp_gen01}\\
e_{m+k,m+n+k}(d'_k)&=-\frac{1}{2}[{g}_{k1}(d'_k),{g}_{k1}(1)],&&\text{for }k=1,\ldots,n,\label{eq:osp_gen02}\\
e_{m+n+k,m+k}(d''_k)&=\frac{1}{2}[{f}_{1k}(d''_k),{f}_{1k}(1)],&&\text{for }k=1,\ldots,n,\label{eq:osp_gen03}\\
{u}_{kk}(d_k)+e_{11}(\rho(d_k)-\rho(\overline{d_k}))&=[{g}_{k1}(1),{f}_{1k}(d_k)],&&\text{for }k=1,\ldots,n,\label{eq:osp_gen04}\\
{t}_{ij}(a_{ij})&=[{f}_{i1}(1),{g}_{1j}(a_{ij})],&&\text{for }1\leqslant i\neq j\leqslant m,\label{eq:osp_gen05}\\
{u}_{kl}(d_{kl})&=[{g}_{k1}(d_{kl}),{f}_{1l}(1)],&&\text{for }1\leqslant k\neq l\leqslant n,\label{eq:osp_gen06}\\
{v}_{kl}(d'_{kl})&=-[{g}_{k1}(d'_{kl}),{g}_{l1}(1)],&&\text{for }1\leqslant k\neq l\leqslant n,\label{eq:osp_gen07}\\
{w}_{kl}(d''_{kl})&=[{f}_{1k}(1),{f}_{1l}(d''_{kl})],&&\text{for }1\leqslant k \neq l\leqslant n,\label{eq:osp_gen08}\\
{f}_{ik}(b_{ik})&=[{f}_{ik}(b_{ik}),{u}_{kk}(1)],&&\text{for }1\leqslant i\leqslant  m,1\leqslant k\leqslant n,\label{eq:osp_gen09}\\
{g}_{ki}(c_{ki})&=[{u}_{kk}(1),{g}_{ki}(c_{ki})],&&\text{for }1\leqslant i\leqslant  m,1\leqslant k\leqslant n,\label{eq:osp_gen10}
\end{align}
where $a_i\in R_-$, $d'_k,d''_k\in R_+$ and $a_{ij}, b_{ik}, c_{ki}, d_{kl}, d'_{kl}, d''_{kl}\in R$.

Hence, all items except $e_{11}(a)$ appeared in the decomposition (\ref{eq:osp_ele}) of $x$ are contained in the derived Lie superalgebra $\mathfrak{osp}_{m|2n}(R,{}^-)$. The problem is reduced to proving that $e_{11}(a)\in\mathfrak{osp}_{m|2n}(R,{}^-)$ if and only if $a\in[R,R]\cap R_-$.

If $a\in [R,R]\cap R_-$, we write $a=\sum [a_i',a_i'']$ with $a_i', a_i''\in R$ and compute that
\begin{align*}
\sum_i([{f}_{11}(a_i'),{g}_{11}(a_i'')]-(-1)^{|a_i'||a_i''|}[{f}_{11}(1),{g}_{11}(a_i''a_i')])
=e_{11}\left(\sum_i([a_i',a_i'']-\overline{[a_i',a_i'']})\right)
=2e_{11}(a),
\end{align*}
which yields that $e_{11}(a)\in\mathfrak{osp}_{m|2n}(R,{}^-)$ since $2$ is invertible in $R$.

Conversely, we define a $\Bbbk$--linear map
\begin{equation}
\varepsilon:\widetilde{\mathfrak{osp}}_{m|2n}(R,{}^-)\rightarrow R_-,
\quad
\begin{pmatrix}
A&B&-\overline{\rho(C)}^t\\
C&D_{11}&D_{12}\\
\overline{\rho(B)}^t&D_{21}&-\overline{D}_{11}^t
\end{pmatrix}
\mapsto\mathrm{Tr}(A)-\mathrm{Tr}(\rho(D_{11})-\overline{\rho(D_{11})}).
\label{eq:osp_tr}
\end{equation}
Then it is directly verified that $\varepsilon(x)\in[R,R]\cap R_-$ for every element $x\in\mathfrak{osp}_{m|2n}(R,{}^-)$. Hence, $e_{11}(a)\in\mathfrak{osp}_{m|2n}(R,{}^-)$ implies that $\varepsilon(e_{11}(a))=a\in[R,R]\cap R_-$.  This completes the proof.
\end{proof}

Now, we may proceed to state and prove the main proposition in this section.

\begin{proposition}
\label{prop:ospalg}
Let $m,n\in\mathbb{N}$ and $(R,{}^-)$ a unital associative superalgebra with superinvolution.
\begin{enumerate}
\item There is an exact sequence of Lie superalgebras
$$0\rightarrow\mathfrak{osp}_{m|2n}(R,{}^-)\rightarrow\widetilde{\mathfrak{osp}}_{m|2n}(R,{}^-)\rightarrow \frac{R_-}{[R,R]\cap R_-}\rightarrow0.$$
\item The Lie superalgebra $\mathfrak{osp}_{m|2n}(R,{}^-)$ is the Lie sub-superalgebra of $\mathfrak{gl}_{m|2n}(R)$ generated by ${t}_{ij}(a)$, ${u}_{kl}(a)$, ${v}_{kl}(a)$, ${w}_{kl}(a)$, ${f}_{i'k'}(a)$ and ${g}_{k'i'}(a)$ for $a\in R$ , $1\leqslant i',i,j\leqslant m$ with $i\neq j$ and $1\leqslant k',k,l\leqslant n$ with $k\neq l$.
\item The Lie superalgebra $\mathfrak{osp}_{m|2n}(R,{}^-)$ is perfect, i.e.,
$$\mathfrak{osp}_{m|2n}(R,{}^-)=[\mathfrak{osp}_{m|2n}(R,{}^-),\mathfrak{osp}_{m|2n}(R,{}^-)].$$
\end{enumerate}
\end{proposition}
\begin{proof}
(i) We consider the surjective $\Bbbk$--linear map:
$$\tilde{\varepsilon}:\widetilde{\mathfrak{osp}}_{m|2n}(R,{}^-)\xrightarrow{\varepsilon} R_-\rightarrow R_-/([R,R]\cap R_-),$$
where $\varepsilon$ is the $\Bbbk$-linear map defined in (\ref{eq:osp_tr}). It follows from Lemma~\ref{lem:osp_ele} that $\mathfrak{osp}_{m|2n}(R,{}^-)=\ker(\tilde{\varepsilon})$. Hence, we obtain an exact sequence of $\Bbbk$--modules:
$$0\rightarrow \mathfrak{osp}_{m|2n}(R,{}^-)\rightarrow\widetilde{\mathfrak{osp}}_{m|2n}(R,{}^-)\rightarrow\frac{R_-}{[R,R]\cap R_-}\rightarrow0,$$
in which all $\Bbbk$--linear maps are homomorphisms of Lie superalgebras since $R_-/([R,R]\cap R_-)$ is a super-commutative Lie superalgebra. Hence, it is an exact sequence of Lie superalgebras.
\medskip

(ii) Let $x\in\mathfrak{osp}_{m|2n}(R,{}^-)$. It follows from Lemma~\ref{lem:osp_ele} that $x$ is written as
\begin{align*}
x=&e_{11}(a)+\sum\limits_{i=2}^m(e_{ii}(a_i)-e_{11}(a_i))
+\sum\limits_{1\leqslant i< j\leqslant m}{t}_{ij}(a_{ij})\nonumber\\
&+\sum\limits_{k=1}^n({u}_{kk}(d_k)+e_{11}(\rho(d_k)-\rho(\overline{d_k})))
+\sum\limits_{1\leqslant k\neq l\leqslant n}{u}_{kl}(d_{kl})\nonumber\\
&+\sum\limits_{k=1}^n(e_{m+k,m+n+k}(d'_k)+e_{m+n+k,m+k}(d''_k))
+\sum\limits_{1\leqslant k< l\leqslant n}({v}_{kl}(d'_{kl})+{w}_{kl}(d''_{kl}))\nonumber\\
&+\sum\limits_{\substack{1\leqslant i\leqslant m\\ 1\leqslant k\leqslant n}}({f}_{ik}(b_{ik})+{g}_{ki}(c_{ki})),
\end{align*}
where $a\in[R,R]\cap R_-,a_i\in R_-, a_{ij},b_{ik},c_{ki},d_{kl},d'_{kl},d''_{kl},d_k\in R$, and $d'_k, d''_k\in R_+$ are uniquely determined by $x$.

Thus, we deduce from (\ref{eq:osp_gen01})-(\ref{eq:osp_gen04}) that $x-e_{11}(a)$ is generated by
${t}_{ij}(a)$, ${u}_{kl}(a)$, ${v}_{kl}(a)$, ${w}_{kl}(a)$, ${f}_{ik}(a)$, ${g}_{ki}(a)$ with $a\in R$, $1\leqslant i\neq j\leqslant m$ and $1\leqslant k\neq l\leqslant n$. Moreover, for $a\in [R,R]\cap R_-$, we write $a=\sum [a_i',a_i'']$ and deduce that
$$e_{11}(a)=\frac{1}{2}\sum_i([{f}_{11}(a_i'),{g}_{11}(a_i'')]-(-1)^{|a_i'||a_i''|}[{f}_{11}(1),{g}_{11}(a_i''a_i')])$$
This proves (ii).
\medskip

(iii) By (ii), the Lie superalgebra $\mathfrak{osp}_{m|2n}(R,{}^-)$ is the Lie sub-superalgebra of $\mathfrak{gl}_{m|2n}(R)$ generated by ${t}_{ij}(a)$, ${u}_{kl}(a)$, ${v}_{kl}(a)$, ${w}_{kl}(a)$, ${f}_{i'k'}(a)$, ${g}_{k'i'}(a)$ for $a\in R$, $1\leqslant i',i, j\leqslant m$ with $i\neq j$ and $1\leqslant k',k,l\leqslant n$ with $k\neq l$. Hence, it suffices to show that these generators are contained in the derived superalgebra $[\mathfrak{osp}_{m|2n}(R,{}^-),\mathfrak{osp}_{m|2n}(R,{}^-)]$.

We have shown in (\ref{eq:osp_gen05})-(\ref{eq:osp_gen08}) that ${t}_{ij}(a)$, ${u}_{kl}(a)$, ${v}_{kl}(a)$, and ${w}_{kl}(a)$ are contained in the derived Lie superalgebra $[\mathfrak{osp}_{m|2n}(R,{}^-),\mathfrak{osp}_{m|2n}(R,{}^-)]$ for $a\in R$, $1\leqslant i\neq j\leqslant m$ and $1\leqslant k\neq l\leqslant n$. We also observe that ${u}_{kk}(1)=[{g}_{k1}(1),{f}_{1k}(1)]\in\mathfrak{osp}_{m|2n}(R,{}^-)$ for $k=1,\ldots,n$, and hence
\begin{align*}
{f}_{ik}(a)=[{f}_{ik}(a),{u}_{kk}(1)]\in[\mathfrak{osp}_{m|2n}(R,{}^-),\mathfrak{osp}_{m|2n}(R,{}^-)],\\
{g}_{ki}(a)=[{u}_{kk}(1),{g}_{ki}(a)]\in[\mathfrak{osp}_{m|2n}(R,{}^-),\mathfrak{osp}_{m|2n}(R,{}^-)],
\end{align*}
for $a\in R$, $1\leqslant i\leqslant m$ and $1\leqslant k\leqslant n$. This proves (iii).
\end{proof}

\section{Steinberg orthosymplectic Lie superalgebras}
\label{sec:sto}

We will introduce Steinberg orthosymplectic Lie superalgebras coordinatized by a unital associative superalgebra $(R,{}^-)$ with superinvolution in this section. It will play crucial role in the study of central extensions of the Lie superalgebra $\mathfrak{osp}_{m|2n}(R,{}^-)$.
\begin{definition}
\label{def:sto}
Let $m,n\in\mathbb{N}$ and $(R,{}^-)$ be a unital associative superalgebra over $\Bbbk$ with superinvolution.
{\it The Steinberg orthosymplectic Lie superalgebra coordinatized by $(R,{}^-)$}, denoted by $\mathfrak{sto}_{m|2n}(R,{}^-)$, is the abstract Lie superalgebra generated by homogenous elements $\mathbf{t}_{ij}(a)$, $\mathbf{u}_{kl}(a)$, $\mathbf{v}_{kl}(a)$, $\mathbf{w}_{kl}(a)$ of degree $|a|$ and homogeneous elements $\mathbf{f}_{i'k'}(a)$, $\mathbf{g}_{k'i'}(a)$ of degree $1+|a|$ for homogeneous $a\in R$, $1\leqslant i', i, j\leqslant m$ with $i\neq j$ and $1\leqslant k', k, l\leqslant n$ with $k\neq l$. The defining relations for $\mathfrak{sto}_{m|2n}(R,{}^-)$ are given as follows:
\begin{align}
&\mathbf{t}_{ij},\mathbf{u}_{kl},\mathbf{v}_{kl},\mathbf{w}_{kl},\mathbf{f}_{i'k'},\mathbf{g}_{k'i'}\text{ are all }\Bbbk\text{-linear},
&&\tag{STO00}\label{STO00}\\
&\mathbf{t}_{ij}(a)=-\mathbf{t}_{ji}(\bar{a}),
&&\text{for }i\neq j,\tag{STO01}\label{STO01}\\
&\mathbf{v}_{kl}(a)=\mathbf{v}_{lk}(\bar{a}),\text{ and }
\mathbf{w}_{kl}(a)=\mathbf{w}_{lk}(\bar{a}),
&&\text{for }k\neq l,\tag{STO02}\label{STO02}\\
&[\mathbf{t}_{ii'}(a),\mathbf{t}_{i'j}(b)]=\mathbf{t}_{ij}(ab),
&&\text{for distinct }i,i',j,\tag{STO03}\label{STO03}\\
&[\mathbf{t}_{ij}(a),\mathbf{t}_{i'j'}(b)]=0,
&&\text{for distinct }i,j,i',j',\tag{STO04}\label{STO04}\\
&[\mathbf{t}_{ij}(a),\mathbf{u}_{kl}(b)]
=[\mathbf{t}_{ij}(a),\mathbf{v}_{kl}(b)]
=[\mathbf{t}_{ij}(a),\mathbf{w}_{kl}(b)]=0,
&&\text{for }i\neq j\text{ and }k\neq l,\tag{STO05}\label{STO05}\\
&[\mathbf{u}_{kk'}(a),\mathbf{u}_{k'l}(b)]=\mathbf{u}_{kl}(ab),
&&\text{for distinct }k, k', l,\tag{STO06}\label{STO06}\\
&[\mathbf{u}_{kl}(a),\mathbf{u}_{k'l'}(b)]=0,
&&\text{for }k\ne l\ne k'\ne l'\ne k,\tag{STO07}\label{STO07}\\
&[\mathbf{u}_{kk'}(a),\mathbf{v}_{k'l}(b)]=\mathbf{v}_{kl}(ab),
&&\text{for distinct }k,k',l,\tag{STO08}\label{STO08}\\
&[\mathbf{u}_{kl}(a),\mathbf{v}_{k'l'}(b)]=0,
&&\text{for }k\ne l\ne k'\ne l'\ne l,\tag{STO09}\label{STO09}\\
&[\mathbf{w}_{lk'}(a),\mathbf{u}_{k'k}(b)]=\mathbf{w}_{lk}(ab),
&&\text{for distinct }k, k', l,\tag{STO10}\label{STO10}\\
&[\mathbf{w}_{l'k'}(a),\mathbf{u}_{lk}(b)]=0,
&&\text{for }k\neq l\ne k'\ne l'\ne l,\tag{STO11}\label{STO11}\\
&[\mathbf{v}_{kl}(a),\mathbf{v}_{k'l'}(b)]=
[\mathbf{w}_{kl}(a),\mathbf{w}_{k'l'}(b)]=0,
&&\text{for }k\ne l\text{ and } k'\ne l',\tag{STO12}\label{STO12}\\
&[\mathbf{v}_{kk'}(a),\mathbf{w}_{k'l}(b)]=\mathbf{u}_{kl}(ab),
&&\text{for distinct }k, k', l,\tag{STO13}\label{STO13}\\
&[\mathbf{v}_{kl}(a),\mathbf{w}_{k'l'}(b)]=0,
&&\text{for distinct }k, l, k', l',\tag{STO14}\label{STO14}\\
&[\mathbf{t}_{ij}(a),\mathbf{f}_{i'k}(b)]=\delta_{i'j}\mathbf{f}_{ik}(ab)-\delta_{i'i}\mathbf{f}_{jk}(\bar{a}b),
&&\text{for }i\neq j,\tag{STO15}\label{STO15}\\
&[\mathbf{g}_{ki'}(a),\mathbf{t}_{ij}(b)]=\delta_{i'i}\mathbf{g}_{kj}(ab)-\delta_{i'j}\mathbf{g}_{ki}(a\bar{b}),
&&\text{for } i\ne j,\tag{STO16}\label{STO16}\\
&[\mathbf{f}_{ik'}(a),\mathbf{u}_{kl}(b)]=\delta_{kk'}\mathbf{f}_{il}(ab),
&&\text{for }k\ne l,\tag{STO17}\label{STO17}\\
&[\mathbf{u}_{lk}(a),\mathbf{g}_{k'i}(b)]=\delta_{k'k}\mathbf{g}_{li}(ab),
&&\text{for }k\ne l,\tag{STO18}\label{STO18}\\
&[\mathbf{v}_{kl}(a),\mathbf{f}_{il'}(b)]
=(-1)^{|b|}(\delta_{l'l}\mathbf{g}_{ki}(a\bar{b})+\delta_{l'k}\mathbf{g}_{li}(\bar{a}\bar{b})),
&&\text{for }k\ne l,\tag{STO19}\label{STO19}\\
&[\mathbf{g}_{k'i}(a),\mathbf{v}_{kl}(b)]=0,
&&\text{for }k\ne l,\tag{STO20}\label{STO20}\\
&[\mathbf{w}_{kl}(a),\mathbf{f}_{ik'}(b)]=0,
&&\text{for } k\ne l,\tag{STO21}\label{STO21}\\
&[\mathbf{g}_{k'i}(a),\mathbf{w}_{kl}(b)]
=-(-1)^{|a|}(\delta_{k'k}\mathbf{f}_{il}(\bar{a}b)+\delta_{k'l}\mathbf{f}_{ik}(\bar{a}\bar{b})),
&&\text{for }k\ne l,\tag{STO22}\label{STO22}\\
&[\mathbf{f}_{ik}(a),\mathbf{f}_{il}(b)]=(-1)^{|a|}\mathbf{w}_{kl}(\bar{a}b),
&&\text{for }k\ne l,\tag{STO23}\label{STO23}\\
&[\mathbf{f}_{ik}(a),\mathbf{f}_{jl}(b)]=0,
&&\text{for }i\ne j,\tag{STO24}\label{STO24}\\
&[\mathbf{g}_{ki}(a),\mathbf{g}_{li}(b)]=-(-1)^{|b|}\mathbf{v}_{kl}(a\bar{b}),
&&\text{for }k\ne l,\tag{STO25}\label{STO25}\\
&[\mathbf{g}_{ki}(a),\mathbf{g}_{lj}(b)]=0,
&&\text{for }i\ne j,\tag{STO26}\label{STO26}\\
&[\mathbf{f}_{ik}(a),\mathbf{g}_{lj}(b)]=\delta_{kl}\mathbf{t}_{ij}(ab),
&&\text{for }i\ne j,\tag{STO27}\label{STO27}\\
&[\mathbf{g}_{ki}(a),\mathbf{f}_{il}(b)]=\mathbf{u}_{kl}(ab),
&&\text{for }k\ne l,\tag{STO28}\label{STO28}
\end{align}
where $a,b\in R$ are homogeneous, $1\leqslant i,j,i',j'\leqslant m$ and $1\leqslant k,l,k',l'\leqslant n$.
\end{definition}

Recall from Proposition~\ref{prop:ospalg} that the Lie superalgebra $\mathfrak{osp}_{m|2n}(R,{}^-)$ has a family of generators consisting of $t_{ij}(a)$, $u_{kl}(a)$, $v_{kl}(a)$, $w_{kl}(a)$, $f_{i'k'}(a)$ and $g_{k'i'}(a)$, where $a\in R$ is homogeneous, $1\leqslant i',i,j\leqslant m$ with $i\neq j$ and $1\leqslant k',k,l\leqslant n$ with $k\neq l$. These generators essentially satisfy all the relations (\ref{STO00})-(\ref{STO28}). Hence, there is a canonical epimorphism of Lie superalgebras
\begin{equation}
\psi:\mathfrak{sto}_{m|2n}(R,{}^-)\rightarrow\mathfrak{osp}_{m|2n}(R,{}^-)\label{eq:ce}
\end{equation}
such that
\begin{align*}
&\psi(\mathbf{t}_{ij}(a))=t_{ij}(a),
&&\psi(\mathbf{u}_{kl}(a))=u_{kl}(a),
&&\psi(\mathbf{v}_{kl}(a))=v_{kl}(a),\\
&\psi(\mathbf{w}_{kl}(a))=w_{kl}(a),
&&\psi(\mathbf{f}_{i'k'}(a))=f_{i'k'}(a),
&&\psi(\mathbf{g}_{k'i'}(a))=g_{k'i'}(a),
\end{align*}
for $a\in R$, $1\leqslant i',i,j\leqslant m$ with $i\neq j$ and $1\leqslant k', k, l\leqslant n$ with $k\neq l$.

Our primary aim is to show that (\ref{eq:ce}) is a central extension if $(m,n)\neq(1,1)$ and to further describe the universal central extension of $\mathfrak{osp}_{m|2n}(R,{}^-)$ for $m,n\in\mathbb{N}$. In this section, we first prove the perfectness of $\mathfrak{sto}_{m|2n}(R,{}^-)$ and give a decomposition of $\mathfrak{sto}_{m|2n}(R,{}^-)$.

\begin{proposition}
\label{prop:sto_pft}
The Lie superalgebra $\mathfrak{sto}_{m|2n}(R,{}^-)$ is perfect for $(m,n)\neq(1,1)$.
\end{proposition}
\begin{proof}
The Lie superalgebra $\mathfrak{sto}_{m|2n}(R,{}^-)$ is generated by $\mathbf{t}_{ij}(a)$, $\mathbf{u}_{kl}(a)$, $\mathbf{v}_{kl}(a)$, $\mathbf{w}_{kl}(a)$, $\mathbf{f}_{i'k'}(a)$ and $\mathbf{g}_{k'i'}(a)$, where $a\in R$, $1\leqslant i',i,j\leqslant m$ with $i\neq j$ and $1\leqslant k',k,l\leqslant n$ with $k\neq l$. By (\ref{STO27}), (\ref{STO28}), (\ref{STO25}), (\ref{STO23}), the generators $\mathbf{t}_{ij}(a)$, $\mathbf{u}_{kl}(a)$, $\mathbf{v}_{kl}(a)$ and $\mathbf{w}_{kl}(a)$ are contained in the derived Lie superalgebra $[\mathfrak{sto}_{m|2n}(R,{}^-),\mathfrak{sto}_{m|2n}(R,{}^-)]$.

Since $(m,n)\neq(1,1)$, we have $m\geqslant2$ or $n\geqslant2$. Then (\ref{STO15}) and (\ref{STO16}) imply that $\mathbf{f}_{i'k'}(a)$ and $\mathbf{g}_{k'i'}(a)$ are contained in the derived Lie superalgebra $[\mathfrak{sto}_{m|2n}(R,{}^-),\mathfrak{sto}_{m|2n}(R,{}^-)]$ if $m\geqslant2$, while (\ref{STO17}) and (\ref{STO18}) yield that $\mathbf{f}_{i'k'}(a)$ and $\mathbf{g}_{k'i'}(a)$ are contained in the derived Lie superalgebra $[\mathfrak{sto}_{m|2n}(R,{}^-),\mathfrak{sto}_{m|2n}(R,{}^-)]$ if $n\geqslant2$. Hence, the Lie superalgebra $\mathfrak{sto}_{m|2n}(R,{}^-)$ is perfect.
\end{proof}
\medskip

The generalized orthosymplectic Lie superalgebra $\mathfrak{osp}_{m|2n}(R,{}^-)$ is a Lie sub-superalgebra of $\mathfrak{gl}_{m|2n}(R)$. Hence, it is decomposed as a direct sum of $\Bbbk$-modules:
$$\mathfrak{osp}_{m|2n}(R,{}^-)=\mathfrak{osp}_{m|2n}^0(R,{}^-)\oplus\mathfrak{osp}_{m|2n}^1(R,{}^-),$$
where $\mathfrak{osp}_{m|2n}^0(R,{}^-)$ consists of all diagonal matrices in $\mathfrak{osp}_{m|2n}(R,{}^-)$ and $\mathfrak{osp}_{m|2n}^1(R,{}^-)$ consists of all matrices in $\mathfrak{osp}_{m|2n}(R,{}^-)$ with $0$ diagonal. Then $\mathfrak{osp}_{m|2n}^0(R,{}^-)$ is a Lie sub-supealgebra of $\mathfrak{osp}_{m|2n}(R,{}^-)$. Although $\mathfrak{osp}_{m|2n}^1(R,{}^-)$ is not a Lie sub-superalgebra of $\mathfrak{osp}_{m|2n}(R,{}^-)$, we have
$$[\mathfrak{osp}_{m|2n}^0(R,{}^-),\mathfrak{osp}_{m|2n}^1(R,{}^-)]\subseteq\mathfrak{osp}_{m|2n}^1(R,{}^-).$$
In order to figure out a similar decomposition for $\mathfrak{sto}_{m|2n}(R,{}^-)$, we denote the following elements of $\mathfrak{sto}_{m|2n}(R,{}^-)$ for $m,n\in\mathbb{N}$:
\begin{align*}
\mathbf{h}_{ik}(a,b):&=[\mathbf{f}_{ik}(a),\mathbf{g}_{ki}(b)],\\
\mathbf{v}_{k}(a):&=-[\mathbf{g}_{k1}(a),\mathbf{g}_{k1}(1)],\\
\mathbf{w}_{k}(a):&=[\mathbf{f}_{1k}(1),\mathbf{f}_{1k}(a)],
\end{align*}
where $a,b\in R$, $1\leqslant i\leqslant m$ and $1\leqslant k\leqslant n$. Using the relations (\ref{STO01})-(\ref{STO28}), we may deduce the following Lemmas~\ref{lem:sto_kp},~\ref{lem:sto_lmd34} and~\ref{lem:sto_lmd_kp} through direct computation.

\begin{lemma}
\label{lem:sto_kp}
The following equalities hold in the Lie superalgebra $\mathfrak{sto}_{m|2n}(R,{}^-)$ for $m,n\in\mathbb{N}$:
\begin{align*}
[\mathbf{h}_{ik}(a,b),\mathbf{t}_{ij}(c)]
&=\mathbf{t}_{ij}(abc-\overline{ab}c),
&&\text{for }i\neq j,\\
[\mathbf{h}_{ik}(a,b),\mathbf{t}_{jj'}(c)]&=0,
&&\text{for distinct }i,j,j',\\
[\mathbf{h}_{ik}(a,b),\mathbf{u}_{kl}(c)]
&=-(-1)^{(|a|+1)(|b|+1)}\mathbf{u}_{kl}(bac),
&&\text{for }k\neq l,\\
[\mathbf{h}_{ik}(a,b),\mathbf{u}_{lk}(c)]
&=-(-1)^{(|a|+1)(|b|+1)+|b||c|+|c||a|}\mathbf{u}_{lk}(cba),
&&\text{for }k\neq l,\\
[\mathbf{h}_{ik}(a,b),\mathbf{u}_{ll'}(c)]&=0,
&&\text{for distinct }k,l,l',\\
[\mathbf{h}_{ik}(a,b),\mathbf{v}_{kl}(c)]&=-(-1)^{(|a|+1)(|b|+1)}\mathbf{v}_{ll'}(bac),
&&\text{for }k\neq l,\\
[\mathbf{h}_{ik}(a,b),\mathbf{v}_{ll'}(c)]&=0,
&&\text{for distinct }k,l,l',\\
[\mathbf{h}_{ik}(a,b),\mathbf{w}_{kl}(c)]
&=-(-1)^{(|a|+1)(|b|+1)}\mathbf{w}_{kl}(\overline{ba}c),
&&\text{for }k\neq l,\\
[\mathbf{h}_{ik}(a,b),\mathbf{w}_{ll'}(c)]&=0,
&&\text{for distinct }k,l,l',\\
[\mathbf{h}_{ik}(a,b),\mathbf{f}_{il}(c)]
&=\mathbf{f}_{il}(abc-\overline{ab}c),
&&\text{for }k\neq l,\\
[\mathbf{h}_{ik}(a,b),\mathbf{f}_{jk}(c)]
&=-(-1)^{|a||b|+|b||c|+|c||a|}\mathbf{f}_{jk}(cba),
&&\text{for }i\neq j,\\
[\mathbf{h}_{ik}(a,b),\mathbf{f}_{jl}(c)]&=0,
&&\text{for }i\neq j\text{ and }k\neq l,\\
[\mathbf{g}_{li}(a),\mathbf{h}_{ik}(b,c)]
&=\mathbf{g}_{li}(abc-a\overline{bc}),
&&\text{for }k\neq l,\\
[\mathbf{g}_{kj}(a),\mathbf{h}_{ik}(b,c)]
&=-(-1)^{|a||b|+|b||c|+|c||a|}\mathbf{g}_{kj}(cba),
&&\text{for }i\neq j,\\
[\mathbf{g}_{lj}(a),\mathbf{h}_{ik}(b,c)]&=0,
&&\text{for }i\neq j\text{ and }k\neq l,
\end{align*}
where $a,b,c\in R$ are homogeneous, $1\leqslant i,j,j'\leqslant m$ and $1\leqslant k,l,l'\leqslant n$.

If in addition $(m,n)\neq(1,1)$, we also have
\begin{align*}
[\mathbf{h}_{ik}(a,b),\mathbf{f}_{ik}(c)]&=\mathbf{f}_{ik}(abc-\overline{ab}c-(-1)^{|a||b|+|b||c|+|c||a|}cba),\\
[\mathbf{g}_{ki}(c),\mathbf{h}_{ik}(a,b)]&=\mathbf{g}_{ki}(cab-c\overline{ab}-(-1)^{|a||b|+|b||c|+|c||a|}bac),
\end{align*}
for homogeneous $a,b,c\in R$, $1\leqslant i\leqslant m$ and $1\leqslant k\leqslant n$.\qed
\end{lemma}

\begin{lemma}
\label{lem:sto_lmd34}
In the Lie superalgebra $\mathfrak{sto}_{m|2n}(R,{}^-)$ with $(m,n)\neq(1,1)$,  we have
\begin{equation*}
\mathbf{v}_k(a)=\mathbf{v}_k(\bar{a}),\text{ and }
\mathbf{w}_k(a)=\mathbf{w}_k(\bar{a}),\\
\end{equation*}
for $a\in R$ and $1\leqslant k\leqslant n$. Moreover, the following equalities hold:
\begin{align*}
[\mathbf{g}_{ki}(a),\mathbf{g}_{ki}(b)]&=-(-1)^{|b|}\mathbf{v}_{k}(a\bar{b}),
&[\mathbf{f}_{ik}(a),\mathbf{f}_{ik}(b)]&=(-1)^{|a|}\mathbf{w}_{k}(\bar{a}b),\\
[\mathbf{v}_{k}(a),\mathbf{t}_{ij}(b)]&=0,
&[\mathbf{w}_{k}(a),\mathbf{t}_{ij}(b)]&=0,\\
[\mathbf{u}_{kl}(a),\mathbf{v}_{lk}(b)]&=\mathbf{v}_{k}(ab),
&[\mathbf{w}_{kl}(a),\mathbf{u}_{lk}(b)]&=\mathbf{w}_{k}(ab),\\
[\mathbf{v}_{k}(a),\mathbf{w}_{ll'}(b)]&=\delta_{kl'}\mathbf{u}_{kl}(a_{+}\bar{b}),
&[\mathbf{v}_{kl}(a),\mathbf{w}_{l'}(b)]&=\delta_{kl'}\mathbf{u}_{kl}(\bar{a}b_+),\\
[\mathbf{v}_{k}(a),\mathbf{f}_{il}(b)]&=(-1)^{|b|}\delta_{kl}\mathbf{g}_{ki}(a_{+}\bar{b}),
&[\mathbf{v}_{k}(a),\mathbf{g}_{li}(b)]&=0,\\
[\mathbf{w}_{k}(a),\mathbf{f}_{il}(b)]&=0,
&[\mathbf{g}_{li}(a),\mathbf{w}_{k}(b)]&=-(-1)^{|a|}\delta_{kl}\mathbf{f}_{ik}(\bar{a}b_+),
\end{align*}
where $a,b\in R$ are homogeneous, $1\leqslant i\leqslant m$ and $1\leqslant k,l,l'\leqslant n$.\qed
\end{lemma}

\begin{lemma}
\label{lem:sto_lmd_kp}
The following equalities hold in $\mathfrak{sto}_{m|2n}(R,{}^-)$ with $(m,n)\neq(1,1)$:
\begin{align*}
[\mathbf{t}_{ij}(a),\mathbf{t}_{ji}(b)]
&=\mathbf{h}_{ik}(a,b)-(-1)^{|r||s|}\mathbf{h}_{jk}(1,ba),
&&\text{for }1\leqslant i\ne j\leqslant m, 1\leqslant k\leqslant n,\\
[\mathbf{u}_{kl}(a),\mathbf{u}_{lk}(b)]
&=-(-1)^{|a|+|b|}(\mathbf{h}_{il}(a,b)-\mathbf{h}_{ik}(1,ab)),
&&\text{for }1\leqslant i\leqslant m, 1\leqslant k\ne l\leqslant n,\\
[\mathbf{v}_{kl}(a),\mathbf{w}_{lk}(b)]
&=(-1)^{(1+|a|)(1+|b|)}(\mathbf{h}_{ik}(b,a)+\mathbf{h}_{il}(1,\overline{ab})),
&&\text{for }1\leqslant i\leqslant m, 1\leqslant k\ne l\leqslant n,\\
[\mathbf{v}_{k}(a),\mathbf{w}_{l}(b)]
&=-(-1)^{(|a|+1)(|b|+1)}\delta_{kl}(\mathbf{h}_{ik}(b,\bar{a})+\mathbf{h}_{ik}(1,\overline{ba})),
&&\text{for }1\leqslant i\leqslant m, 1\leqslant k, l\leqslant n,
\end{align*}
where $a,b\in R$ are homogeneous.\qed
\end{lemma}

Based on the lemmas above, we obtain the following decomposition of $\mathfrak{sto}_{m|2n}(R,{}^-)$:
\begin{proposition}
\label{prop:sto_tridec}
Suppose $m,n\in\mathbb{N}$ satisfy $(m,n)\neq(1,1)$. The Lie superalgebra $\mathfrak{sto}_{m|2n}(R,{}^-)$ is decomposed as the following direct sum of $\Bbbk$-modules:
\begin{equation}
\mathfrak{sto}_{m|2n}(R,{}^-)=\mathfrak{sto}_{m|2n}^0(R,{}^-)\oplus\mathfrak{sto}_{m|2n}^1(R,{}^-),
\label{eq:sto_tridec}
\end{equation}
where $\mathfrak{sto}_{m|2n}^0(R,{}^-)=
\mathrm{span}_{\Bbbk}\left\{\mathbf{h}_{ik}(a,b)
\middle|a,b\in R, 1\leqslant i\leqslant m, 1\leqslant k\leqslant n\right\}$ and
$\mathfrak{sto}_{m|2n}^1(R,{}^-)$ is the $\Bbbk$-module spanned by $\mathbf{t}_{ij}(a)$, $\mathbf{u}_{kl}(a)$, $\mathbf{v}_{kl}(a)$, $\mathbf{v}_{k}(a)$, $\mathbf{w}_{kl}(a)$, $\mathbf{w}_{k}(a)$, $\mathbf{f}_{ik}(a)$ and $\mathbf{g}_{ki}(a)$ with $a\in R$, $1\leqslant i\neq j\leqslant m$ and $1\leqslant k\neq l\leqslant n$.
They satisfy $$[\mathfrak{sto}_{m|2n}^0(R,{}^-),\mathfrak{sto}_{m|2n}^1(R,{}^-)]\subseteq
\mathfrak{sto}_{m|2n}^1(R,{}^-).$$
\end{proposition}
\begin{proof}
From Lemma~\ref{lem:sto_kp}, we deduce that $\mathfrak{sto}_{m|2n}^0(R,{}^-)$ is a Lie sub-superalgebra of $\mathfrak{sto}_{m|2n}(R,{}^-)$ and
$$[\mathfrak{sto}_{m|2n}^0(R,{}^-),\mathfrak{sto}_{m|2n}^1(R,{}^-)]=\mathfrak{sto}_{m|2n}^1(R,{}^-).$$
Note that $\mathfrak{sto}_{m|2n}^1(R,{}^-)$ is not a Lie sub-superalgebra of $\mathfrak{sto}_{m|2n}(R,{}^-)$. Nonetheless, we may also deduce from Definition~\ref{def:sto}, Lemmas~\ref{lem:sto_lmd34} and~\ref{lem:sto_lmd_kp} that
$$[\mathfrak{sto}_{m|2n}^1(R,{}^-),\mathfrak{sto}_{m|2n}^1(R,{}^-)]\subseteq\mathfrak{sto}_{m|2n}^0(R,{}^-)+\mathfrak{sto}_{m|2n}^1(R,{}^-).$$
Since all generators of $\mathfrak{sto}_{m|2n}(R,{}^-)$ are contained in $\mathfrak{sto}_{m|2n}^1(R,{}^-)$, it follows that
\begin{equation}
\mathfrak{sto}_{m|2n}(R,{}^-)=\mathfrak{sto}_{m|2n}^0(R,{}^-)+\mathfrak{sto}_{m|2n}^1(R,{}^-).\label{eq:sto_tridec1}
\end{equation}

We next show that the above summation is a direct sum of $\Bbbk$-modules. Let $0=\boldsymbol{x}^0+\boldsymbol{x}^1$ be a decomposition of $0$ with respect to (\ref{eq:sto_tridec1}), then
$$0=\psi(0)=\psi(\boldsymbol{x}^0)+\psi(\boldsymbol{x}^1)$$
is a decomposition of $0$ in $\mathfrak{osp}_{m|2n}(R,{}^-)$, where $\psi(\boldsymbol{x}^0)$ is a diagonal matrix and $\psi(\boldsymbol{x}^1)$ is a matrix with $0$ diagonal. Hence, $\psi(\boldsymbol{x}^0)=\psi(\boldsymbol{x}^1)=0$. Observing that $\psi|_{\mathfrak{sto}_{m|2n}^1(R,{}^-)}$ is injective, we further deduce that $\boldsymbol{x}^1=0$. This shows the summation in (\ref{eq:sto_tridec1}) is a direct sum.
\end{proof}

To conclude this section, we consider the Lie superalgebra $\mathfrak{sto}_{m|2n}(R,{}^-)$ in the situation where the associative superalgebra $(R,{}^-)$ is specified to $(S\oplus S^{\mathrm{op}},\mathrm{ex})$ for an arbitrary unital associative superalgebra $S$.

\begin{proposition}
\label{prop:sto_SS}
Let $m,n\in\mathbb{N}$ such that $(m,n)\neq(1,1)$ and $S$ an arbitrary unital associative superalgebra. Then
$$\mathfrak{sto}_{m|2n}(S\oplus S^{\mathrm{op}},\mathrm{ex})\cong\mathfrak{st}_{m|2n}(S),$$
where $\mathfrak{st}_{m|2n}(S)$ is the Steinberg Lie superalgebra coordinatized by $S$ \cite{ChangWang2015} and \cite{ChenSun2015}.
\end{proposition}
\begin{proof}
The proof is similar to~\cite[Proposition 5.7]{ChangWang2015}. Recall that the Steinberg Lie superalgebra $\mathfrak{st}_{m|2n}(S)$ is the abstract Lie superalgebra generated by homogeneous elements $\boldsymbol{e}_{ij}(a)$ of degree $|i|+|j|+|a|$ for $a\in R$ and $1\leqslant i\neq j\leqslant m+2n$, subjecting to the relations:
\begin{align}
&a\mapsto\boldsymbol{e}_{ij}(a)\text{ is }\Bbbk\text{-linear},\tag{ST0}\label{ST0}\\
&[\boldsymbol{e}_{ij}(a),\boldsymbol{e}_{jk}(b)]=\boldsymbol{e}_{ik}(ab),
&&\text{ for distinct }i,j,k,\tag{ST1}\label{ST1}\\
&[\boldsymbol{e}_{ij}(a),\boldsymbol{e}_{kl}(b)]=0,
&&\text{ for }i\neq j\neq k\neq l\neq i,\tag{ST2}\label{ST2}
\end{align}
where $a,b\in R$ and $1\leqslant i,j,k,l\leqslant m+n$.

We may define a homomorphism $\phi:\mathfrak{st}_{m|2n}(S)\rightarrow\mathfrak{sto}_{m|2n}(S\oplus S^{\mathrm{op}},\mathrm{ex})$
such that:
\begin{align*}
&\phi(\boldsymbol{e}_{ij}(a))=\mathbf{t}_{ij}(a\oplus0), \\
&\phi(\boldsymbol{e}_{i',m+k'}(a))=\mathbf{f}_{i'k'}(a\oplus0),
&&\phi(\boldsymbol{e}_{m+n+k',i'}(a))=\mathbf{f}_{i'k'}(0\oplus \rho(a)),\\
&\phi(\boldsymbol{e}_{i',m+n+k'}(a))=-\mathbf{g}_{k'i'}(0\oplus \rho(a)),
&&\phi(\boldsymbol{e}_{m+k',i'}(a))=\mathbf{g}_{k'i'}(a\oplus0),\\
&\phi(\boldsymbol{e}_{m+k,m+l}(a))=\mathbf{u}_{kl}(a\oplus0),
&&\phi(\boldsymbol{e}_{m+n+k,m+n+l}(a))=-\mathbf{u}_{kl}(0\oplus a),\\
&\phi(\boldsymbol{e}_{m+k,m+n+l}(a))=\mathbf{v}_{kl}(a\oplus0),
&&\phi(\boldsymbol{e}_{m+n+k,m+l}(a))=\mathbf{w}_{kl}(a\oplus0),\\
&\phi(\boldsymbol{e}_{m+k',m+n+k'}(a))=\mathbf{v}_{k'}(a\oplus0),
&&\phi(\boldsymbol{e}_{m+n+k',m+k'}(a))=\mathbf{w}_{k'}(a\oplus0),
\end{align*}
where $a\in S$, $1\leqslant i', i, j\leqslant m$ with $i\neq j$ and $1\leqslant k', k, l\leqslant n$ with $k\neq l$. It is an isomorphism since it has an inverse $\phi^{-1}:\mathfrak{sto}_{m|2n}(S\oplus S^{\mathrm{op}},\mathrm{ex})\rightarrow\mathfrak{st}_{m|2n}(S)$ defined by
\begin{align*}
&\phi(\mathbf{t}_{ij}(a\oplus b))=\boldsymbol{e}_{ij}(a)-\boldsymbol{e}_{ji}(b),\\
&\phi(\mathbf{u}_{kl}(a\oplus b))=\boldsymbol{e}_{m+k,m+l}(a)-\boldsymbol{e}_{m+n+l,m+n+k}(b),\\
&\phi(\mathbf{v}_{kl}(a\oplus b))=\boldsymbol{e}_{m+k,m+n+l}(a)+\boldsymbol{e}_{m+l,m+n+k}(b),\\
&\phi(\mathbf{w}_{kl}(a\oplus b))=\boldsymbol{e}_{m+n+k,m+l}(a)+\boldsymbol{e}_{m+n+l,m+k}(b),\\
&\phi(\mathbf{f}_{i'k'}(a\oplus b))=\boldsymbol{e}_{i',m+k'}(a)+\boldsymbol{e}_{m+n+k',i'}(\rho(b)),\\
&\phi(\mathbf{g}_{k'i'}(a\oplus b))=\boldsymbol{e}_{m+k',i'}(a)-\boldsymbol{e}_{i',m+n+k'}(\rho(b)),
\end{align*}
for $a,b\in S$, $1\leqslant i', i, j\leqslant m$ with $i\neq j$ and $1\leqslant k', k, l\leqslant n$ with $k\neq l$.
\end{proof}

\section{Central extensions of $\mathfrak{osp}_{m|2n}(R,{}^-)$ with $(m,n)\neq(1,1)$}
\label{sec:cext}

Given a unital associative superalgebra with superinvolution $(R,{}^-)$, we have introduced the Steinberg orthosymplectic Lie superalgebra $\mathfrak{sto}_{m|2n}(R,{}^-)$ coordinatized by  $(R,{}^-)$ and obtained the canonical homomorphism $\psi:\mathfrak{sto}_{m|2n}(R,{}^-)\rightarrow\mathfrak{osp}_{m|2n}(R,{}^-)$. This section serves for proving that $\psi:\mathfrak{sto}_{m|2n}(R,{}^-)\rightarrow\mathfrak{osp}_{m|2n}(R,{}^-)$ is a central extension and describing the $\ker\psi$.

\begin{proposition}
\label{prop:osp_ce}
Suppose $(m,n)\neq(1,1)$, then the canonical homomorphism
$$\psi:\mathfrak{sto}_{m|2n}(R,{}^-)\rightarrow\mathfrak{osp}_{m|2n}(R,{}^-)$$
is a central extension and $\ker\psi\subseteq\mathfrak{sto}_{m|2n}^0(R,{}^-)$, where $\mathfrak{sto}_{m|2n}^0(R,{}^-)$ is the Lie sub-superalgebra of $\mathfrak{sto}_{m|2n}(R,{}^-)$ as defined in Proposition~\ref{prop:sto_tridec}.
\end{proposition}
\begin{proof}
Let $\boldsymbol{x}\in\ker\psi$. We will show that $\boldsymbol{x}\in\mathfrak{sto}_{m|2n}^0(R,{}^-)$ and then prove that $\boldsymbol{x}$ is contained in the center of $\mathfrak{sto}_{m|2n}(R,{}^-)$.

Recall that $\boldsymbol{x}$ is decomposed as $\boldsymbol{x}=\boldsymbol{x}^0+\boldsymbol{x}^1$ with respect to the decomposition (\ref{eq:sto_tridec}). Hence, $\boldsymbol{x}\in\ker\psi$ implies that
$$0=\psi(\boldsymbol{x}^0)+\psi(\boldsymbol{x}^1)$$
in $\mathfrak{osp}_{m|2n}(R,{}^-)$. Since $\psi(\boldsymbol{x}^0)$ is a diagonal matrix and $\psi(\boldsymbol{x}^1)$ is a matrix with 0 diagonal, it follows that $\psi(\boldsymbol{x}^0)=\psi(\boldsymbol{x}^1)=0$. We have known from Proposition~\ref{prop:sto_tridec} that $\psi|_{\mathfrak{sto}_{m|2n}^1(R,{}^-)}$ is injective. It follows that $\boldsymbol{x}^1=0$, and hence $$\boldsymbol{x}=\boldsymbol{x}^0\in\mathfrak{sto}_{m|2n}^0(R,{}^-).$$

Next, we will show that $\boldsymbol{x}$ is contained in the center of $\mathfrak{sto}_{m|2n}(R,{}^-)$. It suffices to show that $\mathrm{ad}\boldsymbol{x}$ annihilates a complete family of generators of $\mathfrak{sto}_{m|2n}(R,{}^-)$, i.e.,
\begin{equation}
[\boldsymbol{x},\mathbf{t}_{ij}(a)],\quad
[\boldsymbol{x},\mathbf{u}_{kl}(a)], \quad
[\boldsymbol{x},\mathbf{v}_{kl}(a)], \quad
[\boldsymbol{x},\mathbf{w}_{kl}(a)], \quad
[\boldsymbol{x},\mathbf{f}_{i'k'}(a)], \quad
[\boldsymbol{x},\mathbf{g}_{k'i'}(a)]\label{eq:sto_cth}
\end{equation}
are zero for all $a\in R$, $1\leqslant i',i,j\leqslant m$ with $i\neq j$ and $1\leqslant k',k,l\leqslant n$ with $k\neq l$. Since
$$[\mathfrak{sto}_{m|2n}^0(R,{}^-),\mathfrak{sto}_{m|2n}^1(R,{}^-)]\subseteq\mathfrak{sto}_{m|2n}^1(R,{}^-)$$
and $\boldsymbol{x}\in\mathfrak{sto}_{m|2n}^0(R,{}^-)$, we deduce that every element in (\ref{eq:sto_cth}) is contained in $\mathfrak{sto}_{m|2n}^1(R,{}^-)$. On the other hand, all elements in (\ref{eq:sto_cth}) are killed by $\psi$ since $\psi(\boldsymbol{x})=0$. Therefore, we conclude that all elements in (\ref{eq:sto_cth}) are zero since $\psi|_{\mathfrak{sto}_{m|2n}^1(R,{}^-)}$ is injective. It follows that $\boldsymbol{x}$ is contained in the center of $\mathfrak{sto}_{m|2n}(R,{}^-)$ and thus $\psi:\mathfrak{sto}_{m|2n}(R,{}^-)\rightarrow\mathfrak{osp}_{m|2n}(R,{}^-)$ is a central extension.
\end{proof}

The remaining part of this section is devoted to explicitly describe $\ker\psi$. We will show that $\ker\psi$ can be identified with the first $\mathbb{Z}/2\mathbb{Z}$-graded skew-dihedral homology $\fourIdx{}{-}{}{1}{\mathrm{HD}}(R,{}^-)$. Recall from~\cite[Proposition 6.3]{ChangWang2015} that
\begin{equation}
\fourIdx{}{-}{}{1}{\mathrm{HD}}(R,{}^-):=\left\{\textstyle{\sum}_i \langle a_i,b_i\rangle_{\mathsf{d}}^-\in\langle R,R\rangle_{\mathsf{d}}^-\middle|\textstyle{\sum}_i[a_i,b_i]=\textstyle{\sum}_i\overline{[a_i,b_i]}\right\},
\end{equation}
where $\langle R,R\rangle_{\mathsf{d}}^-=(R\otimes_{\Bbbk}R)/I_{\mathsf{d}}^-$ and $I_{\mathsf{d}}^-$ is the $\Bbbk$--submodule of $R\otimes_{\Bbbk}R$ spanned by $a\otimes b-\bar{a}\otimes\bar{b}$, $a\otimes b+(-1)^{|a||b|}b\otimes a$ and $(-1)^{|a||c|}ab\otimes c+(-1)^{|a||b|}bc\otimes a+(-1)^{|b||c|}ca\otimes b$ for homogeneous $a,b,c\in R$. Then we have:

\begin{proposition}
\label{prop:sto_ker}
Let $\psi:\mathfrak{sto}_{m|2n}(R,{}^-)\rightarrow\mathfrak{osp}_{m|2n}(R,{}^-)$ be the canonical homomorphism with $(m,n)\neq(1,1)$. Then $\ker\psi\cong\fourIdx{}{-}{}{1}{\mathrm{HD}}(R,{}^-)$ as $\Bbbk$-modules.
\end{proposition}

In order to prove this proposition, we need a fine characterization for elements of the subalgebra $\mathfrak{sto}_{m|2n}^0(R,{}^-)\subseteq\mathfrak{sto}_{m|2n}(R,^-)$. We always assume $(m,n)\neq(1,1)$ in the rest of this section. Let
$$\boldsymbol{\lambda}(a,b):=\mathbf{h}_{11}(a,b)-(-1)^{|a||b|}\mathbf{h}_{11}(1,ba),$$
for homogeneous $a,b\in R$. Then we have:
\begin{lemma}
\label{lem:sto_h_ele}
Every element $\boldsymbol{x}$ of $\mathfrak{sto}_{m|2n}^0(R,{}^-)$ can be written as
\begin{equation}
\boldsymbol{x}=\sum\limits_{i\in I_{\boldsymbol{x}}}\boldsymbol{\lambda}(a_i,b_i)
+\sum\limits_{j=1}^m\mathbf{h}_{j1}(1,c_j)+\sum\limits_{k=2}^n\mathbf{h}_{1k}(1,d_k),
\label{eq:sto_h_ele}
\end{equation}
where $I_{\boldsymbol{x}}$ is a finite index set, $a_i, b_i, c_j, d_k\in R$ for $i\in I_{\boldsymbol{x}}$, $1\leqslant j\leqslant m$ and $2\leqslant k\leqslant n$.
\end{lemma}
\begin{proof}
Recall from Proposition~\ref{prop:sto_tridec} that $\mathfrak{osp}_{m|2n}^0(R,{}^-)$ is spanned by $\mathbf{h}_{ik}(a,b)$ for homogeneous $a,b\in R,$ $1\leqslant i\leqslant m$ and $1\leqslant k\leqslant n$. It suffices to show that $\mathbf{h}_{ik}(a,b)$ is of the form (\ref{eq:sto_h_ele}).

We first claim that
\begin{equation}
\boldsymbol{\lambda}(a,b)=\mathbf{h}_{1k}(a,b)-(-1)^{|a||b|}\mathbf{h}_{1k}(1,ba),\label{eq:sto_kp_1}
\end{equation}
for homogeneous $a,b\in R$ and $2\leqslant k\leqslant n$. It follows from the following computation:
\begin{align*}
\mathbf{h}_{1k}(a,b)=&[\mathbf{f}_{1k}(a),\mathbf{g}_{k1}(b)]\\
=&[[\mathbf{f}_{11}(a),\mathbf{u}_{1k}(1)],\mathbf{g}_{k1}(b)]\\
=&[\mathbf{f}_{11}(a),[\mathbf{u}_{1k}(1),\mathbf{g}_{k1}(b)]]
-[\mathbf{u}_{1k}(1),[\mathbf{f}_{11}(a),\mathbf{g}_{k1}(b)]]\\
=&[\mathbf{f}_{11}(a),\mathbf{g}_{11}(b)]
+(-1)^{(|a|+1)(|b|+1)}[\mathbf{u}_{1k}(1),\mathbf{u}_{k1}(ba)]\\
=&\mathbf{h}_{11}(a,b)+(-1)^{|a||b|}(\mathbf{h}_{1k}(1,ba)-\mathbf{h}_{11}(1,ba))\\
=&\boldsymbol{\lambda}(a,b)+(-1)^{|a||b|}\mathbf{h}_{1k}(1,ba),
\end{align*}
where the second last equality follows from Lemma~\ref{lem:sto_lmd_kp}. Now, the claim follows and hence
$$\mathbf{h}_{1k}(a,b)=\boldsymbol{\lambda}(a,b)+(-1)^{|a||b|}\mathbf{h}_{1k}(1,ba)$$
is of the form (\ref{eq:sto_h_ele}) for $1\leqslant k\leqslant n$.

For $i\geqslant2$, we also deduce from Lemma~\ref{lem:sto_lmd_kp} that
\begin{align*}
\mathbf{h}_{ik}(a,b)
=&[\mathbf{f}_{ik}(a),\mathbf{g}_{ki}(b)]\\
=&[[\mathbf{t}_{i1}(1),\mathbf{f}_{1k}(a)],\mathbf{g}_{ki}(b)]\\
=&[\mathbf{t}_{i1}(1),[\mathbf{f}_{1k}(a),\mathbf{g}_{ki}(b)]]
+[\mathbf{f}_{1k}(a),\mathbf{g}_{k1}(b)]\\
=&[\mathbf{t}_{i1}(1),\mathbf{t}_{1i}(ab)]+\mathbf{h}_{1k}(a,b)\\
=&\boldsymbol{\lambda}(a,b)+(-1)^{|a||b|}\mathbf{h}_{1k}(1,ba)+\mathbf{h}_{i1}(1,ab)-\mathbf{h}_{11}(1,ab)
\end{align*}
is of the form (\ref{eq:sto_h_ele}). The lemma follows.
\end{proof}

Using the lemma above, we may obtain a characterization of $\ker\psi$:

\begin{lemma}
\label{lem:sto_ker}
$\ker\psi=\left\{\sum_i\boldsymbol{\lambda}(a_i,b_i)\middle|\sum_i[a_i,b_i]=\sum_i\overline{[a_i,b_i]}\right\}$.
\end{lemma}
\begin{proof}
It is easy to verify that $\ker\psi\supseteq\left\{\sum_i\boldsymbol{\lambda}(a_i,b_i)\middle|\sum_i[a_i,b_i]=\sum_i\overline{[a_i,b_i]}\right\}$, we show the converse inclusion.
Let $\boldsymbol{x}\in\ker{\psi}$. We have known from Proposition~\ref{prop:osp_ce} that $\ker\psi\subseteq\mathfrak{sto}_{m|2n}^0(R,{}^-)$. It follows from Lemma~\ref{lem:sto_h_ele} that
$$\boldsymbol{x}=\sum\limits_{i\in I_{\boldsymbol{x}}}\boldsymbol{\lambda}(a_i,b_i)+\sum\limits_{j=1}^m\mathbf{h}_{j1}(1,c_j)+\sum\limits_{k=2}^m\mathbf{h}_{1k}(1,d_k),$$
where $I_{\boldsymbol{x}}$ is a finite index set, $a_i, b_i, c_j, d_k\in R$ for $i\in I_x$, $1\leqslant j\leqslant m$ and $2\leqslant k\leqslant n$. We deduce from $\boldsymbol{x}\in\ker\psi$ that
\begin{align*}
0=\psi(\boldsymbol{x})=&\sum_{i\in I_{\boldsymbol{x}}}e_{11}([a_i,b_i]-\overline{[a_i,b_i]})
+\sum\limits_{j=1}^me_{jj}(c_j-\bar{c}_j)+\sum\limits_{j=1}^mu_{11}(\rho(c_j))\\
&+\sum\limits_{k=2}^ne_{11}(d_k-\bar{d}_k)+\sum\limits_{k=2}^n(-1)^{|d_k|}u_{kk}(d_k),
\end{align*}
which yields that $\bar{c}_j=c_j$ for $j=2,\ldots,m$ and $d_k=0$ for $k=2,\ldots,n$,
$$\sum_{i\in I_x}([a_i,b_i]-\overline{[a_i,b_i]})+c_1-\bar{c}_1=0,\text{ and }
\sum_{j=1}^m\rho(c_j)=0.$$

Now, we claim that $\mathbf{h}_{i1}(1,c)=\mathbf{h}_{11}(1,c)$ if $c=\bar{c}\in R$ and $2\leqslant i\leqslant m$. Indeed, we deduce from Lemma~\ref{lem:sto_lmd_kp} that
$$[\mathbf{t}_{i1}(1),\mathbf{t}_{1i}(c)]=\mathbf{h}_{i1}(1,c)-\mathbf{h}_{11}(1,c).$$
Since $c=\bar{c}$, we have $\mathbf{t}_{1i}(c)=-\mathbf{t}_{i1}(\bar{c})=-\mathbf{t}_{i1}(c)$. Hence, Lemma~\ref{lem:sto_lmd_kp} also implies that
$$[\mathbf{t}_{i1}(1),\mathbf{t}_{1i}(c)]=[\mathbf{t}_{1i}(1),\mathbf{t}_{i1}(c)]
=\mathbf{h}_{11}(1,c)-\mathbf{h}_{i1}(1,c),$$
which yields that $\mathbf{h}_{11}(1,c)-\mathbf{h}_{i1}(1,c)=\mathbf{h}_{i1}(1,c)-\mathbf{h}_{11}(1,c)$. The claim follows since $\frac{1}{2}\in\Bbbk$.

Return to the proof of the lemma. Since $c_j=\bar{c}_j$ for $j=2,\ldots,m$, we deduce that
\begin{align*}
\sum\limits_{j=1}^m\mathbf{h}_{j1}(1,c_j)
=\mathbf{h}_{11}(1,c_1)+\sum\limits_{j=2}^m\mathbf{h}_{11}(1,c_j)
=\mathbf{h}_{11}\left(1,\sum\limits_{j=1}^mc_j\right)
=0,
\end{align*}
where $\sum_{j=1}^mc_j=0$ since $\sum_{j=1}^m\rho(c_j)=0$. Combining the fact that $d_k=0$ for $k=2,\ldots,n$, we have
$$\boldsymbol{x}=\sum\limits_{i\in I_{\boldsymbol{x}}}\boldsymbol{\lambda}(a_i,b_i),$$
which also satisfy $\sum\limits_{i\in I_{\boldsymbol{x}}}[a_i,b_i]=\sum\limits_{i\in I_{\boldsymbol{x}}}\overline{[a_i,b_i]}$ since
$\psi(\boldsymbol{x})=0$. This shows the converse inclusion and completes the proof.
\end{proof}

In order to identify $\ker\psi$ with $\fourIdx{}{-}{}{1}{\mathrm{HD}}(R,{}^-)$, we need more properties of $\boldsymbol{\lambda}(a,b)$:

\begin{lemma}
\label{lem:sto_h_rln}
The following equalities hold in $\mathfrak{sto}_{m|2n}(R,{}^-)$ with $(m,n)\neq(1,1)$:
\begin{enumerate}
\item $\boldsymbol{\lambda}(a,b)=\boldsymbol{\lambda}(\bar{a},\bar{b})$,
\item $\boldsymbol{\lambda}(a,b)=-(-1)^{|a||b|}\boldsymbol{\lambda}(b,a)$,
\item $(-1)^{|a||c|}\boldsymbol{\lambda}(ab,c)+(-1)^{|b||a|}\boldsymbol{\lambda}(bc,a)+(-1)^{|c||b|}\boldsymbol{\lambda}(ca,b)=0$,
\end{enumerate}
where $a,b,c\in R$ are homogeneous.
\end{lemma}
\begin{proof}
We first consider the situation where $n\geqslant 2$:

(i) It follows from Lemma~\ref{lem:sto_lmd_kp} that
\begin{align*}
[\mathbf{v}_{12}(b),\mathbf{w}_{21}(a)]
&=(-1)^{(1+|a|)(1+|b|)}(\mathbf{h}_{11}(a,b)+\mathbf{h}_{12}(1,\overline{ab})),\\
[\mathbf{v}_{21}(\bar{b}),\mathbf{w}_{12}(\bar{a})]
&=(-1)^{(1+|a|)(1+|b|)}(\mathbf{h}_{12}(\bar{a},\bar{b})+\mathbf{h}_{11}(1,\overline{\bar{a}\bar{b}})).
\end{align*}
Note that $\mathbf{w}_{12}(\bar{a})=\mathbf{w}_{21}(a)$ and $\mathbf{v}_{12}(b)=\mathbf{v}_{21}(\bar{b})$, we obtained that
$$\boldsymbol{\lambda}(a,b)=\mathbf{h}_{11}(a,b)-(-1)^{|a||b|}\mathbf{h}_{11}(1,ba)
=\mathbf{h}_{12}(\bar{a},\bar{b})-(-1)^{|a||b|}\mathbf{h}_{12}(1,ba)
=\boldsymbol{\lambda}(\bar{a},\bar{b}).$$

For (ii) and (iii), we deduce from (\ref{eq:sto_kp_1}) that
\begin{align*}
\boldsymbol{\lambda}(ab,c)
&=\mathbf{h}_{12}(ab,c)-(-1)^{(|a|+|b|)|c|}\mathbf{h}_{12}(1,cab)\\
&=[\mathbf{f}_{12}(ab),\mathbf{g}_{21}(c)]
-(-1)^{(|a|+|b|)|c|}[\mathbf{f}_{12}(1),\mathbf{g}_{21}(cab)]\\
&=[[\mathbf{f}_{11}(a),\mathbf{u}_{12}(b)],\mathbf{g}_{21}(c)]
-(-1)^{(|a|+|b|)|c|}[\mathbf{f}_{12}(1),[\mathbf{u}_{21}(ca),\mathbf{g}_{11}(b)]]\\
&=[\mathbf{f}_{11}(a),[\mathbf{u}_{12}(b),\mathbf{g}_{21}(c)]]
-(-1)^{|b|(1+|a|)}[\mathbf{u}_{12}(b),[\mathbf{f}_{11}(a),\mathbf{g}_{21}(c)]]\\
&\quad-(-1)^{(|a|+|b|)|c|}[[\mathbf{f}_{12}(1),\mathbf{u}_{21}(ca)],\mathbf{g}_{11}(b)]\\
&\quad+(-1)^{|a|+|c|+|a||b|+|a||c|}[[\mathbf{f}_{12}(1),\mathbf{g}_{11}(b)],\mathbf{u}_{21}(ca)]\\
&=[\mathbf{f}_{11}(a),\mathbf{g}_{11}(bc)]+(-1)^{(1+|b|+|c|)(1+|a|)}[\mathbf{u}_{12}(b),\mathbf{u}_{21}(ca)]\\
&\quad-(-1)^{(|a|+|b|)|c|}[\mathbf{f}_{11}(ca),\mathbf{g}_{11}(b)]
+(-1)^{|a|+|b|+|c|+|a||b|+|a||c|}[\mathbf{u}_{12}(b),\mathbf{u}_{21}(ca)]\\
&=[\mathbf{f}_{11}(a),\mathbf{g}_{11}(bc)]-(-1)^{(|a|+|b|)|c|}[\mathbf{f}_{11}(ca),\mathbf{g}_{11}(b)]\\
&=\mathbf{h}_{11}(a,bc)-(-1)^{(|a|+|b|)|c|}\mathbf{h}_{11}(ca,b)\\
&=\mathbf{h}_{11}(a,bc)-(-1)^{|a|(|b|+|c|)}\mathbf{h}_{11}(1,bca)\\
&\quad+(-1)^{|a|(|b|+|c|)}\mathbf{h}_{11}(1,bca)
-(-1)^{(|a|+|b|)|c|}\mathbf{h}_{11}(ca,b)\\
&=\boldsymbol{\lambda}(a,bc)-(-1)^{(|a|+|b|)|c|}\boldsymbol{\lambda}(ca,b).
\end{align*}
We observe that $\boldsymbol{\lambda}(1,bc)=0$ from the definition of $\boldsymbol{\lambda}(a,b)$. Hence, (ii) follows from the above equality by setting $a=1$. Applying (ii) to the equality above, we obtain (iii).
\medskip

Now, we assume that $m\geqslant2$ and $n=1$ and show that (i), (ii) and (iii) also hold in this situation. We claim that
\begin{equation}
\boldsymbol{\lambda}(ab,c)=[\mathbf{t}_{12}(a),\mathbf{t}_{21}(bc)]
-(-1)^{|c|(|a|+|b|)}[\mathbf{t}_{12}(ca),\mathbf{t}_{21}(b)]
\label{eq:sto_h_rlns1}
\end{equation}
for homogeneous $a,b,c\in R$. Indeed, this can be proved as follows:
\begin{align*}
\boldsymbol{\lambda}(ab,c)
&=[\mathbf{f}_{11}(ab),\mathbf{g}_{11}(c)]
-(-1)^{(|a|+|b|)|c|}[\mathbf{f}_{11}(1),\mathbf{g}_{11}(cab)]\\
&=[[\mathbf{t}_{12}(a),\mathbf{f}_{21}(b)],\mathbf{g}_{11}(c)]
-(-1)^{(|a|+|b|)|c|}[\mathbf{f}_{11}(1),[\mathbf{g}_{12}(ca),\mathbf{t}_{21}(b)]]\\
&=[\mathbf{t}_{12}(a),[\mathbf{f}_{21}(b),\mathbf{g}_{11}(c)]]
-(-1)^{|a|(1+|b|)}[\mathbf{f}_{21}(b),[\mathbf{t}_{12}(a),\mathbf{g}_{11}(c)]]\\
&\quad-(-1)^{(|a|+|b|)|c|}[[\mathbf{f}_{11}(1),\mathbf{g}_{12}(ca)],\mathbf{t}_{21}(b)]\\
&\quad+(-1)^{(|a|+|b|)|c|+|a|+|c|}[\mathbf{g}_{12}(ca),[\mathbf{f}_{11}(1),\mathbf{t}_{21}(b)]]\\
&=[\mathbf{t}_{12}(a),\mathbf{t}_{21}(bc)]
+(-1)^{|a|(|b|+|c|)}[\mathbf{f}_{21}(b),\mathbf{g}_{12}(ca)]\\
&\quad-(-1)^{(|a|+|b|)|c|}[\mathbf{t}_{12}(ca),\mathbf{t}_{21}(b)]
-(-1)^{(|a|+|b|)|c|+|a|+|b|+|c|}[\mathbf{g}_{12}(ca),\mathbf{f}_{21}(b)]\\
&=[\mathbf{t}_{12}(a),\mathbf{t}_{21}(bc)]
-(-1)^{|c|(|a|+|b|)}[\mathbf{t}_{12}(ca),\mathbf{t}_{21}(b)].
\end{align*}

(i) Setting $a=1$ or $b=1$ in (\ref{eq:sto_h_rlns1}), we obtain
\begin{align*}
\boldsymbol{\lambda}(b,c)&=[\mathbf{t}_{12}(1),\mathbf{t}_{21}(bc)]
-(-1)^{|c||b|}[\mathbf{t}_{12}(c),\mathbf{t}_{21}(b)],\\
\boldsymbol{\lambda}(a,c)&=[\mathbf{t}_{12}(a),\mathbf{t}_{21}(c)]
-(-1)^{|c||a|}[\mathbf{t}_{12}(ca),\mathbf{t}_{21}(1)].
\end{align*}
Hence, we deduce that
\begin{align*}
\boldsymbol{\lambda}(\bar{a},\bar{b})
&=[\mathbf{t}_{12}(\bar{a}),\mathbf{t}_{21}(\bar{b})]
-(-1)^{|a||b|}[\mathbf{t}_{12}(\bar{b}\bar{a}),\mathbf{t}_{21}(1)]\\
&=-(-1)^{|a||b|}[\mathbf{t}_{12}(b),\mathbf{t}_{21}(a)]+[\mathbf{t}_{12}(1),\mathbf{t}_{21}(ab)]
=\boldsymbol{\lambda}(a,b)
\end{align*}

For (ii) and (iii), we deduce from Lemma~\ref{lem:sto_lmd_kp} that
\begin{align*}
[\mathbf{t}_{12}(a),\mathbf{t}_{21}(bc)]
&=\mathbf{h}_{11}(a,bc)-(-1)^{|a|(|b|+|c|)}\mathbf{h}_{21}(1,bca),\\
[\mathbf{t}_{12}(ca),\mathbf{t}_{21}(b)]
&=\mathbf{h}_{11}(ca,b)-(-1)^{|b|(|a|+|c|)}\mathbf{h}_{21}(1,bca).
\end{align*}
Now, (\ref{eq:sto_h_rlns1}) implies that
\begin{align*}
\boldsymbol{\lambda}(ab,c)=\mathbf{h}_{11}(a,bc)-(-1)^{|c|(|a|+|b|)}\mathbf{h}_{11}(ca,b)
=\boldsymbol{\lambda}(a,bc)-(-1)^{|c|(|a|+|b|)}\boldsymbol{\lambda}(ca,b).
\end{align*}
Hence, the equality (ii) holds by setting $a=1$, from which we conclude that (iii) also holds.
\end{proof}
\medskip

\begin{remark}
The equalities $\boldsymbol{\lambda}(a,1)=\boldsymbol{\lambda}(1,a)=0$ also hold in the Lie superalgebra $\mathfrak{sto}_{m|2n}(R,{}^-)$. Indeed, we know $\boldsymbol{\lambda}(1,a)=0$ from its definition, and $\boldsymbol{\lambda}(a,1)=0$ follows from Lemma~\ref{lem:sto_h_rln} (ii).
\end{remark}

Now, we may proceed to prove Proposition~\ref{prop:sto_ker}

\begin{proof}[Proof of Proposition~\ref{prop:sto_ker}]
By Lemma~\ref{lem:sto_h_rln}, there is a well-defined $\Bbbk$--linear map
$$\eta:\langle R,R\rangle_{\mathsf{d}}^-\rightarrow\mathfrak{sto}_{m|2n}(R,{}^-),\quad \langle a,b\rangle_{\mathsf{d}}^-\mapsto\boldsymbol{\lambda}(a,b).$$
We will prove that its restriction on $\fourIdx{}{-}{}{1}{\mathrm{HD}}(R,{}^-)$ is an isomorphism of $\Bbbk$-modules.
\medskip

We have already known from Lemma~\ref{lem:sto_ker} that $\eta(\fourIdx{}{-}{}{1}{\mathrm{HD}}(R,{}^-))=\ker\psi$. In order to show that $\eta$ is injective, we need define a $\Bbbk$-bilinear map
$$\alpha:\mathfrak{gl}_{m|2n}(R)\times\mathfrak{gl}_{m|2n}(R)\rightarrow\langle R, R\rangle_{\mathsf{d}}^-$$
by
$$\alpha(e_{ij}(a),e_{kl}(b))=\delta_{jk}\delta_{il}(-1)^{|i|(|i|+|a|+|b|)}\langle a,b\rangle_{\mathsf{d}}^-$$
for $1\leqslant i,j\leqslant m+2n$ and homogeneous $a,b\in R$. It is directly verified that $\alpha$ is a $2$-cocycle on the Lie superalgebra $\mathfrak{gl}_{m|2n}(R)$, i.e., $\alpha$ satisfies
\begin{align}
\alpha(x,y)=-&(-1)^{|x||y|}\alpha(y,x),&\tag{CC1}\label{eq:2cyl1}\\
(-1)^{|x||z|}\alpha([x,y],z)+(-1)^{|y||x|}&\alpha([y,z],x)+(-1)^{|z||y|}\alpha([z,x],y)=0,&\tag{CC2}\label{eq:2cyl2}
\end{align}
for homogeneous $x,y,z\in\mathfrak{gl}_{m|2n}(R)$. Hence, $\alpha$ induces a $2$-cocycle on the Lie superalgebra $\mathfrak{osp}_{m|2n}(R,{}^-)$ through restriction. This yields a new Lie superalgebra
$$\mathfrak{C}:=\mathfrak{osp}_{m|2n}(R,{}^-)\oplus\langle R,R\rangle_{\mathsf{d}}^-$$
with the multiplication
$$[x\oplus c, y\oplus c']=[x,y]\oplus \alpha(x,y),\quad x,y\in\mathfrak{osp}_{m|2n}(R,{}^-)\text{ and }c,c'\in\langle R,R\rangle_{\mathsf{d}}^-,$$
where $[x,y]$ denotes the multiplication in $\mathfrak{osp}_{m|2n}(R,{}^-)$.

In the Lie superalgebra $\mathfrak{C}$, we denote
\begin{align*}
\tilde{t}_{ij}(a):&=t_{ij}(a)\oplus0,
&\tilde{u}_{kl}(a):&=u_{kl}(a)\oplus0,
&\tilde{v}_{kl}(a):&=v_{kl}(a)\oplus0,\\
\tilde{w}_{kl}(a):&=w_{kl}(a)\oplus0,
&\tilde{f}_{i'k'}(a):&=f_{i'k'}(a)\oplus0,
&\tilde{g}_{k'i'}(a):&=g_{k'i'}(a)\oplus0,
\end{align*}
for $a\in R$, $1\leqslant i',i,j\leqslant m$ with $i\neq j$ and $1\leqslant k',k,l\leqslant n$ with $k\neq l$. These elements of $\mathfrak{C}$ satisfy all relations (\ref{STO01})-(\ref{STO28}). Hence, there is a canonical homomorphism of Lie superalgebras
$$\phi:\mathfrak{sto}_{m|2n}(R,{}^-)\rightarrow\mathfrak{C}$$
such that
\begin{align*}
\phi(\mathbf{t}_{ij}(a))&=\tilde{t}_{ij}(a),
&\phi(\mathbf{u}_{kl}(a))&=\tilde{u}_{kl}(a),
&\phi(\mathbf{v}_{kl}(a))&=\tilde{v}_{kl}(a),\\
\phi(\mathbf{w}_{kl}(a))&=\tilde{w}_{kl}(a),
&\phi(\mathbf{f}_{i'k'}(a))&=\tilde{f}_{i'k'}(a),
&\phi(\mathbf{g}_{k'i'}(a))&=\tilde{g}_{k'i'}(a).
\end{align*}
We now compute that
\begin{align*}
\phi(\mathbf{h}_{ik}(a,b))
&=\phi([\mathbf{f}_{ik}(a),\mathbf{g}_{ki}(b)])
=[\tilde{f}_{ik}(a),\tilde{g}_{ki}(b)]\\
&=[{f}_{ik}(a),{g}_{ki}(b)]\oplus \alpha({f}_{ik}(a),{g}_{ki}(b))\\
&=(e_{ii}(ab-\overline{ab})-(-1)^{(|a|+1)(|b|+1)}{u}_{kk}(ba))\oplus(\langle a,b\rangle_{\mathsf{d}}^-+\langle\bar{a},\bar{b}\rangle_{\mathsf{d}}^-)\\
&=(e_{ii}(ab-\overline{ab})-(-1)^{(|a|+1)(|b|+1)}{u}_{kk}(ba))\oplus2\langle a,b\rangle_{\mathsf{d}}^-.
\end{align*}
It yields that
\begin{align*}
\phi(\boldsymbol{\lambda}(a,b))&=\phi(\mathbf{h}_{11}(a,b)-(-1)^{|a||b|}\mathbf{h}_{11}(1,ba))\\
&=e_{11}([a,b]-\overline{[a,b]})\oplus(2\langle a,b\rangle_{\mathsf{d}}^--2\langle1,ba\rangle_{\mathsf{d}}^-).
\end{align*}
Note that $\langle 1,ba\rangle_{\mathsf{d}}^-=0$, we obtain
$$\phi(\boldsymbol{\lambda}(a,b))=e_{11}([a,b]-\overline{[a,b]})\oplus2\langle a,b\rangle_{\mathsf{d}}^-.$$
Hence,
\begin{align*}
\phi(\eta(\langle a,b\rangle))&=\phi(\boldsymbol{\lambda}(a,b))
=e_{11}([a,b]-\overline{[a,b]})\oplus2\langle a,b\rangle_{\mathsf{d}}^-,
\end{align*}
which shows that $\eta$ is injective.
\end{proof}

Summarizing the main results in this section, we have the following:

\begin{theorem}
\label{thm:osp_ce}
Let $m,n\in\mathbb{N}$ such that $(m,n)\neq(1,1)$ and $(R,{}^-)$ be a unital associative superalgebra with superinvolution. Then $\psi:\mathfrak{sto}_{m|2n}(R,{}^-)\rightarrow\mathfrak{osp}_{m|2n}(R,{}^-)$ is a central extension and
$$\ker\psi\cong\fourIdx{}{-}{}{1}{\mathrm{HD}}(R,{}^-)$$
as $\Bbbk$--modules.\qed
\end{theorem}

\section{The universal central extension of $\mathfrak{osp}_{m|2n}(R,{}^-)$}
\label{sec:uce}
In this section, we will show that the central extension $\psi:\mathfrak{sto}_{m|2n}(R,{}^-)\rightarrow\mathfrak{osp}_{m|2n}(R,{}^-)$ is universal for $(m,n)\ne(1,1),(2,1)$.

Let $\varphi:\mathfrak{E}\rightarrow\mathfrak{osp}_{m|2n}(R,{}^-)$ be an arbitrary central extension. For $a\in R$, $1\leqslant i',i,j\leqslant m$ with $i\neq j$ and $1\leqslant k',k,l\leqslant n$ with $k\neq l$, we pick
\begin{align*}
&\hat{t}_{ij}(a)\in\varphi^{-1}({t}_{ij}(a)),
\quad\hat{u}_{kl}(a)\in\varphi^{-1}({u}_{kl}(a)),\quad\hat{v}_{kl}(a)\in\varphi^{-1}({v}_{kl}(a)),\\
&\hat{w}_{kl}(a)\in\varphi^{-1}({w}_{kl}(a)),\quad\hat{f}_{i'k'}(a)\in\varphi^{-1}({f}_{i'k'}(a)),\quad\hat{g}_{k'i'}(a)\in\varphi^{-1}({g}_{k'i'}(a)).
\end{align*}
Then $[\hat{x},\hat{y}]$ is independent of the choice of the representatives $\hat{x}\in\varphi^{-1}(x)$ and $\hat{y}\in\varphi^{-1}(y)$ for $x,y\in\mathfrak{osp}_{m|2n}(R,{}^-)$. Hence, we have the following well-defined elements in $\mathfrak{E}$:
\begin{equation}
\begin{aligned}
\tilde{h}_{k}:&=[\hat{g}_{k1}(1),\hat{f}_{1k}(1)],\\
\tilde{t}_{ij}(a):&=[\hat{f}_{i1}(1),\hat{g}_{1j}(a)],
&\tilde{u}_{kl}(a):&=[\hat{g}_{k1}(1),\hat{f}_{1l}(a)],\\
\tilde{v}_{kl}(a):&=-[\hat{g}_{k1}(a),\hat{g}_{l1}(1)],
&\tilde{w}_{kl}(a):&=[\hat{f}_{1k}(1),\hat{f}_{1l}(a)],\\
\tilde{f}_{i'k'}(a):&=[\hat{f}_{i'k'}(a),\tilde{h}_{k'}],
&\tilde{g}_{k'i'}(a):&=[\tilde{h}_{k'},\hat{g}_{k'i'}(a)],
\end{aligned}\label{eq:osptilde}
\end{equation}
where $a\in R$, $1\leqslant i',i,j\leqslant m$ with $i\neq j$ and $1\leqslant k',k,l\leqslant n$ with $k\neq l$.

\begin{lemma}
\label{lem:tilde_rln}
Suppose that $m,n\in\mathbb{N}$. In the Lie superalgebra $\mathfrak{E}$, the elements $\tilde{t}_{ij}(a)$, $\tilde{u}_{kl}(a)$, $\tilde{v}_{kl}(a)$, $\tilde{w}_{kl}(a)$, $\tilde{f}_{i'k'}(a)$, $\tilde{g}_{k'i'}(a)$, where $a\in R$, $1\leqslant i',i,j\leqslant m$ with $i\neq j$ and $1\leqslant k',k,l\leqslant n$ with $k\neq l$ satisfy relations {\normalfont(\ref{STO00}), (\ref{STO02})-(\ref{STO26}), (\ref{STO28})} and the following relation:
\begin{align}
[\tilde{f}_{ik}(a),\tilde{g}_{lj}(b)]&=0,&&\text{for }i\neq j\text{ and }k\neq l.\tag{STO27a}\label{STO27a}
\end{align}
\end{lemma}

\begin{remark}
We should notice that (\ref{STO27a}) is deferent from (\ref{STO27}). Indeed, (\ref{STO27}) is equivalent to (\ref{STO27a}) and the following
\begin{equation}
[\mathbf{f}_{ik}(a),\mathbf{g}_{kj}(b)]=\mathbf{t}_{ij}(ab),\quad\text{for }i\neq j.\tag{STO27b}\label{STO27b}
\end{equation}
We will see in the next section that (\ref{STO01}) and (\ref{STO27b}) are not necessarily true in a central extension of $\mathfrak{osp}_{2|2}(R,{}^-)$.
\end{remark}

\begin{proof}[Proof of Lemma~\ref{lem:tilde_rln}]
The proof follows from a case by case verification. The computation is quite straightforward but tedious. We only show (\ref{STO05}) here as an example.

In order to show that $\tilde{t}_{ij}(a)$ and $\tilde{u}_{kl}(b)$ satisfy (\ref{STO05}), we claim that
$$[\tilde{h}_k,\tilde{t}_{ij}(a)]=0,\text{ and }[\tilde{h}_k,\tilde{u}_{kl}(b)]=\tilde{u}_{kl}(b).$$
These can be directly verified as follows:
\begin{align*}
[\tilde{h}_{k}, \tilde{t}_{ij}(a)]
&=[\tilde{h}_{k},[\hat{f}_{i1}(1),\hat{g}_{1j}(a)]]
=-\delta_{k1}[\hat{f}_{i1}(1),\hat{g}_{1j}(a)]+\delta_{k1}[\hat{f}_{i1}(1), \hat{g}_{1j}(a)]
=0,\\
[\tilde{h}_{k}, \tilde{u}_{kl}(a)]
&=[\tilde{h}_{k},[\hat{g}_{k1}(1),\hat{f}_{1l}(a)]]
=[\tilde{g}_{k1}(1),\hat{f}_{1l}(a)]+0
=\tilde{u}_{kl}(a).
\end{align*}

Now, we observe that $[\tilde{t}_{ij}(a),\tilde{u}_{kl}(b)]\in\ker\varphi$ which is contained in the center of $\mathfrak{E}$. Using the claim, we deduce that
$$0=[\tilde{h}_k,[\tilde{t}_{ij}(a),\tilde{u}_{kl}(b)]]
=[[\tilde{h}_k,\tilde{t}_{ij}(a)],\tilde{u}_{kl}(b)]
+[\tilde{t}_{ij}(a),[\tilde{h}_k,\tilde{u}_{kl}(b)]]
=0+[\tilde{t}_{ij}(a),\tilde{u}_{kl}(b)],$$
which shows $\tilde{t}_{ij}(a)$ and $\tilde{u}_{kl}(b)$ for $i\neq j$ and $k\neq l$ satisfy (\ref{STO05}).

To avoid the tedious computational details, we omit the verification for other relations here.
\end{proof}

\begin{theorem}
\label{thm:sto_uce}
Let $(R,^-)$ be a unital associative superalgebra with superinvolution.
Suppose $m,n\in\mathbb{N}$ such that $(m,n)\neq(1,1), (2,1)$. Then the central extension $$\psi:\mathfrak{sto}_{m|2n}(R,^-)\rightarrow\mathfrak{osp}_{m|2n}(R,{}^-)$$
is universal.
\end{theorem}
\begin{proof}
Let $\varphi:\mathfrak{E}\rightarrow\mathfrak{osp}_{m|2n}(R,{}^-)$ be an arbitrary central extension. We have to show that there is a unique homomorphism $\varphi':\mathfrak{sto}_{m|2n}(R,{}^-)\rightarrow\mathfrak{E}$ such that $\varphi\circ\varphi'=\psi$.

Take the elements $\tilde{t}_{ij}(a),$ $\tilde{u}_{kl}(a),$ $\tilde{v}_{kl}(a),$ $\tilde{w}_{kl}(a),$ $\tilde{f}_{i'k'}(a),$ $\tilde{g}_{k'i'}(a)\in\mathfrak{E}$ for $a\in R$, $1\leqslant i',i, j\leqslant m$ with $i\neq j$ and $1\leqslant k',k,l\leqslant n$ with $k\neq l$ as in (\ref{eq:osptilde}). We have already known from Lemma~\ref{lem:tilde_rln} that they satisfy all relations (\ref{STO00}), (\ref{STO02})--(\ref{STO26}) and (\ref{STO28}). Under the additional assumption that $(m,n)\ne (1,1), (2,1)$, we will show that these elements also satisfy relations (\ref{STO01}) and (\ref{STO27}).

For (\ref{STO27}), we have obtained (\ref{STO27a}) in Lemma~\ref{lem:tilde_rln}. It suffices to show (\ref{STO27b}). Since (\ref{STO27b}) is null if $m=1$, we assume $m\geqslant2$ and $1\leqslant i\neq j\leqslant m$. We need to prove (\ref{STO27b}) in the case of $n\geqslant2$ and the case of $n=1$ separately. If $n\geqslant2$, we choose $1\leqslant l\neq k\leqslant n$ and deduce that
\begin{align*}
[\tilde{f}_{ik}(a),\tilde{g}_{kj}(b)]
&=[[\tilde{f}_{il}(1),\tilde{u}_{lk}(a)],\tilde{g}_{kj}(b)]\\
&=[\tilde{f}_{il}(1),[\tilde{u}_{lk}(a),\tilde{g}_{kj}(b)]]-(-1)^{|a|}[\tilde{u}_{lk}(a),[\tilde{f}_{il}(1),\tilde{g}_{kj}(b)]]\\
&=[\tilde{f}_{il}(1),\tilde{g}_{lj}(ab)]]
\end{align*}
which is equal to $\tilde{t}_{ij}(ab)$ when $l=1$. If $l\neq1$, the same argument shows that
\begin{align*}
[\tilde{f}_{il}(1),\tilde{g}_{lj}(ab)]=[\tilde{f}_{i1}(1),\tilde{g}_{1j}(ab)]=\tilde{t}_{ij}(ab).
\end{align*}

If $n=1$, we may assume $m\geqslant3$ since $(m,n)\neq(1,1),(2,1)$. Choose $i'$ such that $i'\ne i,j$. Then
\begin{align*}
[\tilde{f}_{i1}(a),\tilde{g}_{1j}(b)]
&=[[\tilde{t}_{ii'}(a),\tilde{f}_{i'1}(1)],\tilde{g}_{1j}(b)]\\
&=[\tilde{t}_{ii'}(a),[\tilde{f}_{i'1}(1),\tilde{g}_{1j}(b)]]-(-1)^{|a|}[\tilde{f}_{i'1}(1),[\tilde{t}_{ii'}(a),\tilde{g}_{1j}(b)]]\\
&=[\hat{t}_{ii'}(a),\tilde{t}_{i'j}(b)]]\\
&=\tilde{t}_{ij}(ab),
\end{align*}
where the last equality follows from (\ref{STO03}) that has been proved in Lemma~\ref{lem:tilde_rln}. Hence, (\ref{STO27b}) holds for $(m,n)\neq(1,1), (2,1)$.

The relation (\ref{STO01}) can be proved similarly. We also assume that $m\geqslant2$ and $1\leqslant i\neq j\leqslant m$. We consider the case of $n\geqslant2$ and the case of $n=1$ separately. If $n\geqslant2$, we have $\hat{f}_{j2}(1),\hat{g}_{2i}(a),\hat{v}_{12}(\bar{a})\in\mathfrak{E}$ and deduce that
\begin{align*}
\tilde{t}_{ij}(\bar{a})&=[\hat{f}_{i1}(1),\hat{g}_{1j}(\bar{a})]\\
&=-(-1)^{|a|}[\hat{f}_{i1}(1),[\hat{f}_{j2}(1),\hat{v}_{21}(a)]]\\
&=-(-1)^{|a|}[[\hat{f}_{i1}(1),\hat{f}_{j2}(1)],\hat{v}_{21}(a)]+(-1)^{|a|}[\hat{f}_{j2}(1),[\hat{f}_{i1}(1),\hat{v}_{12}(\bar{a})]]\\
&=-[\hat{f}_{j2}(1),\hat{g}_{2i}(a)]\\
&=-\tilde{t}_{ji}(a),
\end{align*}
where the last equality follows from (\ref{STO27b}) that we just proved.

In the case where $n=1$, we know that $m\geqslant3$ since $(m,n)\neq(1,1),(2,1)$.  Choose $j'$ such that $j'\ne i,j$, then (\ref{STO01}) follows from the computation as follows:
$$\tilde{t}_{ij}(\bar{a})=[\tilde{t}_{ij'}(\bar{a}),\tilde{t}_{j'j}(1)]=[\hat{t}_{j'i}(a),\hat{t}_{jj'}(1)]]=-\tilde{t}_{ji}(a),$$
where the last equality follows from (\ref{STO03}) that has been proved in Lemma~\ref{lem:tilde_rln}.

Summarizing, the elements $\tilde{t}_{ij}(a),\tilde{u}_{kl}(a),\tilde{v}_{kl}(a),\tilde{w}_{kl}(a),\tilde{f}_{i'k'}(a),\tilde{g}_{k'i'}(a)$ for $a\in R$, $1\leqslant i',i,j\leqslant m$ with $i\neq j$ and $1\leqslant k',k,l\leqslant n$ satisfy all relations (\ref{STO00})-(\ref{STO28}). Hence, there is a homomorphism of Lie superalgebras
$$\varphi':\mathfrak{sto}_{m|2n}(R,{}^-)\rightarrow\mathfrak{E}$$
such that
\begin{align*}
\varphi'(\mathbf{t}_{ij}(a))&=\tilde{t}_{ij}(a),
&\varphi'(\mathbf{u}_{kl}(a))&=\tilde{u}_{kl}(a),
&\varphi'(\mathbf{v}_{kl}(a))&=\tilde{v}_{kl}(a),\\
\varphi'(\mathbf{w}_{kl}(a))&=\tilde{w}_{kl}(a),
&\varphi'(\mathbf{f}_{ik}(a))&=\tilde{f}_{ik}(a),
&\varphi'(\mathbf{g}_{ki}(a))&=\tilde{g}_{ki}(a).
\end{align*}
i.e., $\varphi\circ\varphi'=\psi$.

The homomorphism $\varphi'$ above is unique. Suppose that $\tilde{\varphi}':\mathfrak{sto}_{m|2n}(R,{}^-)\rightarrow\mathfrak{osp}_{m|2n}(R,{}^-)$ is another homomorphism satisfying $\varphi\circ\tilde{\varphi}'=\psi$. Since $\mathbf{t}_{ij}(a)=[\mathbf{f}_{i1}(a),\mathbf{g}_{1j}(1)]$ in $\mathfrak{osp}_{m|2n}(R,{}^-)$, we deduce that
$$\tilde{\varphi}'(\mathbf{t}_{ij}(a))=[\tilde{\varphi}'(\mathbf{f}_{i1}(a)),\tilde{\varphi}'(\mathbf{g}_{1j}(1))]
\in[\varphi^{-1}(f_{i1}(a)),\varphi^{-1}(g_{1j}(1)]$$
which contains the unique element $\tilde{t}_{ij}(a)=\varphi'(\mathbf{t}_{ij}(a))$ since $\varphi:\mathfrak{E}\rightarrow\mathfrak{osp}_{m|2n}(R,{}^-)$ is a central extension. Similarly, $\tilde{\varphi}'$ and $\varphi'$ coincide when evaluating at $\mathbf{u}_{kl}(a),\mathbf{v}_{kl}(a),\mathbf{w}_{kl}(a),\mathbf{f}_{i'k'}(a),\mathbf{g}_{k'i'}(a)$ for $a\in R$, $1\leqslant i'\leqslant m$ and $1\leqslant k',k,l\leqslant n$ with $k\neq l$. It follows that $\tilde{\varphi}'=\varphi'$. This shows the uniqueness of $\varphi'$ and the universality of $\psi$ follows.
\end{proof}

Theorems~\ref{thm:osp_ce} and~\ref{thm:sto_uce} imply that
\begin{corollary}
\label{cor:hml}
$\mathrm{H}_2(\mathfrak{osp}_{m|2n}(R,{}^-),\Bbbk)=\fourIdx{}{-}{}{1}{\mathrm{HD}}(R,{}^-)$ if $(m,n)\neq(1,1), (2,1)$.\qed
\end{corollary}

\begin{remark}
\label{rmk:hml_com}
If $R$ is a unital super-commutative associative superalgebra with the identity superinvolution, then the Lie superalgebra $\mathfrak{osp}_{m|2n}(R,\mathrm{id})$ is isomorphic to $\mathfrak{osp}_{m|2n}(\Bbbk)\otimes_{\Bbbk}R$ (see Example~\ref{eg:osp_com}). Hence,
$$\mathrm{H}_2(\mathfrak{osp}_{m|2n}(\Bbbk)\otimes_{\Bbbk}R,\Bbbk)
=\fourIdx{}{-}{}{1}{\mathrm{HD}}(R,\mathrm{id})
=\mathrm{HC}_1(R),$$
for $(m,n)\neq (1,1),(2,1)$, which coincides with the second homology groups of $\mathfrak{osp}_{m|2n}(\Bbbk)\otimes_{\Bbbk}R$ given in \cite{IoharaKoga2005}.
\end{remark}

In the situation where $(R,{}^-)$ is equal to $(S\oplus S^{\mathrm{op}},\mathrm{ex})$ for an arbitrary unital associative superalgebra $S$, we obtain the following corollary, which is part of the results proved in \cite{ChenSun2015}.
\begin{corollary}
\label{cor:slm2n}
Let $S$ be an arbitrary unital associative superalgebra. Assume that $(m,n)\neq(1,1)$ and $(2,1)$. Then
$\mathfrak{st}_{m|2n}(S)$ is the universal central extension of $\mathfrak{sl}_{m|2n}(S)$ and
$$\mathrm{H}_2(\mathfrak{sl}_{m|2n}(S),\Bbbk)=\mathrm{HC}_1(S).$$
\end{corollary}
\begin{proof}
Similar to generalized periplectic Lie superalgebras, there is a commutative diagram:
$$\xymatrix{
0\ar[r]&
\fourIdx{}{-}{}{1}{\mathrm{HD}}(S\oplus S^{\mathrm{op}},\mathrm{ex})\ar[r]\ar[d]&
\mathfrak{sto}_{m|2n}(S\oplus S^{\mathrm{op}},\mathrm{ex})\ar[r]\ar[d]&
\mathfrak{osp}_{m|2n}(S\oplus S^{\mathrm{op}},\mathrm{ex})\ar[r]\ar[d]&0\\
0\ar[r]&
\mathrm{HC}_1(S)\ar[r]&
\mathfrak{st}_{m|2n}(S)\ar[r]&
\mathfrak{sl}_{m|2n}(S)\ar[r]&0
}$$
where the three vertical arrows are isomorphisms given in \cite[Proposition~6.6]{ChangWang2015},  Proposition~\ref{prop:sto_SS} and Example~\ref{eg:osp_SS}, respectively. Hence, the lemma follows.
\end{proof}

\section{The universal central extension of $\mathfrak{osp}_{2|2}(R,{}^-)$}
\label{sec:osp22}

We have shown before that the central extension $\psi:\mathfrak{sto}_{m|2n}(R,{}^-)\rightarrow\mathfrak{osp}_{m|2n}(R,{}^-)$ is universal for $(m,n)\neq(1,1), (2,1)$. For the Lie superalgebra $\mathfrak{osp}_{2|2}(R,{}^-)$, it has been proved in \cite{MikhalevPinchuk2002} that $\psi:\mathfrak{sto}_{m|2n}(R,{}^-)\rightarrow\mathfrak{osp}_{m|2n}(R,{}^-)$ is also a universal central extension provided that $R$ is super-commutative and the superinvolution ${}^-$ is the identity map. However, this is not necessarily true if $R$ is not super-commutative. This section will contribute to discuss the universal central extension of $\mathfrak{sto}_{2|2}(R,{}^-)$. We will first create a nontrivial central extension of $\mathfrak{sto}_{2|2}(R,{}^-)$ using a concrete $2$-cocycle and then prove the universality of this central extension under certain assumption.
\medskip

\begin{remark}
In the case where $(R,{}^-)=(\Bbbk,\mathrm{id})$, the Lie superalgebra $\mathfrak{osp}_{2|2}(\Bbbk,\mathrm{id})$ is isomorphic to $\mathfrak{sl}_{1|2}(\Bbbk)$ (c.f. \cite[Section~4.2.2]{Kac1977}). However, this isomorphism is not functorial in $(R,{}^-)$. If we take $(R,{}^-)=(\mathrm{M}_{1|1}(\Bbbk),\mathrm{prp})$, then it follows from Example~\ref{eg:osp_prp} that $\mathfrak{osp}_{2|2}(\mathrm{M}_{1|1}(\Bbbk),\mathrm{prp})$
is isomorphic to $\mathfrak{p}_4(\Bbbk)$, which is obviously not isomorphic to $\mathfrak{sl}_{1|2}(\mathrm{M}_{1|1}(\Bbbk))$.
\end{remark}

It is obvious that the $\Bbbk$-module $R/([R,R]R)$ is a super-commutative associative superalgebra. We use $\boldsymbol{\pi}(a)$ to denote the canonical image of $a\in R$ in $R/([R,R]R)$. Recall that $\mathfrak{sto}_{2|2}(R,{}^-)$ (as a $\Bbbk$-module) is spanned by $$\mathfrak{B}:=\{\mathbf{h}_{i1}(a,b),\mathbf{t}_{12}(a),\mathbf{v}_1(a),\mathbf{w}_1(a),\mathbf{f}_{i1}(a),\mathbf{g}_{1i}(a)|i=1,2,a\in R\}.$$
We define a $\Bbbk$--bilinear map $\beta:\mathfrak{sto}_{2|2}(R,{}^-)\times\mathfrak{sto}_{2|2}(R,{}^-)\rightarrow R/([R,R] R)$ by setting
\begin{align*}
\beta(\mathbf{f}_{11}(a),\mathbf{g}_{12}(b))
&=-(-1)^{(1+|a|)(1+|b|)}\beta(\mathbf{g}_{12}(b),\mathbf{f}_{11}(a))
=\boldsymbol{\pi}(a\bar{b}),\\
\beta(\mathbf{f}_{21}(a),\mathbf{g}_{11}(b))
&=-(-1)^{(1+|a|)(1+|b|)}\beta(\mathbf{g}_{11}(b),\mathbf{f}_{21}(a))
=-\boldsymbol{\pi}(a\bar{b}),
\end{align*}
and $\beta(x,y)=0$ for all other pairs $(x,y)$ of elements in $\mathfrak{B}$. Then we have

\begin{lemma}
\label{lem:osp22_2cy}
Let $(R,{}^-)$ be a unital associative superalgebra with superinvolution. Then $\beta$ is a $2$-cocycle on $\mathfrak{sto}_{2|2}(R,{}^-)$ with values in $R/([R,R]R)$.
\end{lemma}
\begin{proof}
It is obvious from the definition of $\beta$ that $\beta$ satisfies (\ref{eq:2cyl1}). It suffices to show
$$J(x,y,z):=(-1)^{|x||z|}\beta([x,y],z)+(-1)^{|y||x|}\beta([y,z],x)+(-1)^{|z||x|}\beta([z,x],y)=0$$
for $x,\,y,\,z\in \mathfrak{B}$. Noting that $J(x,y,z)$ is symmetric with respect to all permutations on $\{x,y,z\}$ and eliminating the cases where $\beta([x,y],z)=\beta([y,z],x)=\beta([z,x],y)=0$, the problem is reduced to verifying $J(x,y,z)=0$ in the following eight cases:
\begin{enumerate}
\item $x=\mathbf{h}_{11}(a,b),\,y=\mathbf{f}_{11}(c)$ and $z=\mathbf{g}_{12}(d)$,
\item $x=\mathbf{h}_{11}(a,b),\,y=\mathbf{f}_{21}(c)$ and $z=\mathbf{g}_{11}(d)$,
\item $x=\mathbf{h}_{21}(a,b),\,y=\mathbf{f}_{11}(c)$ and $z=\mathbf{g}_{12}(d)$,
\item $x=\mathbf{h}_{21}(a,b),\,y=\mathbf{f}_{21}(c)$ and $z=\mathbf{g}_{11}(d)$,
\item $x=\mathbf{t}_{12}(a),\,y=\mathbf{f}_{11}(b)$ and $z=\mathbf{g}_{11}(c)$,
\item $x=\mathbf{t}_{12}(a),\,y=\mathbf{f}_{21}(b)$ and $z=\mathbf{g}_{12}(c)$,
\item $x=\mathbf{v}_{1}(a),\,y=\mathbf{f}_{11}(b)$ and $z=\mathbf{f}_{21}(c)$,
\item $x=\mathbf{w}_{1}(a),\,y=\mathbf{g}_{12}(b)$ and $z=\mathbf{g}_{11}(c)$,
\end{enumerate}
where $a,b,c,d\in R$ are homogeneous. The verification can be done by direct computation. For instance, in case (i), we compute that
\begin{align*}
J(x,y,z)=
&(-1)^{(|a|+|b|)(1+|d|)}\beta([\mathbf{h}_{11}(a,b),\mathbf{f}_{11}(c)],\mathbf{g}_{12}(d))\\
&+(-1)^{(|1|+|c|)(|a|+|b|)}\beta([\mathbf{f}_{11}(c),\mathbf{g}_{12}(d)],\mathbf{h}_{11}(a,b))\\
&+(-1)^{(|1|+|c|)(|1|+|d|)}\beta([\mathbf{g}_{12}(d),\mathbf{h}_{11}(a,b)],\mathbf{f}_{11}(c))\\
=&(-1)^{(|a|+|b|)(1+|d|)}\beta(\mathbf{f}_{11}(abc-\overline{ab}c-(-1)^{|a||b|+|b||c|+|a||c|}cba),\mathbf{g}_{12}(d))\\
&+(-1)^{(|1|+|c|)(|a|+|b|)}\beta(\mathbf{t}_{12}(cd),\mathbf{h}_{11}(a,b))\\
&+(-1)^{|c|+|d|+|a||b|+|b||d|+|d||a|+|c||d|}\beta(\mathbf{g}_{12}(bad),\mathbf{f}_{11}(c))\\
=&(-1)^{(|a|+|b|)(1+|d|)}\boldsymbol{\pi}(abc\bar{d}-\overline{ab}c\bar{d}-(-1)^{|a||b|+|b||c|+|a||c|}cba\bar{d})\\
&+0+(-1)^{|a||b|+(|a|+|b|)(|c|+|d|+1)}\boldsymbol{\pi}(c\overline{bad})\\
=&-(-1)^{(|a|+|b|)(1+|d|)}\boldsymbol{\pi}(\overline{ab}c\bar{d})+(-1)^{|a||b|+(|a|+|b|)(|c|+|d|+1)}\boldsymbol{\pi}(c\overline{bad})\\
=&-(-1)^{(|a|+|b|)(1+|d|)}\boldsymbol{\pi}(\bar{a}\bar{b}c\bar{d})+(-1)^{(|a|+|b|)(1+|d|)}\boldsymbol{\pi}(\bar{a}\bar{b}c\bar{d})\\
=&0.
\end{align*}
Hence, $J(x,y,z)=0$ in Case (i). Similarly, we may verify that $J(x,y,z)=0$ in Cases (ii)-(viii), we omit the details here.
\end{proof}

The $2$-cocycle $\beta$ gives rise to a new Lie superalgebra
$$\widehat{\mathfrak{sto}}_{2|2}(R,{}^-):=\mathfrak{sto}_{2|2}(R,{}^-)\oplus R/([R,R]R),$$
on which the super-bracket is given by
$$[x\oplus c,y\oplus c']:=[x,y]\oplus \beta(x,y),\quad x,y\in\mathfrak{sto}_{2|2}(R,{}^-),\text{ and }c,c'\in R/([R,R]R).$$
Moreover, the canonical projection $\psi':\widehat{\mathfrak{sto}}_{2|2}(R,{}^-)\rightarrow\mathfrak{sto}_{2|2}(R,{}^-)$ is a central extension.

Alternatively, the Lie superalgebra $\widehat{\mathfrak{sto}}_{2|2}(R,{}^-)$ can be defined as the abstract Lie superalgebra generated by $\tilde{\mathbf{t}}_{12}(a)$, $\tilde{\mathbf{f}}_{11}(a)$, $\tilde{\mathbf{f}}_{21}(a)$, $\tilde{\mathbf{g}}_{11}(a)$, $\tilde{\mathbf{g}}_{12}(a)$ for $a\in R$ together with the super-commutative Lie superalgebra $R/([R,R]R)$, subjecting to the relations:
\begin{align}
&[\tilde{\mathbf{t}}_{12}(a),\tilde{\mathbf{f}}_{11}(b)]=-\tilde{\mathbf{f}}_{21}(\bar{a}b),
&&[\tilde{\mathbf{t}}_{12}(a),\tilde{\mathbf{f}}_{21}(b)]=\tilde{\mathbf{f}}_{11}(ab),\label{eq:hatsto22_02}\\
&[\tilde{\mathbf{g}}_{11}(a),\tilde{\mathbf{t}}_{12}(b)]=\tilde{\mathbf{g}}_{12}(ab),
&&[\tilde{\mathbf{g}}_{12}(a),\tilde{\mathbf{t}}_{12}(b)]=-\tilde{\mathbf{g}}_{11}(a\bar{b}),\label{eq:hatsto22_03}\\
&[\tilde{\mathbf{f}}_{11}(a),\tilde{\mathbf{f}}_{21}(b)]=0,
&&[\tilde{\mathbf{g}}_{11}(a),\tilde{\mathbf{g}}_{12}(b)]=0,\label{eq:hatsto22_04}\\
&[\tilde{\mathbf{f}}_{11}(a),\tilde{\mathbf{g}}_{12}(b)]
=\tilde{\mathbf{t}}_{12}(ab)+\boldsymbol{\pi}(a\bar{b}),
&&[\tilde{\mathbf{f}}_{21}(a),\tilde{\mathbf{g}}_{11}(b)]
=-\tilde{\mathbf{t}}_{12}(\overline{ab})-\boldsymbol{\pi}(a\bar{b}).\label{eq:hatsto22_05}
\end{align}
\medskip

Now, we have obtained a nontrivial central extension $\psi':\widehat{\mathfrak{sto}}_{m|2n}(R,{}^-)\rightarrow\mathfrak{sto}_{m|2n}(R,{}^-)$. In the remaining part of this section, we will show its universality under the following assumption:

\begin{assumption}
\label{asmp_osp22}
Let $(R,{}^-)$ be a unital associative superalgebra with superinvolution. We say that $(R,{}^-)$ satisfies Assumption~\ref{asmp_osp22} if $R$ contains a homogeneous element $e$ such that
\begin{enumerate}
\item $\bar{e}=-e$,
\item $e$ is contained in the center of $R$, and
\item $e$ is a unit element.
\end{enumerate}
\end{assumption}

\begin{remark}
\label{rmk:osp12_asmp}
As indicated in \cite{Gao1996II}, Assumption~\ref{asmp_osp22} is not restrictive. If $R$ has no element $a$ such that $a^2=-1$, then $R\otimes\Bbbk[\sqrt{-1}]=R\oplus R\sqrt{-1}$ is a unital associative superalgebra with the superinvolution given by
$$\overline{a\oplus b\sqrt{-1}}=\bar{a}\oplus(-\bar{b}\sqrt{-1})$$
for $a,b\in R$. In this situation, $R\otimes\Bbbk[\sqrt{-1}]$ satisfies Assumption~\ref{asmp_osp22} since one can choose $e=\sqrt{-1}$. For another example, given a unital associative algebra $S$, the associative superalgebra with superinvolution $(S\oplus S^{\mathrm{op}},\mathrm{ex})$ always satisfies Assumption~\ref{asmp_osp22}, where $e=1\oplus(-1)\in S\oplus S^{\mathrm{op}}$ is a required element.
\end{remark}

Our aim is to prove the universality of the central extension $\psi':\widehat{\mathfrak{sto}}_{2|2}(R,{}^-)\rightarrow\mathfrak{sto}_{2|2}(R,{}^-)$ under the above assumption. Since $\psi:\mathfrak{sto}_{2|2}(R,{}^-)\rightarrow\mathfrak{osp}_{2|2}(R,{}^-)$ is also a central extension, the universality of $\psi'$ is equivalent to the universality of $\psi\circ\psi':\widehat{\mathfrak{sto}}_{2|2}(R,{}^-)\rightarrow\mathfrak{osp}_{2|2}(R,{}^-)$, which will be proved below.

Let $(R,{}^-)$ be a unital associative superalgebra with superinvolution satisfying Assumption~\ref{asmp_osp22} and  $\varphi:\mathfrak{E}\rightarrow\mathfrak{osp}_{2|2}(R,{}^-)$ be an arbitrary central extension of $\mathfrak{osp}_{2|2}(R,{}^-)$.
Let $\hat{x}\in\varphi^{-1}(x)$ for $x\in\mathfrak{osp}_{2|2}(R,{}^-)$ and $e\in R$ an element satisfying Assumption~\ref{asmp_osp22}. Then we define
\begin{align}
\tilde{\pi}(a):&=\frac{1}{2}[\hat{f}_{11}(a),\hat{g}_{12}(1)]-\frac{1}{2}[\hat{f}_{11}(ae),\hat{g}_{12}(e^{-1})],\label{eq:osphat22_1}\\
\tilde{t}_{12}(a):&=\frac{1}{2}[\hat{f}_{11}(a),\hat{g}_{12}(1)]+\frac{1}{2}[\hat{f}_{11}(ae),\hat{g}_{12}(e^{-1})],\label{eq:osphat22_2}\\
\tilde{f}_{11}(a):&=[\hat{t}_{12}(1),\hat{f}_{21}(a)],\label{eq:osphat22_3}\\
\tilde{f}_{21}(a):&=-[\hat{t}_{12}(1),\hat{f}_{11}(a)],\label{eq:osphat22_4}\\
\tilde{g}_{11}(a):&=-[\hat{g}_{12}(a),\hat{t}_{12}(1)],\label{eq:osphat22_5}\\
\tilde{g}_{12}(a):&=[\hat{g}_{11}(a),\hat{t}_{12}(1)],\label{eq:osphat22_6}
\end{align}
which are all independent of the representatives $\hat{x}\in\varphi^{-1}(x)$ for $x\in\mathfrak{osp}_{2|2}(R,{}^-)$. It is directly verified that $\tilde{\pi}(a)\in\ker(\varphi)$ for $a\in R$. Moreover, we have

\begin{lemma}
\label{lem:sto22_ct}
$\tilde{\pi}([R,R]R)=0$.
\end{lemma}
\begin{proof}
We denote
\begin{align*}
\tilde{\pi}_1(a):&=[\tilde{f}_{11}(a),\tilde{g}_{12}(1)]-[\tilde{f}_{11}(1),\tilde{g}_{12}(a)],\\
\tilde{\pi}_2(a):&=[\tilde{f}_{11}(a),\tilde{g}_{12}(e^{-1})]-[\tilde{f}_{11}(1),\tilde{g}_{12}(ae^{-1}),
\end{align*}
for $a\in R$. Then $\tilde{\pi}(a)=\tilde{\pi}_1(a)-\tilde{\pi}_2(ae)$. We observe that $\tilde{\pi}([R,R]R)=0$ if $\tilde{\pi}_i([R,R]R)=0$ for $i=1,2$. We will show that $\tilde{\pi}_i([R,R]R)=0$ in two steps:
\begin{itemize}
\item Step 1: Prove that $\tilde{\pi}_i(R_+\cdot R_+\cdot R_+)=0$ for $i=1,2$.
\item Step 2: Show that $[R,R]R\subseteq R_+\cdot R_+\cdot R_+$.
\end{itemize}
\medskip

Step 1: In order to show that $\tilde{\pi}_i(R_+\cdot R_+\cdot R_+)=0$ for $i=1,2$, we claim that
\begin{equation}
[\tilde{f}_{11}(ab),\tilde{g}_{12}(c)]=[\tilde{f}_{11}(a),\tilde{g}_{12}(bc)]\label{eq:osp22_tilde}
\end{equation}
for $a,c\in R$ and $b\in R_+$. Without loss of generality, we assume $a,b,c$ are homogenous and set
$$\tilde{v}(a):=-[\hat{g}_{11}(a),\hat{g}_{11}(1)]\text{ and }\tilde{w}(a):=[\hat{f}_{11}(1),\hat{f}_{11}(a)].$$
Then we deduce that
\begin{align*}
[\tilde{f}_{11}(ab),\tilde{g}_{12}(c)]
&=\frac{1}{2}[\tilde{f}_{11}(ab_+),\tilde{g}_{12}(c)]\\
&=-\frac{1}{2}(-1)^{|a|}[[\tilde{g}_{11}(\bar{a}),\tilde{w}(b)],\tilde{g}_{12}(c)]\\
&=-\frac{1}{2}(-1)^{|a|}[\tilde{g}_{11}(\bar{a}),[\tilde{w}(b),\tilde{g}_{12}(c)]]
+\frac{1}{2}(-1)^{|a||b|}[\tilde{w}(b),[\tilde{g}_{11}(\bar{a}),\tilde{g}_{12}(c)]]\\
&=-\frac{1}{2}(-1)^{|a|+|b|+|c|+|b||c|}[\tilde{g}_{11}(\bar{a}),\tilde{f}_{21}(\bar{c}b_+)]+0\\
&=-(-1)^{|a|+|b|+|c|+|b||c|}[\tilde{g}_{11}(\bar{a}),\tilde{f}_{21}(\bar{c}b)]\\
&=-\frac{1}{2}(-1)^{|b|+|c|+|b||c|}[[\tilde{v}(1),\tilde{f}_{11}(a)],\tilde{f}_{21}(\bar{c}b)]\\
&=-\frac{1}{2}(-1)^{|b|+|c|+|b||c|}([\tilde{v}(1),[\tilde{f}_{11}(a),\tilde{f}_{21}(\bar{c}b)]]
-[\tilde{f}_{11}(a),[\tilde{v}(1),\tilde{f}_{21}(\bar{c}b)]])\\
&=0+(-1)^{|b||c|}[\tilde{f}_{11}(a),\tilde{g}_{12}(\overline{\bar{c}b})]\\
&=[\tilde{f}_{11}(a),\tilde{g}_{12}(bc)].
\end{align*}
This proves the claim.

For $a,b,c\in R_+$, it follows from~(\ref{eq:osp22_tilde}) that
\begin{align*}
\tilde{\pi}_1(abc)&=[\tilde{f}_{11}(abc),\tilde{g}_{12}(1)]-[\tilde{f}_{11}(1),\tilde{g}_{12}(abc)]
=[\tilde{f}_{11}(ab),\tilde{g}_{12}(c)]-[\tilde{f}_{11}(a),\tilde{g}_{12}(bc)]=0,
\end{align*}
which shows $\tilde{\pi}_1(R_+\cdot R_+\cdot R_+)=0$. We also obtain $\tilde{\pi}_2(R_+\cdot R_+\cdot R_+)=0$ similarly. Step 1 is completed.
\medskip

Step 2: We first observe that $R_+\subseteq R_+\cdot R_+\subseteq R_+\cdot R_+\cdot R_+$ since $1\in R_+$. Secondly, we claim that $[R,R]\subseteq R_+\cdot R_+$. Indeed, for homogeneous $a,b\in R$, we have
$$[a,b]=\frac{1}{4}[a_+,b_+]+\frac{1}{4}[a_+,b_-]+\frac{1}{4}[a_-,b_+]+\frac{1}{4}[a_-,b_-],$$
in which $[a_+,b_+]\in R_+\cdot R_+$ and $[a_+,b_-],[a_-,b_+]\in R_+\subseteq R_+\cdot R_+$. Note that $(R,{}^-)$ satisfies Assumption~\ref{asmp_osp22}, we know that $a_-e^{-1},eb_-\in R_+$ and $[a_-,b_-]=[a_-e^{-1},eb_-]\in R_+\cdot R_+$.

In order to show that $[R,R]R\subseteq R_+\cdot R_+\cdot R_+$, it suffices to prove $[a_{\pm},b_{\pm}]c_{\pm}\in R_+\cdot R_+$ for homogeneous $a,b,c\in R$. We observe that
\begin{align*}
[a,b]c&=-(-1)^{|a||b|}[b,a]c,\\
[a,b]c-(-1)^{|a|(|b|+|c|)}[b,c]a&=[a,bc]-(-1)^{|a||b|}[ba,c]\in[R,R]\subseteq R_+\cdot R_+.
\end{align*}
It suffices to show that $[a_+,b_+]c_+, [a_+,b_+]c_-, [a_+,b_-]c_-, [a_-,b_-]c_-\in R_+\cdot R_+\cdot R_+$. It obvious that $[a_+,b_+]c_+\in R_+\cdot R_+\cdot R_+$ and $[a_+,b_-]c_-=[a_+,b_-e^{-1}]ec_-\in R_+\cdot R_+\cdot R_+$. Moreover, since $[a_+,b_+], [a_-,b_-]\in R_-$, we obtain that $[a_+,b_+]c_-, [a_-,b_-]c_-\in R_-\cdot R_-=R_+e^{-1}\cdot eR_+=R_+\cdot R_+$. Hence, $[R,R]R\subseteq R_+\cdot R_+\cdot R_+$. This finishes Step 2.
\end{proof}

Now, we may proceed to prove the main theorem in this section:

\begin{theorem}
\label{thm:osp22_uce}
Let $(R,{}^-)$ be a unital associative superalgebra with superinvolution satisfying Assumption~\ref{asmp_osp22}. Then the central extension $\psi':\widehat{\mathfrak{sto}}_{2|2}(R,{}^-)\rightarrow\mathfrak{sto}_{2|2}(R,{}^-)$ is universal.
\end{theorem}
\begin{proof}
It suffices to show that the central extension $\psi\circ\psi':\widehat{\mathfrak{sto}}_{2|2}(R,{}^-)\rightarrow\mathfrak{osp}_{2|2}(R,{}^-)$ is universal. Let $\varphi:\mathfrak{E}\rightarrow\mathfrak{osp}_{m|2n}(R,{}^-)$ be an arbitrary central extension. We have to show that there is a unique homomorphism $\varphi':\widehat{\mathfrak{sto}}_{2|2}(R,{}^-)\rightarrow\mathfrak{E}$ such that $\varphi\circ\varphi'=\psi\circ\psi'$.

Let $\tilde{\pi}(a)$, $\tilde{t}_{12}(a)$, $\tilde{f}_{11}(a)$, $\tilde{f}_{21}(a)$, $\tilde{g}_{11}(a)$, $\tilde{g}_{12}(a)$ with $a\in R$ be the elements of $\mathfrak{E}$ as defined in (\ref{eq:osphat22_1})-(\ref{eq:osphat22_6}). We have already shown in Lemma~\ref{lem:sto22_ct} that $\tilde{\pi}([R,R]R)=0$. In order to show the existence of a homomorphism $\varphi':\widehat{\mathfrak{sto}}_{2|2}(R,{}^-)\rightarrow\mathfrak{E}$, it suffices to show that the elements $\tilde{t}_{12}(a)$, $\tilde{f}_{11}(a)$, $\tilde{f}_{21}(a)$, $\tilde{g}_{11}(a)$, $\tilde{g}_{12}(a)$ with $a\in R$ satisfy the relations (\ref{eq:hatsto22_02})-(\ref{eq:hatsto22_05}). It has been shown in Lemma~\ref{lem:tilde_rln} that these elements satisfy the relations (\ref{eq:hatsto22_02})-(\ref{eq:hatsto22_04}). We next show that they also satisfy (\ref{eq:hatsto22_05}).

We deduce from the definitions of $\tilde{\pi}_1$, $\tilde{\pi}_2$
and $\tilde{t}_{12}$ that
$$[\tilde{f}_{11}(a),\tilde{g}_{12}(1)]=\tilde{t}_{12}(a)+\tilde{\pi}(a),\text{ and }
[\tilde{f}_{11}(ae),\tilde{g}_{12}(e^{-1})]=\tilde{t}_{12}(a)-\tilde{\pi}(a).$$
Hence, for $a,b\in R$, we write $b=\frac{1}{2}b_++\frac{1}{2}b_-=\frac{1}{2}b_++\frac{1}{2}b_-ee^{-1}$ with $b_{\pm}\in R_{\pm}$ and deduce by (\ref{eq:osp22_tilde}) that
\begin{align*}
[\tilde{f}_{11}(a),\tilde{g}_{12}(b)]&=\frac{1}{2}[\tilde{f}_{11}(a),\tilde{g}_{12}(b_+)]+\frac{1}{2}[\tilde{f}_{11}(a),\tilde{g}_{12}(b_-ee^{-1})]\\
&=\frac{1}{2}[\tilde{f}_{11}(ab_+),\tilde{g}_{12}(1)]+\frac{1}{2}[\tilde{f}_{11}(ab_-e),\tilde{g}_{12}(e^{-1})]\\
&=\frac{1}{2}(\tilde{t}_{12}(ab_+)+\tilde{\pi}(ab_+))
+\frac{1}{2}(\tilde{t}_{12}(ab_-)-\tilde{\pi}(ab_-))\\
&=\tilde{t}_{12}(ab)+\tilde{\pi}(a\bar{b}).
\end{align*}
Furthermore, we deduce that
\begin{align*}
[\tilde{f}_{21}(a),\tilde{g}_{11}(b)]
&=-\frac{1}{2}(-1)^{|a|}[[\tilde{g}_{12}(\bar{a}),\tilde{w}(1)],\tilde{g}_{11}(b)]\\
&=-\frac{1}{2}(-1)^{|a|}[\tilde{g}_{12}(\bar{a}),[\tilde{w}(1),\tilde{g}_{11}(b)]]
+\frac{1}{2}(-1)^{|a|}[\tilde{w}(1),[\tilde{g}_{12}(\bar{a}),\tilde{g}_{11}(b)]]\\
&=-(-1)^{|a|+|b|}[\tilde{g}_{12}(\bar{a}),\tilde{f}_{11}(\bar{b})]\\
&=-(-1)^{|a||b|}[\tilde{f}_{11}(\bar{b}),\tilde{g}_{12}(\bar{a})]\\
&=-(-1)^{|a||b|}(\tilde{t}_{12}(\bar{b}\bar{a})+\tilde{\pi}(\bar{b}a))\\
&=-\tilde{t}_{12}(\overline{ab})-\tilde{\pi}(a\bar{b}).
\end{align*}

In summary, the elements $\tilde{t}_{12}(a),\tilde{f}_{11}(a),\tilde{f}_{21}(a),\tilde{g}_{11}(a),\tilde{g}_{12}(a)$ with $a\in R$ satisfy relations (\ref{eq:hatsto22_02})-(\ref{eq:hatsto22_05}). Therefore, there is a homomorphism of Lie superalgebras
$\varphi':\widehat{\mathfrak{sto}}_{2|2}(R,{}^-)\rightarrow\mathfrak{E}$ such that $\varphi\circ\varphi'=\psi\circ\psi'$. Following the same arguments as in the proof of Theorem~\ref{thm:sto_uce}, such a homomorphism $\varphi'$ is unique. Hence, the central extension $\psi\circ\psi':\widehat{\mathfrak{sto}}_{2|2}(R,{}^-)\rightarrow\mathfrak{osp}_{2|2}(R,{}^-)$ is universal.
\end{proof}

Summarizing Theorems~\ref{thm:osp_ce} and~\ref{thm:osp22_uce}, we obtian
\begin{corollary}
\label{cor:osp22_hml}
Let $(R,{}^-)$ be a unital associative superalgebra with superinvolution satisfying Assumption~\ref{asmp_osp22}. Then
$$\mathrm{H}_2(\mathfrak{osp}_{2|2}(R,{}^-),\Bbbk)=
\fourIdx{}{-}{}{1}{\mathrm{HD}}(R,{}^-)\oplus R/([R,R]R)$$
as $\Bbbk$-modules.\qed
\end{corollary}

In particular, if $S$ is a unital associative superalgebra, then $(S\oplus S^{\mathrm{op}},\mathrm{ex})$ satisfies Assumption~\ref{asmp_osp22} (see Remark~\ref{rmk:osp12_asmp}). In this situation, the universal central extension of $\mathfrak{osp}_{2|2}(S\oplus S^{\mathrm{op}},\mathrm{ex})$ obtained in Theorem~\ref{thm:osp22_uce} recovers the universal central extension of $\mathfrak{sl}_{2|2}(S)$ given in \cite{ChenSun2015}. More precisely,

\begin{corollary}
\label{cor:osp22_SS}
Let $S$ be a unital associative superalgebra. Then
$$\widehat{\mathfrak{st}}_{2|2}(S):=\mathfrak{st}_{2|2}(S)\oplus S/([S,S]S)\oplus S/([S,S]S)$$
is the universal central extension of $\mathfrak{sl}_{2|2}(S)$ and
$$\mathrm{H}_2(\mathfrak{sl}_{2|2}(S),\Bbbk)=\mathrm{HC}_1(S)\oplus S/([S,S]S)\oplus S/([S,S]S).$$
\end{corollary}
\begin{proof}
Let $(R,{}^-):=(S\oplus S^{\mathrm{op}},\mathrm{ex})$. Then it is obvious that
$$R/([R,R]R)\cong S/([S,S]S)\oplus S/([S,S]S)$$
as $\Bbbk$-modules. Hence, the corollary follows from Theorem~\ref{thm:osp22_uce} and the isomorphisms
$$\mathfrak{sto}_{2|2}(S\oplus S^{\mathrm{op}},\mathrm{ex})\cong\mathfrak{st}_{2|2}(S),\text{ and }
\mathfrak{osp}_{2|2}(S\oplus S^{\mathrm{op}},\mathrm{ex})\cong\mathfrak{sl}_{2|2}(S),$$
which have been proved in Proposition~\ref{prop:sto_SS} and Example~\ref{eg:osp_SS}, respectively.
\end{proof}

\section{The universal central extension of $\mathfrak{osp}_{1|2}(R,{}^-)$}
\label{sec:osp12}

The universal central extension of $\mathfrak{osp}_{m|2n}(R,{}^-)$ with $m,n\in\mathbb{N}$ and $(m,n)\neq(1,1)$ has been discussed in the previous sections. For $\mathfrak{osp}_{1|2}(R,{}^-)$, one easily reads from Definition~\ref{def:sto} that $\mathfrak{sto}_{1|2}(R,{}^-)$ is the Lie superalgebra generated by homogenous elements $\mathbf{f}_{11}(a),\mathbf{g}_{11}(a)$ of degree $1+|a|$ for homogenous element $a\in R$ subjecting to the only relation that $\mathbf{f}_{11}$ and $\mathbf{g}_{11}$ are $\Bbbk$-linear. It is not necessarily a central extension of $\mathfrak{osp}_{1|2}(R,{}^-)$. Instead, we will explicitly construct the universal central extension of $\mathfrak{osp}_{1|2}(R,{}^-)$ in this section.

\begin{definition}
\label{def:hatosp12}
Let $(R,{}^-)$ be a unital associative superalgebra with superinvolution. We define the Lie superalgebra $\widehat{\mathfrak{osp}}_{1|2}(R,{}^-)$ by generators and relations: The generators of $\widehat{\mathfrak{osp}}_{1|2}(R,{}^-)$ are homogeneous $\mathbf{v}(a)$, $\mathbf{w}(a)$ of degree $|a|$ and homogeneous $\mathbf{f}(a)$, $\mathbf{g}(a)$ of degree $1+|a|$ for homogenous $a\in R$. The defining relations for $\widehat{\mathfrak{osp}}_{1|2}(R,{}^-)$ are given as follows:
\begin{align*}
&\begin{aligned}
&\mathbf{v}(\bar{a})=\mathbf{v}(a),
&&\mathbf{w}(\bar{a})=\mathbf{w}(a),\\
&[\mathbf{v}(a),\mathbf{v}(b)]=0,
&&[\mathbf{w}(a),\mathbf{w}(b)]=0,\\
&[\mathbf{v}(a),\mathbf{g}(b)]=0,
&&[\mathbf{f}(a),\mathbf{w}(b)]=0,\\
&[\mathbf{v}(a),\mathbf{f}(b)]=(-1)^{|b|}\mathbf{g}(a_+\bar{b}),
&&[\mathbf{g}(a),\mathbf{w}(b)]=-(-1)^{|a|}\mathbf{f}(\bar{a}b_+),\\
&[\mathbf{f}(a),\mathbf{f}(b)]=(-1)^{|a|}\mathbf{w}(\bar{a}b),
&&[\mathbf{g}(a),\mathbf{g}(b)]=-(-1)^{|b|}\mathbf{v}(a\bar{b}),
\end{aligned}\\
&[[\mathbf{f}(a),\mathbf{g}(b)],\mathbf{f}(c)]=\mathbf{f}(abc-\overline{ab}c-(-1)^{|a||b|+|b||c|+|c||a|}cba),\\
&[\mathbf{g}(a),[\mathbf{f}(b),\mathbf{g}(c)]]=\mathbf{g}(abc-a\overline{bc}-(-1)^{|a||b|+|b||c|+|c||a|}cba),
\end{align*}
where $a,b,c\in R$ are homogeneous.
\end{definition}

Observing that the elements
\begin{align}
v(a):&=e_{23}(a_+),& w(a):&=e_{32}(a_+), \label{eq:osp12_gen1}\\
f(a):&=e_{12}(a)+e_{31}(\rho(\bar{a})),&g(a):&=e_{21}(a)-e_{13}(\rho(\bar{a})),\label{eq:osp12_gen2}
\end{align}
of the Lie superalgebra $\mathfrak{osp}_{1|2}(R,{}^-)$ satisfy all relations in Definition~\ref{def:hatosp12}. We obtain a canonical homomorphism of Lie superalgebras
$$\psi:\widehat{\mathfrak{osp}}_{1|2}(R,{}^-)\rightarrow \mathfrak{osp}_{1|2}(R,{}^-)$$
such that
$$\psi(\mathbf{v}(a))={v}(a),\quad\psi(\mathbf{w}(a))={w}(a),
\quad\psi(\mathbf{f}(a))={f}(a),\quad\psi(\mathbf{g}(a))={g}(a).$$

In the especial case where $(R,{}^-)=(S\oplus S^{\mathrm{op}},\mathrm{ex})$ with $S$ an arbitrary unital associative superalgebra, the Lie superalgebra $\widehat{\mathfrak{osp}}_{1|2}(S\oplus S^{\mathrm{op}},\mathrm{ex})$ also recovers the Steinberg Lie superalgebra $\mathfrak{st}_{1|2}(S)$:

\begin{proposition}
\label{prop:hatosp_SS}
Let $S$ be an arbitrary unital associative superalgebra. Then
$$\widehat{\mathfrak{osp}}_{1|2}(S\oplus S^{\mathrm{op}},\mathrm{ex})\cong\mathfrak{st}_{1|2}(S)$$
as Lie superalgebras.
\end{proposition}
\begin{proof}
An isomorphism $\phi:\widehat{\mathfrak{osp}}_{1|2}(S\oplus S^{\mathrm{op}},\mathrm{ex})\rightarrow\mathfrak{st}_{1|2}(S)$ is given by
\begin{align*}
\phi(\mathbf{v}(a\oplus b))&=\boldsymbol{e}_{23}(a+b),
&\phi(\mathbf{w}(a\oplus b))&=\boldsymbol{e}_{32}(a+b),\\
\phi(\mathbf{f}(a\oplus b))&=\boldsymbol{e}_{12}(a)+\boldsymbol{e}_{31}(\rho(b)),
&\phi(\mathbf{g}(a\oplus b))&=\boldsymbol{e}_{21}(a)-\boldsymbol{e}_{13}(\rho(b)),
\end{align*}
for $a,b\in S$. While the inverse $\phi^{-1}:\mathfrak{st}_{1|2}(S)\rightarrow\widehat{\mathfrak{osp}}_{1|2}(S\oplus S^{\mathrm{op}},\mathrm{ex})$ of $\phi$ is defined by
\begin{align*}
\phi^{-1}(\boldsymbol{e}_{12}(a))&=\mathbf{f}(a\oplus0),
&\phi^{-1}(\boldsymbol{e}_{13}(a))&=-\mathbf{g}(0\oplus\rho(a)),\\
\phi^{-1}(\boldsymbol{e}_{21}(a))&=\mathbf{g}(a\oplus0),
&\phi^{-1}(\boldsymbol{e}_{23}(a))&=\mathbf{v}(a\oplus0),\\
\phi^{-1}(\boldsymbol{e}_{31}(a))&=\mathbf{f}(0\oplus\rho(a)),
&\phi^{-1}(\boldsymbol{e}_{32}(a))&=\mathbf{w}(a\oplus0),
\end{align*}
for $a\in S$.
\end{proof}

In the sequel, we will show that $\psi:\widehat{\mathfrak{osp}}_{1|2}(R,{}^-)\rightarrow\mathfrak{osp}_{1|2}(R,{}^-)$ is a central extension, characterize the kernel of $\psi$, and construct the universal central extension of $\widehat{\mathfrak{osp}}_{1|2}(R,{}^-)$.

\subsection{$\psi:\widehat{\mathfrak{osp}}_{1|2}(R,{}^-)\rightarrow\mathfrak{osp}_{1|2}(R,{}^-)$ is a central extension}
\label{subsubsec:osp12_ce}

In order to show that $\psi:\widehat{\mathfrak{osp}}_{1|2}(R,{}^-)\rightarrow\mathfrak{osp}_{1|2}(R,{}^-)$ is a central extension, we need a few lemmas:

\begin{lemma}
\label{lem:hatosp12_ele}
In the Lie superalgebra $\widehat{\mathfrak{osp}}_{1|2}(R,{}^-)$, we denote
$$\mathbf{h}(a,b):=[\mathbf{f}(a),\mathbf{g}(b)],\text{ and }\quad \boldsymbol{\lambda}(a,b):=\mathbf{h}(a,b)-(-1)^{|a||b|}\mathbf{h}(1,ba),$$
for homogeneous $a,b\in R$. Then every element $x$ of $\widehat{\mathfrak{osp}}_{1|2}(R,{}^-)$ is written as
$$x=\sum_i\boldsymbol{\lambda}(a_i,b_i)+\mathbf{h}(1,c_0)+\mathbf{f}(c_1)+\mathbf{g}(c_2)
+\mathbf{v}(c_3)+\mathbf{w}(c_4),$$
where the summation runs over a finite set, $a_i,b_i,c_0,c_1,c_2\in R$, and $c_3,c_4\in R_+$.
\end{lemma}
\begin{proof}
We first claim that $\widehat{\mathfrak{osp}}_{1|2}(R,{}^-)$ (as a $\Bbbk$-module) is spanned by
$$\mathfrak{B}:=\{\mathbf{h}(a,b),\mathbf{f}(a),\mathbf{g}(a),\mathbf{v}(a),\mathbf{w}(a)| \text{homogeneous }a,b\in R\}.$$
Let $\mathfrak{g}$ be the $\Bbbk$--submodule of $\widehat{\mathfrak{osp}}_{1|2}(R,{}^-)$ spanned by $\mathfrak{B}$. Then it is directly verified that $\mathfrak{g}$ is indeed a Lie sub-superalgebra of $\widehat{\mathfrak{osp}}_{1|2}(R,{}^-)$, containing all the generators $\mathbf{v}(a)$, $\mathbf{w}(a)$, $\mathbf{f}(a)$ and $\mathbf{g}(a)$ for $a\in R$. Hence, $\widehat{\mathfrak{osp}}_{1|2}(R,{}^-)=\mathfrak{g}$, i.e., every element $x\in\widehat{\mathfrak{osp}}_{1|2}(R,{}^-)$ is written as
$$x=\sum_i\mathbf{h}(a_i,b_i)+\mathbf{f}(c_1)+\mathbf{g}(c_2)+\mathbf{v}(c_3)+\mathbf{w}(c_4),$$
where the summation runs over a finite set, $a_i,b_i,c_1,c_2\in R$, and $c_3,c_4\in R_+$.

Now,
$$\mathbf{h}(a_i,b_i)=\boldsymbol{\lambda}(a_i,b_i)+(-1)^{|a_i||b_i|}\mathbf{h}(1,b_ia_i).$$
We conclude that
$$x=\sum_i\boldsymbol{\lambda}(a_i,b_i)+\mathbf{h}(1,c_0)+\mathbf{f}(c_1)+\mathbf{g}(c_2)+\mathbf{v}(c_3)+\mathbf{w}(c_4),$$
where $c_0=\sum_i(-1)^{|a_i||b_i|}b_ia_i$.
\end{proof}

\begin{lemma}
\label{lem:hatosp12_ker1}
$\ker\psi=\{\textstyle{\sum_i}\boldsymbol{\lambda}(a_i,b_i)|\textstyle{\sum_i}[a_i,b_i]=\textstyle{\sum_i}\overline{[a_i,b_i]}\}.$
\end{lemma}
\begin{proof}
If $\sum_i[a_i,b_i]=\sum_i\overline{[a_i,b_i]}$ for $a_i,b_i\in R$, then
$$\psi\left(\sum_i\boldsymbol{\lambda}(a_i,b_i)\right)=e_{11}\left(\sum_i([a_i,b_i]-\overline{[a_i,b_i]})\right)=0.$$
Conversely, we have known from Lemma~\ref{lem:osp_ele} that every element $x\in\widehat{\mathfrak{osp}}_{1|2}(R,{}^-)$ is written as
\begin{align*}
x&=\sum_i\boldsymbol{\lambda}(a_i,b_i)+\mathbf{h}(1,c_0)+\mathbf{f}(c_1)+\mathbf{g}(c_2)
+\mathbf{v}(c_3)+\mathbf{w}(c_4),
\end{align*}
for homogeneous $a_i,b_i,c_0,c_1,c_2\in R$ and $c_3,c_4\in R_+$. Hence, $x\in\ker\psi$ implies that
\begin{align*}
0=\psi(x)=&\sum_ie_{11}([a_i,b_i]-\overline{[a_i,b_i]})+e_{11}(c_0-\bar{c}_0)+e_{22}(\rho(c_0))-e_{33}(\rho(\bar{c}_0))\\
&+{f}(c_1)+{g}(c_2)+{v}(c_3)+{w}(c_4)
\end{align*}
in $\mathfrak{osp}_{1|2}(R,{}^-)$. It follows that $c_0=c_1=c_2=0$ and $c_3+\bar{c}_3=c_4+\bar{c}_4=0$. Since $c_3,c_4\in R_+$ and $\frac{1}{2}\in\Bbbk$, we obtain that $c_3=c_4=0$. Hence, $x=\sum_i\boldsymbol{\lambda}(a_i,b_i)$ and $\sum_i [a_i,b_i]=\sum_i\overline{[a_i,b_i]}$.
\end{proof}

\begin{proposition}
\label{prop:osp12_ce}
$\psi:\widehat{\mathfrak{osp}}_{1|2}(R,{}^-)\rightarrow\mathfrak{osp}_{1|2}(R,{}^-)$ is a central extension.
\end{proposition}
\begin{proof}
It suffices to show that $\ker\psi$ is included in the center of $\widehat{\mathfrak{osp}}_{1|2}(R,{}^-)$. We directly compute that
\begin{align}
[\boldsymbol{\lambda}(a,b),\mathbf{v}(c)]&=0,
&[\boldsymbol{\lambda}(a,b),\mathbf{w}(c)]&=0,\label{eq:osp12_h1}\\
[\boldsymbol{\lambda}(a,b),\mathbf{f}(c)]&=\mathbf{f}(([a,b]-\overline{[a,b]})c),
&[\mathbf{g}(a),\boldsymbol{\lambda}(b,c)]&=\mathbf{g}(a([b,c]-\overline{[b,c]})),\label{eq:osp12_h2}
\end{align}
for homogeneous $a,b,c\in R$. Hence, $x=\sum_i\boldsymbol{\lambda}(a_i,b_i)$ super-commutes with $\mathbf{v}(c)$, $\mathbf{w}(c)$, $\mathbf{f}(c)$ and $\mathbf{g}(c)$ if $\sum_i[a_i,b_i]=\sum_i\overline{[a_i,b_i]}$.

Recall from Lemma~\ref{lem:hatosp12_ker1} that $\ker\psi=\{\sum_i\boldsymbol{\lambda}(a_i,b_i)|\sum_i[a_i,b_i]=\sum_i\overline{[a_i,b_i]}\}$. It follows that $\ker\psi$ is included in the center of $\widehat{\mathfrak{osp}}_{1|2}(R,{}^-)$ and hence $\psi:\widehat{\mathfrak{osp}}_{1|2}(R,{}^-)\rightarrow\mathfrak{osp}_{1|2}(R,{}^-)$ is a central extension.
\end{proof}

\subsection{A modified version of the first $\mathbb{Z}/2\mathbb{Z}$-graded skew-dihedral homology}
\label{subsubsec:hmltilde}

In order to characterize $\ker\psi$, we need a modified version of the first $\mathbb{Z}/2\mathbb{Z}$-graded skew-dihedral homology ~\cite[Section 6]{ChangWang2015}, whose properties will be briefly discussed here.

Let $(R,{}^-)$ be an arbitrary unital associative superalgebra with superinvolution. In the $\Bbbk$-module $R\otimes_{\Bbbk}R$, we denote
$$J(a,b,c):=(-1)^{|a||c|}ab\otimes c+(-1)^{|b||a|}bc\otimes a+(-1)^{|c||b|}ca\otimes b$$
for homogeneous $a,b,c\in R$. Let $\tilde{I}_{\mathsf{d}}^{-}$ be the $\Bbbk$-submodule of $R\otimes_{\Bbbk}R$ spanned by
\begin{enumerate}
\item $a\otimes 1$,
\item $a\otimes b+(-1)^{|a||b|}b\otimes a$,
\item $a\otimes b-\bar{a}\otimes\bar{b}$,
\item $J(b,a,c)-(-1)^{|a||b|+|b||c|+|c||a|}J(a,b,c-\bar{c})$,
\item $J(a,b,[c,d])+(-1)^{|a||c|+|b||d|}J(c,d,[a,b])$, and
\item $J([a_1,a_2],[b_1,b_2],[c_1,c_2])$,
\end{enumerate}
for homogeneous $a,b,c,d,a_1,a_2,b_1,b_2,c_1,c_2\in R$ and $\tildeHD{R}{R}:=R/\tilde{I}_{\mathsf{d}}^{-}$. We use $\tildeHD{a}{b}$ to denote the canonical image of $a\otimes b$ in $\tildeHD{R}{R}$ and define the first modified $\mathbb{Z}/2\mathbb{Z}$-graded skew-dihedral homology group
\begin{align*}
\fourIdx{}{-}{}{1}{\widetilde{\mathrm{HD}}}(R,{}^-)=
\left\{\textstyle{\sum_i}\tildeHD{a}{b}\in\tildeHD{R}{R}\middle|\textstyle{\sum_i}[a_i,b_i]=\textstyle{\sum_i}\overline{[a_i,b_i]}\right\}.
\end{align*}

\begin{remark}
\label{rmk:tildeHD_com}
In the especial case $R$ is a unital super-commutative associative superalgebra and the identity map is the superinvolution on $R$, we observe that $\tilde{I}_d^-$ is the $\Bbbk$-submodule of $R\otimes_{\Bbbk}R$ spanned by $a\otimes1$, $a\otimes b+(-1)^{|a||b|}b\otimes a$ and $J(a,b,c)$ for homogeneous $a,b,c\in R$. Hence,
$$\fourIdx{}{-}{}{1}{\widetilde{\mathrm{HD}}}(R,\mathrm{id})=\mathrm{HC}_1(R).$$
\end{remark}

\begin{lemma}
\label{lem:hmltilde}
The first modified $\mathbb{Z}/2\mathbb{Z}$-graded skew-dihedral homology group $\fourIdx{}{-}{}{1}{\widetilde{\mathrm{HD}}}(R,{}^-)$ satisfies:
\begin{enumerate}
\item If $3$ is invertible in $R$, then
$$\fourIdx{}{-}{}{1}{\widetilde{\mathrm{HD}}}(R,{}^-)\cong \fourIdx{}{-}{}{1}{\mathrm{HD}}(R,{}^-).$$
\item If $(R,{}^-)=(S\oplus S^{\mathrm{op}},\mathrm{ex})$ for an arbitrary unital associative superalgebra $S$, then
$$\fourIdx{}{-}{}{1}{\widetilde{\mathrm{HD}}}(S\oplus S^{\mathrm{op}},\mathrm{ex})\cong\mathrm{HC}_1(S).$$
\end{enumerate}
\end{lemma}
\begin{proof}
(i) It follows from \cite[Proposition~6.3]{ChangWang2015} that
\begin{align*}
\fourIdx{}{-}{}{1}{\mathrm{HD}}(R,{}^-)=
\left\{\textstyle{\sum_i}\langle a_i,b_i\rangle_{\mathsf{d}}^-\in\langle R,R\rangle_{\mathsf{d}}^-\middle|\textstyle{\sum_i}[a_i,b_i]=\textstyle{\sum_i}\overline{[a_i,b_i]}\right\},
\end{align*}
where $\langle R,R\rangle_{\mathsf{d}}^-=(R\otimes_{\Bbbk}R)/I_{\mathsf{d}}^-$ and $I_{\mathsf{d}}^-$ is the $\Bbbk$-submodule of $R\otimes_{\Bbbk}R$ spanned by
$$a\otimes b+(-1)^{|a||b|}b\otimes a,\quad a\otimes b-\bar{a}\otimes\bar{b}, \text{ and }J(a,b,c)$$
for homogeneous $a,b,c\in R$. It is obvious that $\tilde{I}_{\mathsf{d}}^-\subseteq I_{\mathsf{d}}^-$.

Now, $J(b,a,c)-(-1)^{|a||b|+|b||c|+|c||a|}J(a,b,c-\bar{c})\in\tilde{I}_{\mathsf{d}}^-$ implies that $$J(a,b,c)\equiv-J(a,b,\bar{c})\pmod{\tilde{I}_{\mathsf{d}}^-},$$
and hence
\begin{align*}
J(a,b,c)
\equiv2(-1)^{|a||b|+|b||c|+|c||a|}J(b,a,c)
\equiv4J(a,b,c)\pmod{\tilde{I}_{\mathsf{d}}^-},
\end{align*}
i.e., $3J(a,b,c)\equiv0\pmod{\tilde{I}_{\mathsf{d}}^-}$.

If $3$ is invertible in $R$, then $J(a,b,c)\in\tilde{I}_{\mathsf{d}}^-$ and hence $I_{\mathsf{d}}^-\subseteq\tilde{I}_{\mathsf{d}}^-$. It follows that
$I_{\mathsf{d}}^-=\tilde{I}_{\mathsf{d}}^-$ and $\fourIdx{}{-}{}{1}{\widetilde{\mathrm{HD}}}(R,{}^-)=\fourIdx{}{-}{}{1}{\mathrm{HD}}(R,{}^-)$.
\medskip

(ii) For homogeneous $a,b\in S$, we first claim that
$$\tildeHD{a\oplus0}{0\oplus b}=0$$
in $\tildeHD{S\oplus S^{\mathrm{op}}}{S\oplus S^{\mathrm{op}}}$. Indeed, one easily deduces from
$$J(a\oplus0,1\oplus0,0\oplus b)=-(-1)^{|a||b|}J(1\oplus0,a\oplus0,0\oplus b)$$
that $2\tildeHD{a\oplus0}{0\oplus b}=0$, which yields $\tildeHD{a\oplus0}{0\oplus b}=0$ since $\frac{1}{2}\in\Bbbk$.

Consider the $\Bbbk$-linear map $\theta:S\otimes_{\Bbbk}S\rightarrow \tildeHD{S\oplus S^{\mathrm{op}}}{S\oplus S^{\mathrm{op}}},\quad
a\otimes b\mapsto \tildeHD{a\oplus0}{b\oplus0}$. We compute that
\begin{eqnarray*}
&\theta(a\otimes b+(-1)^{|a||b|}b\otimes a)=0,&\\
&\theta((-1)^{|a||c|}ab\otimes c+(-1)^{|b||a|}bc\otimes a+(-1)^{|c||b|}ca\otimes b)=0.&
\end{eqnarray*}
Hence, $\theta$ induces a $\Bbbk$-linear map $\tilde{\theta}:\langle S,S\rangle_{\mathsf{c}}\rightarrow
\tildeHD{S\oplus S^{\mathrm{op}}}{S\oplus S^{\mathrm{op}}}$ such that
$$\langle a,b\rangle_{\mathsf{c}}\mapsto\tildeHD{a\oplus0}{b\oplus0}.$$

Conversely, we consider the $\Bbbk$-linear map
$$\vartheta:(S\oplus S^{\mathrm{op}})\otimes(S\oplus S^{\mathrm{op}})\rightarrow \langle S,S\rangle_{\mathsf{c}},\quad
(a_1\oplus a_2)\otimes(b_1\oplus b_2)\mapsto \langle a_1,b_1\rangle_{\mathsf{c}}+
\langle a_2,b_2\rangle_{\mathsf{c}}.$$
Then it is directly verified that $\tilde{I}_{\mathsf{d}}^-\subseteq\ker\vartheta$. Hence, $\vartheta$ induces a $\Bbbk$-linear map $$\tilde{\vartheta}:\tildeHD{S\oplus S^{\mathrm{op}}}{S\oplus S^{\mathrm{op}}}\rightarrow\langle S,S\rangle_{\mathsf{c}}$$ such that
$$\tildeHD{a_1\oplus a_2}{b_1\oplus b_2}\mapsto
\langle a_1,b_1\rangle_{\mathsf{c}}+\langle a_2,b_2\rangle_{\mathsf{c}}.$$

Moreover, we observe that $\tilde{\vartheta}\circ \tilde{\theta}=\mathrm{id}$ on $\langle S,S\rangle_{\mathsf{c}}$, while $\tildeHD{a\oplus0}{0\oplus b}=0$ implies that $\tilde{\theta}\circ\tilde{\vartheta}=\mathrm{id}$ on $\tildeHD{S\oplus S^{\mathrm{op}}}{S\oplus S^{\mathrm{op}}}$. Therefore, $\langle S,S\rangle_{\mathsf{c}}$ is isomorphic to $\tildeHD{S\oplus S^{\mathrm{op}}}{S\oplus S^{\mathrm{op}}}$ and we conclude that $\mathrm{HC}_1(S)$ is isomorphic to $\fourIdx{}{-}{}{1}{\widetilde{\mathrm{HD}}}(S\oplus S^{\mathrm{op}},\mathrm{ex})$.
\end{proof}

\subsection{The kernel of $\psi$}
\label{subsubsec:osp12_ker}

We have already known from Lemma~\ref{lem:hatosp12_ker1} that $\ker\psi=\{\textstyle{\sum}_i\boldsymbol{\lambda}(a_i,b_i)|
\textstyle{\sum}_i[a_i,b_i]=\textstyle{\sum}_i\overline{[a_i,b_i]}\}$. We will show in addition that $\ker\psi$ can be identified with $\fourIdx{}{-}{}{1}{\widetilde{\mathrm{HD}}}(R,{}^-)$.

\begin{lemma}
\label{lem:hatosp12_h}
In the Lie superalgebra $\widehat{\mathfrak{osp}}_{1|2}(R,{}^-)$, the following equalities hold:
\begin{enumerate}
\item $\boldsymbol{\lambda}(a,1)=\boldsymbol{\lambda}(1,a)=0$,
\item $\boldsymbol{\lambda}(a,b)=\boldsymbol{\lambda}(\bar{a},\bar{b})$,
\item $\boldsymbol{\lambda}(a,b)=-(-1)^{|a||b|}\boldsymbol{\lambda}(b,a)$,
\end{enumerate}
for homogeneous $a,b\in R$. We further denote
$$\mathbf{J}(a,b,c):=(-1)^{|a||c|}\boldsymbol{\lambda}(ab,c)
+(-1)^{|b||a|}\boldsymbol{\lambda}(bc,a)+(-1)^{|c||b|}\boldsymbol{\lambda}(ca,b)$$
for homogeneous $a,b,c\in R$, then they satisfy the following additional equalities:
\begin{enumerate}
\setcounter{enumi}{3}
\item $\mathbf{J}(b,a,c)=(-1)^{|a||b|+|b||c|+|c||a|}\mathbf{J}(a,b,c-\bar{c})$,
\item $\mathbf{J}(a,b,[c,d])=-(-1)^{|a||c|+|b||d|}\mathbf{J}(c,d,[a,b])$, and
\item $\mathbf{J}([a_1,a_2],[b_1,b_2],[c_1,c_2])=0$.
\end{enumerate}
where $a,b,c,d,a_1,a_2,b_1,b_2,c_1,c_2\in R$ are homogeneous.
\end{lemma}
\begin{proof}
(i) $\boldsymbol{\lambda}(1,a)=0$ is trivial. In order to show $\boldsymbol{\lambda}(a,1)=0$, we first claim that
$$[\mathbf{h}(1,1),\mathbf{h}(a,b)]=0$$
for homogeneous $a,b\in R$. Indeed,
\begin{align*}
[\mathbf{h}(1,1),\mathbf{h}(a,b)]
&=[\mathbf{h}(1,1),[\mathbf{f}(a),\mathbf{g}(b)]]\\
&=[[\mathbf{h}(1,1),\mathbf{f}(a)],\mathbf{g}(b)]
+[\mathbf{f}(a),[\mathbf{h}(1,1),\mathbf{g}(b)]\\
&=-[\mathbf{f}(a),\mathbf{g}(b)]+[\mathbf{f}(a),\mathbf{g}(b)]\\
&=0.
\end{align*}
Now, we compute that
\begin{align*}
[\mathbf{h}(1,1),\mathbf{h}(a,b)]
=[\mathbf{h}(1,1),[\mathbf{f}(a),\mathbf{g}(b)]]
=0.
\end{align*}
It follows that
\begin{align*}
[\mathbf{f}(a),\mathbf{g}(1)]
=-[[\mathbf{h}(\bar{a},1),\mathbf{f}(1)],\mathbf{g}(1)]
=-[\mathbf{h}(\bar{a},1),\mathbf{h}(1,1)]+[\mathbf{f}(1),\mathbf{g}(a)]
=[\mathbf{f}(1),\mathbf{g}(a)],
\end{align*}
i.e., $\boldsymbol{\lambda}(a,1)=0$.
\medskip

(ii) We calculate that
\begin{align*}
[\mathbf{h}(\bar{a},1),\mathbf{h}(1,b)]
&=[\mathbf{h}(\bar{a},1),[\mathbf{f}(1),\mathbf{g}(b)]]
=-\boldsymbol{\lambda}(a,b)+[\mathbf{f}(1),\mathbf{g}([\bar{a},b]),\\
[\mathbf{h}(\bar{a},1),\mathbf{h}(1,b)]
&=[[\mathbf{f}(\bar{a}),\mathbf{g}(1)],\mathbf{h}(1,b)]
=-\boldsymbol{\lambda}(\bar{a},\bar{b})+[\mathbf{f}([\bar{a},b]),\mathbf{g}(1)].
\end{align*}
It follows that $\boldsymbol{\lambda}(a,b)-\boldsymbol{\lambda}(\bar{a},\bar{b})=
[\mathbf{f}(1),\mathbf{g}([\bar{a},b])-[\mathbf{f}([\bar{a},b]),\mathbf{g}(1)]
=-\boldsymbol{\lambda}([\bar{a},b],1)=0$ by (i). Hence, (ii) holds.
\medskip

(iii) We have obtained from (i) that $\mathbf{h}(a,1)=\mathbf{h}(1,a)$. By (ii), we deduce that
\begin{align*}
\boldsymbol{\lambda}(a,b)
&=-[\mathbf{h}(\bar{a},1),\mathbf{h}(1,b)]+[\mathbf{f}(1),\mathbf{g}([\bar{a},b])]\\
&=-[\mathbf{h}(1,\bar{a}),\mathbf{h}(b,1)]+[\mathbf{f}(1),\mathbf{g}([\bar{a},b])]\\
&=(-1)^{|a||b|}[\mathbf{h}(b,1),\mathbf{h}(1,\bar{a})]-(-1)^{|a||b|}[\mathbf{f}(1),\mathbf{g}([b,\bar{a}])]\\
&=-(-1)^{|a||b|}\boldsymbol{\lambda}(\bar{b},\bar{a})\\
&=-(-1)^{|a||b|}\boldsymbol{\lambda}(b,a).
\end{align*}
This proves (iii).
\medskip

(iv) For homogeneous $a,b,c,d\in R$, we consider the equality
\begin{align*}
[[\mathbf{f}(a),\mathbf{g}(b)],\mathbf{h}(c,d)]
=[\mathbf{h}(a,b),\mathbf{h}(c,d)]
=[\mathbf{h}(a,b),[\mathbf{f}(c),\mathbf{g}(d)]].
\end{align*}
Expanding the left and right hand sides of the above equality using the Jacobi identity and applying (ii) and (iii), we deduce that
\begin{equation}
\begin{aligned}
&\mathbf{J}(ab-\overline{ab},c,d)
+(-1)^{|a||c|+|b||d|}\mathbf{J}(cd-\overline{cd},a,b)\\
&=(-1)^{|a||b|+|a||c|+|a||d|+|b||c|+|b||d|+|c||d|}\mathbf{J}(c,ba,d)+(-1)^{|a||b|+|a||d|+|b||c|+|c||d|}\mathbf{J}(a,dc,b).
\end{aligned}
\label{eq:hatosp12_h1}
\end{equation}
We observe that $\mathbf{J}(a,b,c)=\mathbf{J}(b,c,a)=\mathbf{J}(c,a,b)$ and $\mathbf{J}(a,b,1)=0$. Taking $d=1$ in (\ref{eq:hatosp12_h1}), we obtain $\mathbf{J}(b,a,c)=(-1)^{|a||b|+|b||c|+|c||a|}\mathbf{J}(a,b,c-\overline{c})$, which is (iv).
\medskip

(v) We may deduce from (iv) that:
\begin{eqnarray}
\mathbf{J}(a,b,c)&=&-\mathbf{J}(a,b,\bar{c}),\label{eq:hatosp12_h2}\\
\mathbf{J}(a,b,c)&=&2(-1)^{|a||b|+|b||c|+|c||a|}\mathbf{J}(b,a,c),\label{eq:hatosp12_h3}\\
3\mathbf{J}(a,b,c)&=&0.\label{eq:hatosp12_h4}
\end{eqnarray}
Indeed, (\ref{eq:hatosp12_h2}) follows from (iv) by replacing $c$ with $c+\bar{c}$. Then (iv) and (\ref{eq:hatosp12_h2}) together imply (\ref{eq:hatosp12_h3}), which further yields (\ref{eq:hatosp12_h4}) easily.

Now, applying (\ref{eq:hatosp12_h2})-(\ref{eq:hatosp12_h4}) to (\ref{eq:hatosp12_h1}), we deduce that
$$2\mathbf{J}(ab,c,d)
+2(-1)^{|a||c|+|b||d|}\mathbf{J}(cd,a,b)=2(-1)^{|a||b|}\mathbf{J}(ba,c,d)+2(-1)^{|a||c|+|b||d|+|c||d|}\mathbf{J}(dc,a,b),$$
which yields that $\mathbf{J}(a,b,[c,d])=-(-1)^{|a||c|+|b||d|}\mathbf{J}(c,d,[a,b])$.
\medskip

(vi) We first claim that $\mathbf{J}(a,b,c), a,b,c\in R$ is contained in the center of $\widehat{\mathfrak{osp}}_{1|2}(R,{}^-)$. Indeed,
it follows from (\ref{eq:osp12_h2}) that
\begin{align*}
[\mathbf{J}(a,b,c),\mathbf{f}(d)]=0,\text{ and }[\mathbf{J}(a,b,c),\mathbf{g}(d)]=0.
\end{align*}
Note that $\mathbf{f}(d)$ and $\mathbf{g}(d)$ with $d\in R$ generate the Lie superalgebra $\widehat{\mathfrak{osp}}_{1|2}(R,{}^-)$. Hence, $\mathbf{J}(a,b,c)$ is contained in the center of $\widehat{\mathfrak{osp}}_{1|2}(R,{}^-)$.

For homogeneous $a_1,a_2,b_1,b_2\in R$, a direct computation shows that
\begin{align*}
[\boldsymbol{\lambda}(a_1,a_2),\boldsymbol{\lambda}(b_1,b_2)]
=2(-1)^{(|a_1|+|a_2|)|b_2|}\mathbf{J}([a_1,a_2],b_1,b_2)
+\boldsymbol{\lambda}([a_1,a_2],[b_1,b_2])
-\boldsymbol{\lambda}(\overline{[a_1,a_2]},[b_1,b_2]).
\end{align*}
Since $\mathbf{J}(a,b,c)$ lies in the center of the Lie superalgebra $\widehat{\mathfrak{osp}}_{1|2}(R,{}^-)$, we further deduce that
\begin{align*}
&[[\boldsymbol{\lambda}(a_1,a_2),\boldsymbol{\lambda}(b_1,b_2)],\boldsymbol{\lambda}(c_1,c_2)]
+4(-1)^{(|a_1|+|a_2|)(|c_1|+|c_2|)}\mathbf{J}([a_1,a_2],[b_1,b_2],[c_1,c_2])\\
&=\boldsymbol{\lambda}([[a_1,a_2],[b_1,b_2]],[c_1,c_2])
-\boldsymbol{\lambda}([[a_1,a_2],[b_1,b_2]],\overline{[c_1,c_2]})\\
&\quad-\boldsymbol{\lambda}([[a_1,a_2],\overline{[b_1,b_2]}],[c_1,c_2])
-\boldsymbol{\lambda}([\overline{[a_1,a_2]},[b_1,b_2]],[c_1,c_2]).
\end{align*}
Hence, (vi) follows from (\ref{eq:hatosp12_h2})-(\ref{eq:hatosp12_h4}) and the Jacobi identify
\begin{align*}
0=&(-1)^{(|a_1|+|a_2|)(|c_1|+|c_2|)}[[\boldsymbol{\lambda}(a_1,a_2),\boldsymbol{\lambda}(b_1,b_2)],\boldsymbol{\lambda}(c_1,c_2)]\\
&+(-1)^{(|b_1|+|b_2|)(|a_1|+|a_2|)}[[\boldsymbol{\lambda}(b_1,b_2),\boldsymbol{\lambda}(c_1,c_2)],\boldsymbol{\lambda}(a_1,a_2)]\\
&+(-1)^{(|c_1|+|c_2|)(|b_1|+|b_2|)}[[\boldsymbol{\lambda}(c_1,c_2),\boldsymbol{\lambda}(a_1,a_2)],\boldsymbol{\lambda}(b_1,b_2)].
\end{align*}
This completes the proof.
\end{proof}

Now, we prove the following proposition:

\begin{proposition}
\label{prop:hatosp12_ker}
$\ker\psi\cong\fourIdx{}{-}{}{1}{\widetilde{\mathrm{HD}}}(R,{}^-)$ as $\Bbbk$-modules.
\end{proposition}
\begin{proof}
By Lemma~\ref{lem:hatosp12_h}, there is a canonical $\Bbbk$-linear map
$$\eta:\tildeHD{R}{R}\rightarrow\widehat{\mathfrak{osp}}_{1|2}(R,{}^-),\quad \tildeHD{a}{b}\mapsto\mathbf{h}(a,b).$$
It follows from Lemma~\ref{lem:hatosp12_ker1} that $\eta$ takes the $\Bbbk$-submodule $\fourIdx{}{-}{}{1}{\widetilde{\mathrm{HD}}}(R,{}^-)\subseteq\tildeHD{R}{R}$ onto $\ker\psi$. It suffices to show that the restriction of $\eta$ on $\fourIdx{}{-}{}{1}{\widetilde{\mathrm{HD}}}(R,{}^-)$ is injective. This will be done through creating a central extension of $\mathfrak{osp}_{m|2n}(R,{}^-)$ with center $\tildeHD{R}{R}$.

Recall that  $\mathfrak{osp}_{1|2}(R,{}^-)$ (as a $\Bbbk$-module) is spanned by
$$\mathfrak{B}:=\{t(a),f(a),g(a),e_{11}(b),e_{23}(c),e_{32}(c)|
a\in R, b\in [R,R]\cap R_-, c\in R_+\},$$
where $t(a):=e_{11}(a-\bar{a})+e_{22}(\rho(a))-e_{33}(\rho(\bar{a}))$ and $f(a), g(a)$ are given in (\ref{eq:osp12_gen2}).

We define a $\Bbbk$-bilinear map
$$\alpha:\mathfrak{osp}_{1|2}(R)\times\mathfrak{osp}_{1|2}(R)\rightarrow\tildeHD{R}{R}$$
as follows:
\begin{align*}
&\alpha(f(a), g(b))=-(-1)^{(1+|a|)(1+|b|)}\alpha(g(b),f(a))=\tildeHD{a}{b},&&a,b\in R,\\
&\alpha(e_{23}(a),e_{32}(b))=-(-1)^{|a||b|}\alpha(e_{32}(b),e_{23}(a))=-\frac{1}{2}(-1)^{|a|+|b|}\tildeHD{a}{b},&&a,b\in R_+,\\
&\alpha(t(a),t(b))=\tildeHD{a}{\bar{b}},&&a,b\in R,\\
&\alpha(t(a),e_{11}(b))=-(-1)^{|a||b|}\alpha(e_{11}(b),t(a))=\tildeHD{a}{b},&&a\in R, b\in R_-\cap[R,R],\\
&\alpha(e_{11}(a),e_{11}(b))=\frac{1}{2}\tildeHD{a}{b}-\sum_i(-1)^{|a_i||b|}\mathfrak{j}(a_i,a_i',b),
&&a, b\in R_-\cap[R,R],
\end{align*}
and $\alpha(x,y)=0$ for other pairs $(x,y)$ with $x,y\in\mathfrak{B}$. In the last equality above, it is written that $a=\sum_i([a_i,a_i']-\overline{[a_i,a_i']})$ and
$$\mathfrak{j}(a,b,c):=(-1)^{|a||c|}\tildeHD{ab}{c}+(-1)^{|b||a|}\tildeHD{bc}{a}+(-1)^{|c||a|}\tildeHD{ca}{b}.$$

We have to show that $\alpha$ is well-defined, i.e., $\alpha(e_{11}(a),e_{11}(b))$ is independent of the expression $a=\sum_i([a_i,a_i']-\overline{[a_i,a_i']})$ for $a,b\in R_-\cap [R,R]$. This can be verified as follows:

We write $b=\sum_j([b_j,b_j']-\overline{[b_j,b_j']})$ since $b\in R_-\cap [R,R]$ and deduce that
\begin{align*}
\sum_i(-1)^{|a_i||b|}\mathfrak{j}(a_i,a_i',b)
&=\sum_{i,j}(-1)^{|a_i||b_j|+|a_i||b_j'|}\mathfrak{j}(a_i,a_i',[b_j,b_j']-\overline{[b_j,b_j']})\\
&=2\sum_{i,j}(-1)^{|a_i||b_j|+|a_i||b_j'|}\mathfrak{j}(a_i,a_i',[b_j,b_j'])\\
&=-2\sum_{i,j}(-1)^{|a_i'||b_j'|+|a_i||b_j'|}\mathfrak{j}(b_j,b_j',[a_i,a_i'])\\
&=-\sum_{i,j}(-1)^{|a_i'||b_j'|+|a_i||b_j'|}\mathfrak{j}(b_j,b_j',[a_i,a_i']-\overline{[a_i,a_i']})\\
&=-\sum_{i,j}(-1)^{|a||b_j'|}\mathfrak{j}(b_j,b_j',a).
\end{align*}
which is independent of the expression $a=\sum_i([a_i,a_i']-\overline{[a_i,a_i']})$. Hence, $\alpha$ is well-defined. Furthermore, using the definition of $\tildeHD{R}{R}$, we directly verify that $\alpha$ is a $2$-cocycle on the Lie superalgebra $\mathfrak{osp}_{1|2}(R,{}^-)$.

The 2-cocycle $\alpha$ gives rise to a new Lie superalgebra
$$\mathfrak{C}:=\mathfrak{osp}_{1|2}(R,{}^-)\oplus\tildeHD{R}{R},$$
with the multiplication
$$[x\oplus c,y\oplus c']=[x,y]\oplus\alpha(x,y),\quad \forall x,y\in\mathfrak{osp}_{1|2}(R,{}^-), c,c'\in\tildeHD{R}{R}.$$
which is a central extension of $\mathfrak{osp}_{1|2}(R,{}^-)$ with $\tildeHD{R}{R}$ in the kernel.

Now, the elements
$$\tilde{v}(a)=e_{23}(a+\bar{a})\oplus0, \quad\tilde{w}(a):=e_{32}(a+\bar{a})\oplus0,\quad\tilde{f}(a):={f}(a)\oplus0,\quad \tilde{g}(a):={g}(a)\oplus0\in\mathfrak{C}$$
satisfy all the defining relations of $\widehat{\mathfrak{osp}}_{1|2}(R,{}^-)$. Hence, there is a canonical homomorphism
$$\phi:\widehat{\mathfrak{osp}}_{1|2}(R,{}^-)\rightarrow\mathfrak{C}$$
such that
$$\phi(\mathbf{f}(a))=f(a)\oplus0,\text{ and }\phi(\mathbf{g}(a))=g(a)\oplus0.$$
Hence,
\begin{align*}
\phi(\mathbf{h}(a,b))
&=\phi([\mathbf{f}(a),\mathbf{g}(b)])
-(-1)^{|a||b|}\phi([\mathbf{f}(1),\mathbf{g}(ba)])\\
&=[f(a)\oplus0,g(b)\oplus0]-(-1)^{|a||b|}[f(1)\oplus0,g(ba)\oplus0]\\
&=[f(a),g(b)]\oplus \alpha(f(a),g(b))-(-1)^{|a||b|}[f(1),g(ba)]\oplus \alpha(f(1),g(ba))\\
&=e_{11}([a,b]-\overline{[a,b]})\oplus2\tildeHD{a}{b},
\end{align*}
which yields the injectivity of $\eta$ since $\eta(\tildeHD{a}{b})=\mathbf{h}(a,b)$.
\end{proof}

\begin{remark}
\label{rmk:osp12_ker_inj}
In order to prove the injectivity of $\eta$, we construct the central extension $\mathfrak{C}$ of $\mathfrak{osp}_{1|2}(R,{}^-)$ via explicitly creating a $2$-cocycle on $\mathfrak{osp}_{1|2}(R,{}^-)$. Similar techniques have been applied in the proof of Theorem~\ref{thm:osp_ce}. However, the method for creating the $2$-cocycles on $\mathfrak{osp}_{1|2}(R,{}^-)$ is quite different from the methods used in Theorem~\ref{thm:osp_ce}. Concretely, a $2$-cocycle on $\mathfrak{osp}_{m|2n}(R,{}^-)$ with $(m,n)\neq(1,1)$ has been created via restricting a $2$-cocycle on $\mathfrak{gl}_{m|2n}(R,{}^-)$. The case of $\mathfrak{osp}_{1|2}(R,{}^-)$ is much complicated since the $2$-cocycle on $\mathfrak{osp}_{1|2}(R,{}^-)$ that we created here does not come from the restriction of a $2$-cocycle on $\mathfrak{gl}_{1|2}(R)$.
\end{remark}

\subsection{A nontrivial central extension of $\widehat{\mathfrak{osp}}_{1|2}(R,{}^-)$}
\label{subsubsec:hat2osp12}

We have shown that $\psi:\widehat{\mathfrak{osp}}_{1|2}(R,{}^-)\rightarrow\mathfrak{osp}_{1|2}(R,{}^-)$ is a central extension, but it is not necessarily universal. We will explicitly create a nontrivial central extension of $\widehat{\mathfrak{osp}}_{1|2}(R,{}^-)$.

Let $\mathcal{I}_3$ be the $\Bbbk$-submodule of $R$ spanned by
$$3a,\quad a-\bar{a},\quad ([a,b]-\overline{[a,b]})c,\quad a_+b_+c-(-1)^{|a||b|}b_+a_+c,$$
and
$$(-1)^{|a||c|}(a\bar{b}c+\bar{a}b\bar{c})+(-1)^{|b||a|}(b\bar{c}a+\bar{b}c\bar{a})+(-1)^{|c||b|}(c\bar{a}b+\bar{c}a\bar{b})$$
for homogeneous $a,b,c\in R$. Let
$$\mathfrak{z}:=R/\mathcal{I}_3\oplus R/\mathcal{I}_3$$
and $\boldsymbol{\pi}_1(a), \boldsymbol{\pi}_2(b)$ denote the canonical image of $a\in R$ in the first and second copy of $R/\mathcal{I}_3$. Then
\begin{eqnarray}
&3\boldsymbol{\pi}_i(a)=0,&\label{eq:osp12_uce_cyl1}\\
&\boldsymbol{\pi}_i(a)=\boldsymbol{\pi}_i(\bar{a}),&\label{eq:osp12_uce_cyl2}\\
&\boldsymbol{\pi}_i([a,b]c)=\boldsymbol{\pi}_i(\overline{[a,b]}c),&\label{eq:osp12_uce_cyl3}\\
&\boldsymbol{\pi}_i(a_+b_+c)=(-1)^{|a||b|}\boldsymbol{\pi}_i(b_+a_+c),&\label{eq:osp12_uce_cyl4}\\
&(-1)^{|a||c|}\boldsymbol{\pi}_i(a\bar{b}c+\bar{a}b\bar{c})
+(-1)^{|b||a|}\boldsymbol{\pi}_i(b\bar{c}a+\bar{b}c\bar{a})
+(-1)^{|c||b|}\boldsymbol{\pi}_i(c\bar{a}b+\bar{c}a\bar{b})=0,&\label{eq:osp12_uce_cyl5}
\end{eqnarray}
where $a,b,c\in R$ are homogeneous.

Recall that $\widehat{\mathfrak{osp}}_{1|2}(R,{}^-)$ is spanned by
$$\mathfrak{B}:=\{\mathbf{h}(a,b),\mathbf{v}(a),\mathbf{w}(a),\mathbf{f}(a),\mathbf{g}(a)|a,b\in R\text{ are homogeneous.}\}$$
as a $\Bbbk$-module, we define a $\Bbbk$-bilinear map
$$\beta:\widehat{\mathfrak{osp}}_{1|2}(R,{}^-)\times\widehat{\mathfrak{osp}}_{1|2}(R,{}^-)\rightarrow\mathfrak{z}$$
as follows,
\begin{align*}
\beta(\mathbf{v}(a),\mathbf{g}(b))&=-(-1)^{|a|(1+|b|)}\beta(\mathbf{g}(b),\mathbf{v}(a))=\boldsymbol{\pi}_1(a_+b),\\
\beta(\mathbf{f}(a),\mathbf{w}(b))&=-(-1)^{(1+|a|)|b|}\beta(\mathbf{w}(b),\mathbf{f}(a))=\boldsymbol{\pi}_2(ab_+),
\end{align*}
and $\beta(x,y)=0$ for all other pairs $(x,y)\in\mathfrak{B}\times\mathfrak{B}$.

\begin{lemma}
\label{lem:hat2osp12_2cy}
The $\Bbbk$-bilinear map $\beta$ is a $2$-cocycle on $\widehat{\mathfrak{osp}}_{1|2}(R,{}^-)$ with values in $\mathfrak{z}$.
\end{lemma}
\begin{proof}
This is verified through direct computation. We omit the details here.
\end{proof}

Using the $2$-cocycle $\beta$, we define a new Lie superalgebra $\mathfrak{uosp}_{1|2}(R,{}^-)$:
$$\mathfrak{uosp}_{1|2}(R,{}^-):=\widehat{\mathfrak{osp}}_{1|2}(R,{}^-)\oplus \mathfrak{z}$$
with the Lie super-bracket $[x\oplus c,y\oplus c']=[x,y]\oplus \beta(x,y)$ for $x,y\in\widehat{\mathfrak{osp}}_{1|2}(R,{}^-)$ and $c,c'\in\mathfrak{z}$. The Lie superalgebra $\mathfrak{uosp}_{1|2}(R,{}^-)$ can also be characterized by generators and relations. A set of generators of $\mathfrak{uosp}_{1|2}(R,{}^-)$ consists of $\mathbf{v}(a)$, $\mathbf{w}(a)$, $\mathbf{f}(a)$, $\mathbf{g}(a)$, $\boldsymbol{\pi}_1(a)$ and $\boldsymbol{\pi}_2(a)$. The defining relations of $\mathfrak{uosp}_{1|2}(R,{}^-)$ are given by
\begin{align*}
&\begin{aligned}
&\mathbf{v}(\bar{a})=\mathbf{v}(a),
&&\mathbf{w}(\bar{a})=\mathbf{w}(a),\\
&[\mathbf{v}(a),\mathbf{v}(b)]=0,
&&[\mathbf{w}(a),\mathbf{w}(b)]=0,\\
&[\mathbf{v}(a),\mathbf{g}(b)]=\boldsymbol{\pi}_1(a_+b),
&&[\mathbf{f}(a),\mathbf{w}(b)]=\boldsymbol{\pi}_2(ab_+),\\
&[\mathbf{v}(a),\mathbf{f}(b)]=(-1)^{|b|}\mathbf{g}(a_+\bar{b}),
&&[\mathbf{g}(a),\mathbf{w}(b)]=-(-1)^{|a|}\mathbf{f}(\bar{a}b_+),\\
&[\mathbf{f}(a),\mathbf{f}(b)]=(-1)^{|a|}\mathbf{w}(\bar{a}b),
&&[\mathbf{g}(a),\mathbf{g}(b)]=-(-1)^{|b|}\mathbf{v}(a\bar{b}),
\end{aligned}\\
&[[\mathbf{f}(a),\mathbf{g}(b)],\mathbf{f}(c)]=\mathbf{f}(abc-\overline{ab}c-(-1)^{|a||b|+|b||c|+|c||a|}cba),\\
&[\mathbf{g}(a),[\mathbf{f}(b),\mathbf{g}(c)]]=\mathbf{g}(abc-a\overline{bc}-(-1)^{|a||b|+|b||c|+|c||a|}cba),
\end{align*}
where $a,b,c\in R$ are homogeneous.

\subsection{The universality of $\mathfrak{uosp}_{1|2}(R,{}^-)$}
\label{subsubsec:osp12_uce}

Let $\psi':\mathfrak{uosp}_{1|2}(R,{}^-)\rightarrow\widehat{\mathfrak{osp}}_{1|2}(R,{}^-)$ be the canonical homomorphism of Lie superalgebras. We have known from Proposition~\ref{prop:osp12_ce} that $\psi:\widehat{\mathfrak{osp}}_{1|2}(R,{}^-)\rightarrow\mathfrak{osp}_{1|2}(R,{}^-)$ is a central extension, so is $\psi\circ\psi':\mathfrak{uosp}_{1|2}(R,{}^-)\rightarrow\mathfrak{osp}_{1|2}(R,{}^-)$. We will show that

\begin{theorem}
\label{thm:osp12_uce}
The central extension
$\psi\circ\psi':\mathfrak{uosp}_{1|2}(R,{}^-)\rightarrow\mathfrak{osp}_{1|2}(R,{}^-)$ is universal.
\end{theorem}
\begin{proof}
Let $\varphi:\mathfrak{E}\rightarrow\mathfrak{osp}_{1|2}(R,{}^-)$ be an arbitrary central extension of $\mathfrak{osp}_{1|2}(R,{}^-)$. We define
\begin{align*}
\tilde{v}(a):&=-\frac{1}{2}[\hat{g}(a_+),\hat{g}(1)],
&\tilde{w}(a):&=\frac{1}{2}[\hat{f}(1),\hat{f}(a_+)],\\
\tilde{f}(a):&=-\frac{1}{2}(-1)^{|a|}[\hat{g}(\bar{a}),\hat{w}(1)],
&\tilde{g}(a):&=\frac{1}{2}(-1)^{|a|}[\hat{v}(1),\hat{f}(\bar{a})],\\
\tilde{\pi}_1(a):&=\frac{1}{2}[\hat{v}(1),\hat{g}(a)],
&\tilde{\pi}_2(a):&=\frac{1}{2}[\hat{f}(a),\hat{w}(1)],
\end{align*}
for homogeneous $a\in R$, where $\hat{x}$ is an arbitrary element of $\varphi^{-1}(x)$ for $x\in\mathfrak{osp}_{1|2}(R,{}^-)$. Since $\varphi:\mathfrak{E}\rightarrow\mathfrak{osp}_{1|2}(R,{}^-)$ is a central extension, the element $[\hat{x},\hat{y}]$ is independent of the choice of $\hat{x}\in\varphi^{-1}(x)$ and $\hat{y}\in\varphi^{-1}(y)$ for $x,y\in\mathfrak{osp}_{1|2}(R,{}^-)$.

We first show that $\tilde{\pi}_i, i=1,2$ satisfy (\ref{eq:osp12_uce_cyl1})-(\ref{eq:osp12_uce_cyl5}). We set $\tilde{h}(a,b):=[\hat{f}(a),\hat{g}(b)]$ and deduce that
$$[\tilde{h}(1,1),\tilde{\pi}_1(a)]=3\tilde{\pi}_1(a),\text{ and }[\tilde{h}(1,1),\tilde{\pi}_2(a)]=-3\tilde{\pi}_2(a),$$
which yields that $3\tilde{\pi}_1(a)=0=3\tilde{\pi}_2(a)$ since $\tilde{\pi}_i(a)\in\ker\varphi$ is contained in the center of $\mathfrak{E}$.

For (\ref{eq:osp12_uce_cyl2}), we compute that
\begin{align*}
\tilde{\pi}_1(\bar{a})
&=\frac{1}{2}[\hat{v}(1),\hat{g}(\bar{a})]
=-\frac{1}{2}[\hat{v}(1),[\hat{g}(1),\tilde{h}(a,1)]]\\
&=-0-\frac{1}{2}[\hat{g}(1),[\hat{v}(1),\tilde{h}(a,1)]]
=(-1)^{|a|}[\hat{g}(1),\hat{v}(a)].
\end{align*}
Hence, $\tilde{\pi}_1(\bar{a})=\tilde{\pi}_1(a)$ since $\hat{v}(a)-\hat{v}(\bar{a})\in\ker\varphi$. Similarly, we have $\tilde{\pi}_2(\bar{a})=\tilde{\pi}_2(a)$.

We show (\ref{eq:osp12_uce_cyl4}) and (\ref{eq:osp12_uce_cyl5}) before proving (\ref{eq:osp12_uce_cyl3}). In order to show (\ref{eq:osp12_uce_cyl4}), we first deduce that
$$[\tilde{v}(a),\tilde{g}(b)]=\tilde{\pi}_1(a_+b),\text{ and }[\tilde{f}(a),\tilde{w}(b)]=\tilde{\pi}_2(ab_+).$$
Then we compute that
\begin{align*}
\tilde{\pi}_1(a_+b_+c)&=[\tilde{v}(a),\tilde{g}(b_+c)]
=(-1)^{|c|}[\tilde{v}(a),[\tilde{v}(b),\tilde{f}(\bar{c})]]\\
&=(-1)^{|c|+|a||b|}[\tilde{v}(b),[\tilde{v}(a),\tilde{f}(\bar{c})]
=(-1)^{|a||b|}\tilde{\pi}_1(b_+a_+c).
\end{align*}
A similar computation also shows $\tilde{\pi}_2(ab_+c_+)=(-1)^{|b||c|}\tilde{\pi}_2(ac_+b_+)$, which yields $\tilde{\pi}_2(a_+b_+c)=(-1)^{|a||b|}\tilde{\pi}_2(b_+a_+c)$ since $\tilde{\pi}_2(\bar{a})=\tilde{\pi}_2(a)$.

For (\ref{eq:osp12_uce_cyl5}), we consider
$$[[\tilde{g}(a),\tilde{g}(b)],\tilde{g}(c)]
=-(-1)^{|b|}[\tilde{v}(a\bar{b}),\tilde{g}(c)]
=-(-1)^{|b|}\tilde{\pi}_1(a\bar{b}c)-(-1)^{|b|+|a||b|}\tilde{\pi}_1(b\bar{a}c).$$
Using the Jacobi identity
\begin{align*}
0=(-1)^{p(a,c)}[[\tilde{g}(a),\tilde{g}(b)],\tilde{g}(c)]
+(-1)^{p(b,a)}[[\tilde{g}(b),\tilde{g}(c)],\tilde{g}(a)]
+(-1)^{p(c,b)}[[\tilde{g}(c),\tilde{g}(a)],\tilde{g}(b)],
\end{align*}
where $p(a,b)=(1+|a|)(1+|b|)$, we conclude that
\begin{align*}
(-1)^{|a||c|}\tilde{\pi}_1(a\bar{b}c+\bar{a}b\bar{c})
+(-1)^{|b||a|}\tilde{\pi}_1(b\bar{c}a+\bar{b}c\bar{a})
+(-1)^{|c||b|}\tilde{\pi}_1(c\bar{a}b+\bar{c}a\bar{b})=0,
\end{align*}
and similarly for $\tilde{\pi}_2$.

Now, we return to prove that $\tilde{\pi}_i, i=1,2$ satisfy (\ref{eq:osp12_uce_cyl3}). Taking $c=1$ in (\ref{eq:osp12_uce_cyl5}) and then applying (\ref{eq:osp12_uce_cyl1}), we obtain
\begin{align*}
\tilde{\pi}_1((a-\bar{a})(b-\bar{b}))=0.
\end{align*}
It follows that
\begin{align*}
[\tilde{v}(ab),\tilde{g}(c)]
&=-\frac{1}{2}(-1)^{(1+|a|)(1+|b|)}[[\tilde{h}(b,a),\tilde{v}(1)],\tilde{g}(c)]\\
&=\frac{1}{2}(-1)^{(1+|a|)(1+|b|)}[\tilde{v}(1),[\tilde{h}(b,a),\tilde{g}(c)]]\\
&=\frac{1}{2}(-1)^{|a||b|+|b||c|+|c||a|}[\tilde{v}(1),\tilde{g}(cba-c\overline{ba}-(-1)^{|a||b|+|b||c|+|c||a|}abc)]\\
&=\frac{1}{2}\tilde{\pi}_1(abc)-\frac{1}{2}(-1)^{|a||b|+|b||c|+|c||a|}\tilde{\pi}_1(c(ba-\overline{ba}))\\
&=\frac{1}{2}\tilde{\pi}_1(abc)+\frac{1}{2}(-1)^{|a||b|}\tilde{\pi}_1((ba-\overline{ba})c).
\end{align*}
On the other hand, $[\tilde{v}(ab),\tilde{g}(c)]=\tilde{\pi}_1(abc)+\tilde{\pi}_1(\overline{ab}c)$. Hence, $\tilde{\pi}_1([a,b]c)=\tilde{\pi}_1(\overline{[a,b]}c)$. A similar argument also shows that (\ref{eq:osp12_uce_cyl3}) holds for $\tilde{\pi}_2$.

In summary, we have shown that $\tilde{\pi}_1$ and $\tilde{\pi}_2$ satisfy (\ref{eq:osp12_uce_cyl1})-(\ref{eq:osp12_uce_cyl5}). Furthermore, we directly verify that $\tilde{v}(a),\tilde{w}(a),\tilde{f}(a),\tilde{g}(a),\tilde{\pi}_1(a),\tilde{\pi}_2(a)$ for $a\in R$ satisfy all relations in the definition of $\mathfrak{uosp}_{1|2}(R,{}^-)$. Hence, there is a unique homomorphism of Lie superalgebras $\varphi':\mathfrak{uosp}_{1|2}(R,{}^-)\rightarrow\mathfrak{E}$ such that $\varphi\circ\varphi'=\psi\circ\psi'$. Therefore, the central extension $\psi\circ\psi':\mathfrak{uosp}_{1|2}(R,{}^-)\rightarrow\mathfrak{osp}_{1|2}(R,{}^-)$ is universal.
\end{proof}

Following Propositions~\ref{prop:osp12_ce},~\ref{prop:hatosp12_ker} and Theorem~\ref{thm:osp12_uce}, we conclude that
\begin{corollary}
\label{cor:osp12_hml}
$\mathrm{H}_2(\mathfrak{osp}_{1|2}(R,{}^-),\Bbbk)\cong\fourIdx{}{-}{}{1}{\widetilde{\mathrm{HD}}}(R,{}^-)\oplus R/\mathcal{I}_3\oplus R/\mathcal{I}_3$, where $\mathcal{I}_3$ is the $\Bbbk$-submodule of $R$ as defined in Subsection~\ref{subsubsec:hat2osp12}.
\end{corollary}

The presentation of the Lie superalgebra $\mathfrak{uosp}_{1|2}(R,{}^-)$ is quite complicated here. The following proposition shows that it is indeed same as the Lie superalgebra $\widehat{\mathfrak{osp}}_{1|2}(R,{}^-)$ in many cases.

\begin{proposition}
\label{prop:osp12_hat2deg}
If $3$ is invertible in $R$ or $R_-+R_-\cdot R_-=R$, then
$$\mathfrak{uosp}_{1|2}(R,{}^-)\cong \widehat{\mathfrak{osp}}_{1|2}(R,{}^-)$$
as Lie superalgebras over $\Bbbk$, and hence $\psi:\widehat{\mathfrak{osp}}_{1|2}(R,{}^-)\rightarrow\mathfrak{osp}_{1|2}(R,{}^-)$ is a universal central extension in this situation. In particular, for every unital associative superalgebra $S$,
$$\mathfrak{uosp}_{1|2}(S\oplus S^{\mathrm{op}},\mathrm{ex})\cong \widehat{\mathfrak{osp}}_{1|2}(S\oplus S^{\mathrm{op}},\mathrm{ex}).$$
\end{proposition}
\begin{proof}
If $3$ is invertible in $R$, then $3R=R$, which yields that $\mathfrak{z}=0$. On the other hand, we observe that (\ref{eq:osp12_uce_cyl2}) implies that $\boldsymbol{\pi}_i(R_-)=0$ for $i=1,2$, while (\ref{eq:osp12_uce_cyl5}) yields that $\boldsymbol{\pi}_i(R_-\cdot R_-)=0$ for $i=1,2$. Hence, $\mathfrak{z}=0$ if $R_-+R_-\cdot R_-=R$. Hence, the first assertion holds.

In $S\oplus S$, we know that
$$a\oplus b=\frac{1}{2}(a-b)\oplus (b-a)+\frac{1}{2}((a+b)\oplus(-a-b))\cdot(1\oplus-1)\in R_-+R_-\cdot R_-,$$
for all $a,b\in S$. Hence, $(S\oplus S^{\mathrm{op}})_-+(S\oplus S^{\mathrm{op}})_-\cdot(S\oplus S^{\mathrm{op}})_-=S\oplus S^{\mathrm{op}}$, which yields the second assertion.
\end{proof}

\begin{remark}
If $R$ is a unital super-commutative associative superalgebra with the identity superinvolution, the Lie superalgebra $\mathfrak{uosp}_{1|2}(R,\mathrm{id})$ is not necessarily equal to $\widehat{\mathfrak{osp}}_{1|2}(R,\mathrm{id})$. However, we observe that $\mathcal{I}_3=3R$ in this situation. By Remark~\ref{rmk:tildeHD_com}, we conclude that
$$\mathrm{H}_2(\mathfrak{osp}_{1|2}(\Bbbk)\otimes_{\Bbbk}R,\Bbbk)\cong\mathrm{HC}_1(R)\oplus (R/3R)\oplus (R/3R),$$
which coincides with the second homology group of $\mathfrak{osp}_{1|2}(\Bbbk)\otimes_{\Bbbk}R$ obtained in \cite{IoharaKoga2005} if we additionally assume that $\Bbbk$ is a filed of characteristic zero.
\end{remark}

If $(R,{}^-)=(S\oplus S^{\mathrm{op}},\mathrm{ex})$, we recover the result about the universal central extension of $\mathfrak{sl}_{1|2}(S)$ obtained by \cite{ChenSun2015}.

\begin{corollary}
\label{cor:osp12_SS}
Let $S$ be a unital associative superalgebra. Then the canonical homomorphism $\mathfrak{st}_{1|2}(S)\rightarrow\mathfrak{sl}_{1|2}(S)$ is the universal central extension and $\mathrm{H}_2(\mathfrak{sl}_{1|2}(S),\Bbbk)=\mathrm{HC}_1(S)$.
\end{corollary}
\begin{proof}
We have known from Lemma~\ref{lem:hmltilde} that $\fourIdx{}{-}{}{1}{\widetilde{\mathrm{HD}}}(S\oplus S^{\mathrm{op}},\mathrm{ex})\cong\mathrm{HC}_1(S)$. Hence, the statements follow from the following commutative diagram
$$\xymatrix{
0\ar[r]&
\fourIdx{}{-}{}{1}{\widetilde{\mathrm{HD}}}(S\oplus S^{\mathrm{op}},\mathrm{ex})\ar[r]\ar[d]&
\widehat{\mathfrak{osp}}_{1|2}(S\oplus S^{\mathrm{op}},\mathrm{ex})\ar[r]\ar[d]&
\mathfrak{osp}_{1|2}(S\oplus S^{\mathrm{op}},\mathrm{ex})\ar[r]\ar[d]&0\\
0\ar[r]&
\mathrm{HC}_1(S)\ar[r]&
\mathfrak{st}_{1|2}(S)\ar[r]&
\mathfrak{sl}_{1|2}(S)\ar[r]&0
}$$
where $\widehat{\mathfrak{osp}}_{1|2}(S\oplus S^{\mathrm{op}},\mathrm{ex})\rightarrow\mathfrak{osp}_{1|2}(S\oplus S^{\mathrm{op}},\mathrm{ex})$ is a universal central extension by Theorem~\ref{thm:osp12_uce} and Proposition~\ref{prop:osp12_hat2deg}.
\end{proof}

\section*{Acknowledgements}
The authors thank Prof. Yun Gao, Prof. Hongjia Chen and Dr. Shikui Shang for useful suggestions. The authors are grateful to the referees for their careful reviews and helpful comments.


\begin{thebibliography}{11}
\bibitem{BenkartElduque2002} G. Benkart and A. Elduque, Lie superalgebras graded by the root systems $C(n)$, $D(m,n)$, $D(2,1;\alpha)$, $F(4)$, $G(3)$, {\it Canad. Math. Bull.} {\bf 45} (2002) 509-524.
\bibitem{BenkartElduque2003} G. Benkart and A. Elduque, Lie superalgebras graded by the root system $B(m,n)$. {\it Selecta Math. (N.S.)} {\bf 9} (2003) 313-360.
\bibitem{ChangWang2015} Z. Chang, and Y. Wang, $\mathbb{Z}/{2\mathbb{Z}}$-graded dihedral homologies and generalized periplectic Lie superalgebras, {\it preprint at arXiv:}1508:07449, (2015).
\bibitem{ChenGuay2013} H. Chen and N. Guay, Central extensions of matrix Lie superalgebras over $\mathbb{Z}/2\mathbb{Z}$--graded algebras, {\it Algebr. Represent. Theory} {\bf16} (2013) 591-604.
\bibitem{ChenSun2015} H. Chen and J. Sun, Universal central extensions of $\mathfrak{sl}_{m|n}$ over $\mathbb{Z}/2\mathbb{Z}$--graded algebras, {\it J. Pure Appl. Algebra} {\bf 219} (2015) 4278-4294.
\bibitem{FuchLeites1984} D. B. Fuch and D. A. Le\u{\i}tes, Cohomology of Lie superalgebras. \textit{C. R. Acad. Bulgare Sci.} {\bf37} (1984) 1595-1596.
\bibitem{Duff2002} A. Duff, {\it Derivations, invariant forms and the second homology group of orthosymplectic Lie superalgebras}. Ph.D. Dissertation, University of Ottawa, (2002).
\bibitem{Gao1996I} Y. Gao, Involutive Lie algebras graded by finite root systems and compact forms of IM algebras, {\it Math. Z.} {\bf 223} (1996) 651-672.
\bibitem{Gao1996II} Y. Gao, Steinberg unitary Lie algebras and skew-dihedral homology, {\it J. Algebra} {\bf179} (1996) 261-304.
\bibitem{IoharaKoga2001} K. Iohara and Y. Koga, Central extensions of Lie superalgebras, {\it Comment. Math. Helv.} {\bf76} (2001) 110-154.
\bibitem{IoharaKoga2005} K. Iohara and Y. Koga, Second homology of Lie superalgebras, {\it Math. Nachr.} {\bf278} (2005) 1041-1053.
\bibitem{Kac1977} V. G. Kac, Lie superalgebras, {\it Adv. Math.} {\bf 26} (1977) 8-96.
\bibitem{KasselLoday1982} C. Kassel and J. L. Loday, Extensions centrales d'alg\`{e}bres de Lie, {\it Ann. Inst. Fourier (Grenoble)} {\bf32} (1982) 119-142.
\bibitem{MikhalevPinchuk2002} A. V. Mikhalev and I. A. Pinchuk, Universal central extensions of the matrix Lie superalgebras $sl(m,n,A)$, {\it Contemp. Math.} {\bf264} (2000) 111-125.
\bibitem{Neher2003} E. Neher, An introduction to universal central extensions of Lie superalgebras, {\it Groups, rings, Lie and Hopf algebras} (St. John's, NF, 2001), 141-146, Math. Appl. 555, Kluwer Acad. Publ., Dordrecht, 2003.
\bibitem{Racine1998} M. L. Racine, Primitive superalgebras with superinvolution, {\it J. Algebra} {\bf 206} (1998) 588-614.
\bibitem{ScheunertZhang1998} M.Scheunert and R. B. Zhang, Cohomology of Lie superalgebras and their generalizations, {\it J. Math. Phys.} {\bf 39} (1998), 5024-5061.
\end{thebibliography}
\end{document}